\documentclass{article}[12pt]
\usepackage{color}
\usepackage[title]{appendix}
\usepackage{mathrsfs}
\usepackage{mathtools}
\usepackage{graphicx}
\usepackage{epstopdf}
\usepackage{stmaryrd}
 \usepackage{subfigure}
\usepackage{float}
\usepackage{amsmath}
    \usepackage{bm}
\usepackage{amsfonts,amssymb}
  \usepackage{dsfont}
\usepackage{amsfonts}
\usepackage{mathrsfs}
  \usepackage{pifont}
\numberwithin{equation}{section}
 \usepackage{amsthm}
 \newtheorem{theorem}{Theorem}[section]
\newtheorem{lemma}[theorem]{Lemma}

\newtheorem{remark}[theorem]{Remark}

 \usepackage{amsmath}
  \usepackage[hidelinks]{hyperref}
 \usepackage{multirow}

\begin{document}
\title{A Mini Immersed Finite Element Method for Two-Phase Stokes Problems on Cartesian Meshes}
\author{
Haifeng Ji\footnotemark[1]$~^,$\footnotemark[2] \qquad
Dong Liang\footnotemark[3] 
\qquad
Qian Zhang\footnotemark[4]
}
\footnotetext[1]{School of Science, Nanjing University of Posts and Telecommunications, Nanjing, Jiangsu 210023, China  (hfji1988@foxmail.com)}
\footnotetext[2]{Key Laboratory of NSLSCS, Ministry of Education, Nanjing Normal University, Nanjing, Jiangsu 210023, China}
\footnotetext[3]{Department of Mathematics and Statistics, York University, Toronto, ON M3J 1P3, Canada (dliang@yorku.ca)}
\footnotetext[4]{School of Artificial Intelligence and Information Technology, Nanjing University of Chinese Medicine, Nanjing, Jiangsu 210023, China (380793@njucm.edu.cn)}

\date{}
\maketitle


\begin{abstract}
This paper presents a mini immersed finite element (IFE) method for solving two- and three-dimensional two-phase Stokes problems on Cartesian meshes. The IFE space is constructed from the conventional mini element, with shape functions modified  on interface elements according to interface jump conditions, while keeping the degrees of freedom unchanged. Both discontinuous viscosity coefficients and surface forces are taken into account in the construction. The interface is approximated using discrete level set functions, and explicit formulas for IFE basis functions and correction functions are derived, facilitating ease of implementation.The inf-sup stability and the optimal a priori error  estimate of the IFE method, along with the optimal approximation capabilities of the IFE space, are derived rigorously, with constants that are independent of the mesh size and the manner in which the interface intersects the mesh, but may depend on the discontinuous viscosity coefficients.  Additionally, it is proved that the condition number has the usual bound independent of the interface. Numerical experiments are provided to confirm the theoretical results.
\end{abstract}



\section{Introduction} 
In this paper we are interested in designing and analyzing an immersed finite element (IFE) method for solving two-phase Stokes problems (also known as the Stokes interface problem in the literature).
Let $\Omega\subset\mathbb{R}^N$, $N=2, 3$,  be the computational domain and $\Gamma$ be a $C^2$-smooth closed hypersurface  immersed in $\Omega$. The interface $\Gamma$  divides $\Omega$ into two phases $\Omega^+$ and $\Omega^-$. Without loss of generality, we assume that $\Omega^-$ is the inclusion, i.e., $\Gamma=\partial\Omega^-$. 
The Stokes interface problem reads as follows: Given a body force $\mathbf{f}\in L^2(\Omega)^N$, a surface force $\mathbf{g}\in L^2(\Gamma)^N$ and a piecewise constant viscosity $\mu|_{\Omega^\pm}=\mu^\pm>0$,  find a velocity $\mathbf{u}$ and a pressure $p$ such that 
\begin{subequations}\label{originalpb0}
\begin{align}
-\nabla\cdot (2\mu\boldsymbol{\epsilon}(\mathbf{u})) +\nabla p&=\mathbf{f}  \qquad\mbox{in } \Omega^+\cup\Omega^-,\label{originalpb1}\\
\nabla\cdot \mathbf{u}&=0\qquad\mbox{in } \Omega,\label{originalpb2}\\
[\sigma(\mu,\mathbf{u},p)\mathbf{n}]_\Gamma&=\mathbf{g}\qquad\mbox{on } \Gamma,\label{jp_cond1}\\
[\mathbf{u}]_\Gamma&=\mathbf{0}\qquad\mbox{on } \Gamma,\label{jp_cond2}\\
\mathbf{u}&=\mathbf{0} \qquad\mbox{on } \partial\Omega,\label{originalpb5}
\end{align}
\end{subequations}
where  $\boldsymbol{\epsilon}(\mathbf{u})=\frac{1}{2}(\nabla \mathbf{u}+(\nabla \mathbf{u})^T)$ represents the strain tensor, $\sigma(\mu,\mathbf{u},p)=2\mu \boldsymbol{\epsilon}(\mathbf{u})-p\mathbb{I}$ denotes  the total stress tensor,  $\mathbb{I}$ is the identity matrix, $\mathbf{n}$ is the unit normal to  $\Gamma$ pointing from $\Omega^-$ to $\Omega^+$, and $[\mathbf{v}]_\Gamma$ stands for  the jump of $\mathbf{v}$ across $\Gamma$, i.e., $[\mathbf{v}]_\Gamma:=\mathbf{v}^+|_\Gamma-\mathbf{v}^-|_\Gamma$ with $\mathbf{v}^\pm:=\mathbf{v}|_{\Omega^\pm}$.  If the trace $(\nabla \cdot \mathbf{u})|_{\Gamma}$ is well defined, (\ref{originalpb2}) provides an additional  relationship
\begin{equation}\label{jp_cond3}
[\nabla \cdot \mathbf{u}]_\Gamma=0\quad\mbox{ on }\Gamma.
\end{equation}
With the usual  spaces $\boldsymbol{V}=H_0^1(\Omega)^N$  and $Q=L_0^2(\Omega):=\{q\in L^2(\Omega) : \int_\Omega q=0\}$, the weak formulation of (\ref{originalpb0}) reads : Find $(\mathbf{u}, p)\in(\boldsymbol{V}, Q)$  such that
\begin{equation}\label{weakform}
a(\mathbf{u}, \mathbf{v})+b(\mathbf{v}, p)-b(\mathbf{u}, q) = l(\mathbf{v})\quad \forall (\mathbf{v},q) \in (\boldsymbol{V}, Q),
\end{equation}
where
$a(\mathbf{u}, \mathbf{v}):=\int_\Omega2\mu\boldsymbol{\epsilon}(\mathbf{u}):\boldsymbol{\epsilon}(\mathbf{v})$, $b(\mathbf{v}, q):=-\int_\Omega q\nabla\cdot \mathbf{v}$ and $l(\mathbf{v}):=\int_\Omega \mathbf{f}\cdot\mathbf{v}-\int_{\Gamma}\mathbf{g}\cdot\mathbf{v}.$ The well-posedness of this weak formulation can be found in \cite{gross2011numerical}.

An important motivation for investigating the Stokes interface problem comes from two-phase incompressible flows (see \cite{kirchhart2016analysis, olshanskii2006analysis, gross2011numerical} and the references therein). For such problems involving interfaces, numerical methods based on fixed background meshes independent of interfaces (called unfitted meshes) 
have attracted a lot of attention because of the relative ease of  handling complex or  moving interfaces, especially in three dimensions. 
Since the interface is not aligned with the  mesh, it can arbitrarily cut  some  elements (called interface elements). 
To achieve a stable and accurate finite element (FE) method, specific efforts are needed for these interface elements. 
Roughly speaking, there are two ways to develop unfitted mesh FE methods with optimal convergences. 
One approach enriches the conventional FE space with extra degrees of freedom on interface elements to capture the discontinuities (see, e.g., \cite{hansbo2014cut, cattaneo2015stabilized, kirchhart2016analysis, guzman2018inf, caceres2020new}). 
The other approach modifies the traditional FE space according to interface conditions to capture the behavior of the exact solution, while keeping the degrees of freedom unchanged. The immersed finite element (IFE) method, originally introduced by Li \cite{li1998immersed} for one-dimensional interface problems, follows the latter approach and has become an efficient method for solving interface problems.
A significant advantage of the IFE method, compared to other unfitted mesh FE methods, is that the IFE space is isomorphic to the traditional FE space on the same unfitted mesh, regardless of the interface position. Thus, the structure of the resulting linear systems remains unchanged when solving moving interface problems.
For second-order elliptic interface problems, IFE methods have been extensively studied (see, e.g., \cite{li2004immersed,he2012convergence,taolin2015siam,GuzmanJSC2017,Guojcp2020,2021ji_IFE,ji3Dnonconforming}).
However, for the Stokes interface problems, the development and analysis of IFE methods are more challenging due to the coupling between velocity and pressure in the interface conditions and the divergence-free equation.
The first IFE method for the Stokes interface problems was developed by Adjerid et al. \cite{adjerid2015immersed}, considering the coupling of velocity and pressure  in the construction of IFE spaces.  Since then, various IFEs have been developed, such as the immersed $CR$--$P_0$ element  \cite{jones2021class}, the immersed rotated $Q_1$--$Q_0$  element  \cite{jones2021class}, and the Taylor--Hood IFE \cite{chen2021p2}. Very recently, Ji et al. \cite{2021ji_IFE_stoke} provide a theoretical analysis of an IFE method based on the immersed $CR$--$P_0$ element. However, to the best of our knowledge, the existing IFEs for the Stokes interface problems are restricted to 2D, and there is no theoretical analysis for IFEs with surface forces (i.e., $\mathbf{g}\not=\mathbf{0}$). 
One major obstacle is that the velocity and the pressure are also coupled in IFE spaces, and this becomes more complicated in 3D.

For the Stokes problems, the so-called mini element, developed by Arnold et al. \cite{arnold1984stable}, is very popular because it is stable, economical, and easy to implement.
In this paper, we propose and analyze an IFE variant of the mini element in 2D and 3D for solving the Stokes interface problems on Cartesian meshes. 
Compared with the conventional mini element method, the new IFE method requires only some modifications near the interface, thus the additional computational cost is low.
On each interface element, we first introduce some discrete interface conditions on approximate interfaces according to the exact interface conditions, and then modify the shape functions of the mini element to ensure these discrete interface conditions are satisfied. 
It should be emphasized that in this paper, we use discrete level set functions to discretize the interface, unlike traditional IFE methods that use points of intersection of the interface and the edges of elements. Our approach addresses the coplanarity issue encountered in constructing 3D IFE spaces, namely, that the interface cuts through four edges of a tetrahedron and the points of intersection are not coplanar. Our approach is compatible with the well-known level set method \cite{osher2001level,tornberg2000finite}, making the IFE method developed in this paper particularly well-suited for solving complex moving interface problems.

Another contribution of this work is the IFE discretization of the surface force (i.e., $\mathbf{g}\not=\mathbf{0}$), which 
is significant and cannot be neglected in many practical applications, such as simulating a (rising) liquid drop contained in a surrounding fluid \cite{gross2011numerical}. Within the IFE framework, a correction function,  nonzero only on interface elements, is pre-calculated based on  non-homogeneous jump conditions, and is then transferred to the right-hand side of the formulation of IFE methods. 
The attractive feature is that the stiffness matrix is  the same as that of IFE methods for problems with homogeneous jump conditions, and only the right-hand side needs to be modified. Generally, there are two approaches in the literature for constructing the correction function. One approach is based on smoothly extending the non-homogeneous jump functions to the neighborhood of the interface (see, e.g., \cite{ygong-li,ji2018high,zhang2023unfitted}). The other approach  involves using the non-homogeneous jump functions directly and solving a linear system with the same coefficient matrix as that used for IFE basis functions to obtain the correction function (see, e.g., \cite{adjerid2023enriched,adjerid2015immersed,guzman2016higher,He2012Immersed}).
In this paper we adopt the second approach. On each interface element $T$, it is natural to use the quantity $\int_{\Gamma\cap T}\mathbf{g}$ to construct the correction function.
However, we find that $\int_{\Gamma\cap T}\mathbf{g}$ may become unbounded as $|\Gamma\cap T|\rightarrow 0$, even if $\mathbf{g}\in H^{1/2}(\Gamma)$ (see \cite[Example 3.1]{zhang2023unfitted}). This issue can arise because the interface $\Gamma$ cuts the element $T$ arbitrarily.
To remedy this, we construct a larger fictitious box $R_T$ encompassing $T$ and use the quantity  $\int_{\Gamma\cap R_T}\mathbf{g}$ to develop correction functions. The fictitious element idea has been used in \cite{guo2019higher,zhuang2019high}, where a larger triangle $T^\lambda$ similar to $T$ is used to develop high degree IFE methods  for solving 2D interface problems with homogeneous jump conditions.  Our choice of a box $R_T$ simplifies the computation of integrals over $\Gamma\cap R_T$, especially in three dimensions.

As mentioned earlier, analyzing the proposed IFE space and IFE method is challenging due to the coupling between velocities and pressures, along with the presence of surface forces.    
Fortunately, we develop a new method to derive explicit expressions for both the IFE basis functions and the correction functions. These expressions are not only practical for implementation but also useful for our analysis.
Thanks to these explicit forms, we  have successfully proven the optimal approximation capabilities of the IFE space and the correction functions. These explicit expressions also enable us to derive some useful inequalities for the coupled IFE functions, such as the trace inequality and the inverse inequality.
With this preparation, we present a complete analysis of the proposed IFE method, including the inf-sup stability, the optimal a priori error estimate, and the condition number estimate, taking into account the dependency of the interface location relative to the mesh.

The paper is organized as follows. In section~\ref{sec_FEM} we formulate our IFE method. 
In section~\ref{sec_approximation} we prove the optimal approximation capabilities of IFE spaces.
In section~\ref{sec_analysis_IFEM} we present the theoretical analysis of our IFE method. 
Finally, some numerical results are provided in section~\ref{sec_num}.

\section{Finite element discretization}\label{sec_FEM}
\subsection{Unfitted meshes}
Since our focus is on the interface, we simply assume that the computational domain $\Omega$ is rectangular/cubic so that there is a family of  Cartesian triangular/tetrahedral meshes, denoted by $\{\mathcal{T}_h\}$, on $\Omega$. The meshes are not fitted to the interface. For example, in three dimensions ($N=3$), Cartesian tetrahedral meshes are obtained by first dividing  $\Omega$ into  cuboids and then subdividing each cuboid into six tetrahedra in the same manner (see Remark 3 in \cite{kvrivzek1992maximum}). 
For an element $T\in\mathcal{T}_h$ (a triangle for $N=2$ and  a tetrahedron for $N=3$),  $h_T$ denotes its diameter, and for a mesh $\mathcal{T}_h$, the index $h$ refers to the mesh size, i.e., $h=\max_{T\in\mathcal{T}_h}h_T$.  
We assume that $\{\mathcal{T}_h\}$ is shape regular, i.e., for all $\mathcal{T}_h$ and for all $T\in \mathcal{T}_h$, there exists a constant $C$ such that $h_T\leq C\rho_T$, where $\rho_T$ is the diameter of the largest ball inscribed in $T$.

In this paper, face means edge/face in two/three dimensions. We denote the sets of nodes and faces of $\mathcal{T}_h$ by $\mathcal{N}_h$ and $\mathcal{F}_h$, respectively. For each $F\in \mathcal{F}_h$, we use  $h_F$ to denote its diameter.
The sets of interface elements and interface faces are  defined by
$
\mathcal{T}_h^\Gamma =\{T\in\mathcal{T}_h :  T\cap \Gamma\not = \emptyset\}$ and $\mathcal{F}_h^\Gamma=\{F\in \mathcal{F}_h : F \cap \Gamma\not = \emptyset\},
$
where we adopt the convention that elements and faces are open sets.
The sets of non-interface elements and non-interface faces are then $\mathcal{T}^{non}_h=\mathcal{T}_h\backslash\mathcal{T}_h^{\Gamma}$ and $\mathcal{F}^{non}_h=\mathcal{F}_h\backslash\mathcal{F}_h^{\Gamma}$, respectively.

Let us introduce the signed distance function: $d(\mathbf{x})|_{\Omega^\pm}=\pm\mbox{dist}(\mathbf{x},\Gamma)$,
and the $\delta$-neighborhood of $\Gamma$:
$
U(\Gamma,\delta)=\{\mathbf{x}\in\mathbb{R}^N: \mbox{dist}(\mathbf{x},\Gamma)< \delta\},
$
where $\mbox{dist}(\cdot,\cdot)$ is the standard distance function  between two  sets.
It is well known that 
for $\Gamma\in C^2$, there exists $\delta_{0}>0$ such that $d(\mathbf{x})$ belongs to $C^2(U(\Gamma,\delta_0))$ and the closest point mapping $\mathbf{p}: U(\Gamma,\delta_0)\rightarrow \Gamma$ maps every $\mathbf{x}$ to precisely one point  on $\Gamma$ (see \cite{foote1984regularity}). 
Now the unit normal vector $\mathbf{n}=\nabla d$ is well defined in $U(\Gamma,\delta_0)$ and  $\mathbf{n}\in C^1(U(\Gamma,\delta_0))^N$.
We  assume $h< \delta_0$ so that the interface is  resolved by the  mesh in the sense that $\overline{T}\subset U(\Gamma,\delta_0)$ for all $T\in\mathcal{T}_h^\Gamma$.
In the neighborhood of $\Gamma$, we recall the following fundamental result which will be useful in the analysis (see, e.g., \cite{burman2018Acut,Li2010Optimal,elliott2013finite}).
\begin{lemma}\label{lem_strip}
For all $\delta\in (0,\delta_0]$, there is a constant $C$ depending only on $\Gamma$ such  that  
\begin{equation*}
\|v\|^2_{L^2(U(\Gamma,\delta))}\leq C(\delta \|v\|^2_{L^2(\Gamma)}+\delta^2  \|\nabla v\|^2_{L^2(U(\Gamma,\delta))})\qquad\forall  v\in H^1(U(\Gamma,\delta_0)).
\end{equation*}
Furthermore, if $v|_{\Gamma}\not=0$, there holds  
$
\|v\|_{L^2(U(\Gamma,\delta))}\leq C\delta^{1/2} \|\nabla v\|_{H^1(U(\Gamma,\delta_0))}.
$
\end{lemma}

Finally, we emphasize that throughout the paper, $C$ or $C$ with a subscript is used to denote generic positive constants that are independent of the mesh size and the interface location relative to the mesh.

\subsection{Discretization of the interface}
Let $I_hd$ be the standard  linear Lagrange interpolant of $d(\mathbf{x})$ associated with $\mathcal{T}_h$, i.e.,
$$(I_hv)|_T \in P_1(T)~ \forall T\in\mathcal{T}_h,   ~(I_hv)(\mathbf{x}_i)=v(\mathbf{x}_i) ~\forall \mathbf{x}_i\in \mathcal{N}_h,$$
where $P_1(T)$ is the space of linear polynomials on $T$.
The approximate interface $\Gamma_h$ is then  chosen as 
$
\Gamma_h=\{\mathbf{x}\in\mathbb{R}^N : I_hd(\mathbf{x})=0\}.
$
Correspondingly, $\Gamma_h$ divides $\Omega$ into two subdomains: $\Omega_h^+$ and $\Omega_h^-$, with $\partial \Omega_h^-=\Gamma_h$.   For all $T\in\mathcal{T}_h^\Gamma$, we define $\Gamma_{h,T}=\Gamma_h \cap T$ and $\Gamma_{T}=\Gamma \cap T$,  assuming $\Gamma_{h,T}\not=\emptyset$ to simplify the presentation and to avoid technical details. 
Note that if there exists an element $T$ such that $\Gamma_{h,T}=\emptyset$,  then the element can be treated as a non-interface element without deteriorating the optimal convergence rate, as the geometric error is of the order $O(h^2)$ (see, e.g., \cite{Li2010Optimal}).

Clearly, $\Gamma_h$ is continuous and piecewise linear, as illustrated in the left plot of Fig~\ref{fig_inter} for the 2D case. 
We note that in traditional 2D IFEs, the interface is discretized by connecting the points where the exact interface intersects the mesh. However, this approach cannot be extended to 3D because there exists the case that the intersection points are not coplanar, as demonstrated in the right plot of Figure~\ref{fig_inter}.
It is also worth noting that our discretization strategy for the interface makes the IFE method developed in this paper especially compatible with the well-known level set method \cite{osher2001level,tornberg2000finite}, which is advantageous for solving complex moving interface problems.

\begin{figure} [htbp]
\centering
\subfigure{
\includegraphics[width=0.28\textwidth]{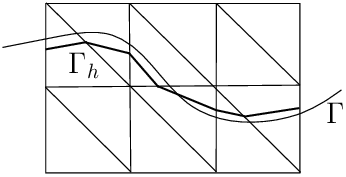}}~~~
\subfigure{
\includegraphics[width=0.2\textwidth]{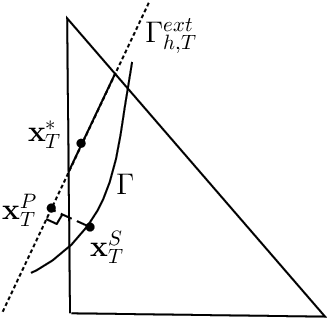}}~~~
\subfigure{
\includegraphics[width=0.15\textwidth]{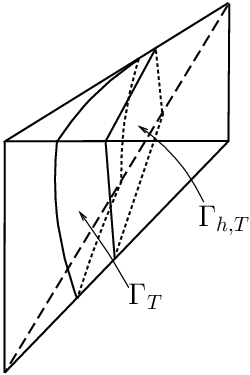}}~~~
 \caption{Left: an example of $\Gamma_h$ in 2D; Middle: an interface element in 2D; Right: an interface element in 3D.\label{fig_inter}} 
\end{figure}

Let $\Gamma^{ext}_{h,T}$ be the $N-1$-dimensional hyperplane containing $\Gamma_{h,T}$ and $\mathbf{n}_{h}$ be a piecewise constant vector defined on all interface elements with $\mathbf{n}_{h}|_T$ being the unit vector perpendicular to $\Gamma^{ext}_{h,T}$  pointing from $\Omega^-_h$ to $\Omega^+_h$.  
We have (see \cite{burman2017cut, ji3Dnonconforming})
\begin{equation}\label{ass_Gamma_h}
\|d\|_{L^\infty(\Gamma_{h,T})}+ \|\mbox{dist}(\cdot,\Gamma_{h,T}^{ext})\|_{L^\infty(\Gamma_T)}+h_T\|\mathbf{n}-\mathbf{n}_{h}\|_{L^\infty(T)}\leq Ch_T^2 \quad \forall T\in\mathcal{T}_h^\Gamma.
\end{equation}
Define a mapping $\mathbf{p}_h:\Gamma_h\rightarrow \Gamma$ by $\mathbf{p}_h(\mathbf{x})=\mathbf{x}+\varrho_h\mathbf{n}_h$ with the smallest $\varrho_h$ chosen such that $\mathbf{x}+\varrho_h\mathbf{n}_h\in\Gamma$. The existence of this mapping is shown in \cite[p. 637]{burman2018Acut}. Moreover, there holds (see (2.18) in \cite{burman2018Acut})
\begin{equation}\label{ph_esti}
\|{\rm id}-\mathbf{p}_h\|_{L^\infty(\Gamma_{h,T})}+ h_T\|\mathbf{n}_h-\mathbf{n}\circ\mathbf{p}\|_{L^\infty(\Gamma_{h,T})}\leq Ch_T^2.
\end{equation}

\subsection{The IFE space for the case $\mathbf{g}=\mathbf{0}$}
For clarity, we first consider the interface problem with homogeneous jump conditions, i.e.,  $\mathbf{g}=\mathbf{0}$.
Let $P_k(D)$ be the set of all polynomials of degree less than or equal to $k$ on a domain $D\subset\mathbb{R}^N$. Define  $D^\pm=D\cap \Omega^\pm$ and $P_k(\cup D^\pm)=\{v : v|_{D^\pm} \in P_k(D^\pm)\}$. Similarly, we define $\boldsymbol{P}_k(D)=P_k(D)^N$ and $\boldsymbol{P}_k(\cup D^\pm)=P_k(\cup D^\pm)^N$.
On each element $T\in\mathcal{T}_h$, the conventional mini element space is
$
(\boldsymbol{V}_h(T), Q_h(T))=(\boldsymbol{P}_1(T), P_1(T))\bigoplus (\mbox{span}\{b_T\}^N, \{0\}),
$
where $b_T$ is the standard bubble function associated with  $T$ (see \cite{arnold1984stable}).

On each interface element $T\in\mathcal{T}_h^\Gamma$, we need to modify $(\boldsymbol{P}_1(T), P_1(T))$ since it cannot  capture the behavior of the exact solution well due to the interface jump conditions. 
Let $\mathcal{I}_N=\{1,2,...,N\}$ and $\mathbf{t}_{i,h}$, $i\in\mathcal{I}_{N-1}$, be unit tangent vectors to $\Gamma_{h}$ such that $\mathbf{t}_{i,h}$ and $\mathbf{n}_h$ form standard basis vectors in $\mathbb{R}^N$. 
Define $T_h^\pm=T\cap\Omega_h^\pm$ and $\mu_h|_{\Omega_h^\pm}=\mu^\pm$.
According to (\ref{jp_cond2}), (\ref{jp_cond3}) and (\ref{jp_cond1}) with $\mathbf{g}=\mathbf{0}$, we introduce the following discrete interface jump conditions for all $ (\mathbf{v}, q) \in  (\boldsymbol{P}_1(\cup T_h^\pm), P_1(\cup T_h^\pm))$:
\begin{subequations}\label{dis_jp0}
\begin{align}
& [\sigma(\mu_h,\mathbf{v},q)\mathbf{n}_h]_{\Gamma_{h,T}}(\mathbf{x}_T^*)=\mathbf{0},~~\mbox{($\mathbf{x}_T^*$ is a point on $\Gamma_{h,T}$),} \label{dis_jp1}\\
&[\mathbf{v}]_{\Gamma_{h,T}}=\mathbf{0} \mbox{ (or, equivalently, $[\mathbf{v}]_{\Gamma_{h,T}}(\mathbf{x}_T^P)=\mathbf{0}$, $[(\nabla \mathbf{v})\mathbf{t}_{i,h}]_{\Gamma_{h,T}}=\mathbf{0}~\forall i\in\mathcal{I}_{N-1}$)}, \label{dis_jp2}\\
&[\nabla\cdot \mathbf{v}]_{\Gamma_{h,T}}=0,\label{dis_jp3}\\
&\nabla q\in P_0(T)^N \mbox{ (or, equivalently, $[\nabla q]_{\Gamma_{h,T}}=\mathbf{0}$)}, \label{dis_jp4}
\end{align}
\end{subequations}
where the gradient of $q$ (also denoted by $\nabla q$ for simplicity) is understood in a piecewise sense since $q$ is broken across $\Gamma_{h,T}$.

The modified space is then defined by
$$
\widetilde{\boldsymbol{P}_1P_1}(T)=\{   (\mathbf{v}, q)\in  (\boldsymbol{P}_1(\cup T_h^\pm), P_1(\cup T_h^\pm)) :  (\mathbf{v}, q) \mbox{  satisfies  (\ref{dis_jp0})} \}.
$$
Let $\mathbf{a}_{j,T}$, $j\in \mathcal{I}_{N+1}$ be vertices of $T$ and write $\mathbf{v}=(v_1,v_2,...,v_{N})^T$. The degrees of freedom of the space $\widetilde{\boldsymbol{P}_1P_1}(T)$, denoted by ${\rm DoF}_{k,T}$, $k=1, ..., (N+1)^2$, are chosen as 
$$
{\rm DoF}_{j+(i-1)(N+1),T}(\mathbf{v},q)=v_i(\mathbf{a}_{j,T}),~ {\rm DoF}_{(N+1)N+j,T}(\mathbf{v},q)=q(\mathbf{a}_{j,T}), ~\forall j\in \mathcal{I}_{N+1},~ \forall i\in \mathcal{I}_{N}.
$$

\begin{remark}
Condition (\ref{dis_jp4})  is added not only for  the unisolvence of basis functions but also for the inf-sup stability (see the proof of Lemma \ref{lem_infsup1}).
Condition (\ref{dis_jp4}) also implies that $[q]_{\Gamma_{h,T}}$ is a constant on $\Gamma_{h,T}$, and so is $[\sigma(\mu_h,\mathbf{v},q)\mathbf{n}_h]_{\Gamma_{h,T}}$. Thus, (\ref{dis_jp1}) is equivalent to $ [\sigma(\mu_h,\mathbf{v},q)\mathbf{n}_h]_{\Gamma_{h,T}}=\mathbf{0}$.
\end{remark}

\begin{remark}
Since  $\mathbf{v}$ and $q$ are piecewise linear, we know $(\mathbf{v}, q)$ has $2(N+1)^2$ parameters. On the other hand, there are $(N+1)^2$ degrees of freedom and $(N+1)^2$  constraints according to (\ref{dis_jp0}). More precisely, (\ref{dis_jp1}) provides $N$ constraints; (\ref{dis_jp2}) provides $N^2$ constraints; (\ref{dis_jp3}) provides one constraint; and (\ref{dis_jp4}) provides $N$ constraints. Intuitively, we can expect that $(\mathbf{v}, q)$ can be uniquely determined by ${\rm DoF}_{k,T}$, $k=1, 2, ..., (N+1)^2$, i.e., the nodal values $\mathbf{v}(\mathbf{a}_{j,T})$ and $q(\mathbf{a}_{j,T})$.  In subsection~\ref{sec_Unisolvence}, we will prove that this property holds on Cartesian meshes.
\end{remark}

\begin{remark}
The point $\mathbf{x}_T^*\in  \Gamma_{h,T}$ is arbitrary but fixed.
The point $\mathbf{x}_T^P\in \Gamma^{ext}_{h,T}$ is only for theoretical analysis.
We set $\mathbf{x}_T^P=\mathbf{p}_{\Gamma_{h,T}^{ext}}(\mathbf{x}_T^S)$ with $\mathbf{x}_T^S$ being a point on the surface $\Gamma_{T}$, where $\mathbf{p}_{\Gamma_{h,T}^{ext}}$ is the orthogonal projection 
onto the plane $\Gamma_{h,T}^{ext}$ (see the middle plot in Figure~\ref{fig_inter} for the 2D case).
From  (\ref{ass_Gamma_h}) we have 
\begin{equation}\label{ineq_xpxs}
\left|\mathbf{x}_T^P-\mathbf{x}_T^S\right|\leq Ch_T^2~\mbox{ and }~|\mathbf{x}-\mathbf{x}_T^P|\leq |\mathbf{x}-\mathbf{x}_T^S|+|\mathbf{x}_T^S-\mathbf{x}_T^P|\leq Ch_T~ \forall \mathbf{x}\in T,
\end{equation}
where $|\cdot |$ denotes the Euclidean norm.
\end{remark}

Now the local and global mini IFE space are given by
$
\widetilde{\boldsymbol{V}Q_h}(T)=\widetilde{\boldsymbol{P}_1P_1}(T)\bigoplus (\mbox{span}\{b_T\}^N,\{0\})
$
and 
\begin{equation*}
\begin{aligned}
\widetilde{\boldsymbol{V}Q_h}(\Omega)=
 \{ (\mathbf{v},q) :  ~&(\mathbf{v}, q)|_T\in  (\boldsymbol{V}_h(T), Q_h(T)) ~\forall T\in\mathcal{T}_h^{non},~(\mathbf{v}, q)|_T\in \widetilde{\boldsymbol{V}Q_h}(T)~\forall T\in\mathcal{T}_h^\Gamma,\\
&~\mathbf{v} \mbox{ and } q \mbox{ are continuous at every } \mathbf{x}_i\in\mathcal{N}_h~  \}.
\end{aligned}
\end{equation*}
In addition, we introduce a subspace of $\widetilde{\boldsymbol{V}Q_h}(\Omega)$ by 
$$\widetilde{\boldsymbol{V}Q_{h,0}}(\Omega)=  \{ (\mathbf{v},q)\in\widetilde{\boldsymbol{V}Q_h}(\Omega):~\mathbf{v}|_{\partial \Omega}=\mathbf{0}, ~ q\in L_0^2(\Omega) \}.$$
\begin{remark}
The standard bubble functions, which do not encode the jump conditions, are added to ensure the inf-sup condition. Because they possess higher frequencies, the optimal approximation capability of the IFE space is not affected. 
We also note that the IFE functions, including the bubble functions, are continuous across the interface. 
This continuity means that the IFE space is conforming on each interface element, thus eliminating the need to add a penalty term across the interface in the finite element formulation (\ref{def_Ah}).
\end{remark}

\subsection{The IFE method for the case $\mathbf{g}=\mathbf{0}$}
Similar to the IFE methods for elliptic interface problems, the velocity in the IFE space is discontinuous across  interface faces, i.e., the IFE space is non-conforming. To get a consistent method, we should add some integral terms on interface faces, like the discontinuous Galerkin method. 
Given a face $F\in \mathcal{F}_h$ shared by two elements $T_1$ and $T_2$, let $\mathbf{n}_F$ be the unit normal vector of $F$ pointing from $T_1$ to $T_2$. 
The jump and the average of a function $v$ across the face $F$ are denoted by 
$[v]_F:=v|_{T_1}-v|_{T_2}$ and  $\{v\}_F:=\frac{1}{2}(v|_{T_1}+v|_{T_2})$, respectively.

Let  $\mathbf{f}^\pm_E$ be some extension of $\mathbf{f}^\pm:=\mathbf{f}|_{\Omega^\pm}$ such that $\mathbf{f}^\pm_E|_{\Omega^\pm}=\mathbf{f}^\pm$ and $\|\mathbf{f}^\pm_E\|_{L^2(\Omega)}\leq \|\mathbf{f}^\pm\|_{L^2(\Omega^\pm)}$. Obviously, the trivial extension, $\mathbf{f}^\pm_E=0$ outside $\Omega^\pm$, satisfies the requirements.
We approximate the source term $\mathbf{f}$ by $\mathbf{f}^{BK}|_{\Omega_h^\pm}=\mathbf{f}_E^\pm|_{\Omega_h^\pm}$. Define the mismatch region by $\Omega^\triangle=(\Omega_h^+\backslash \Omega^+)\cup(\Omega_h^-\backslash \Omega^-)$. Clearly, $\mathbf{f}|_{\Omega\backslash\Omega^\triangle}=\mathbf{f}^{BK}|_{\Omega\backslash\Omega^\triangle}$ and
$\|\mathbf{f}^{BK}\|_{L^2(\Omega^\triangle)}\leq C\|\mathbf{f}\|_{L^2(\Omega)}$.
Define the following forms with $\gamma=1$ or $-1$ and $\eta\geq0$: 
\begin{equation}\label{def_Ah}
\begin{aligned}
&\mathcal{A}_h(\mathbf{u}_h, p_h; \mathbf{v}_h, q_h ):=a_h(\mathbf{u}_h,\mathbf{v}_h)+b_h(\mathbf{v}_h,p_h)-b_h(\mathbf{u}_h,q_h),\\
&a_h(\mathbf{u}_h,\mathbf{v}_h):=\sum_{T\in\mathcal{T}_h}\int_T2\mu_h \boldsymbol{\epsilon}(\mathbf{u}_h):\boldsymbol{\epsilon}(\mathbf{v}_h)+\sum_{F\in\mathcal{F}_h^\Gamma}\frac{1+\eta}{h_F}\int_F[\mathbf{u}_h]_F\cdot[\mathbf{v}_h]_F\\
&\qquad\qquad-\sum_{F\in\mathcal{F}_h^\Gamma}\int_F \left(\{ 2\mu_h\boldsymbol{\epsilon}(\mathbf{u}_h)\mathbf{n}_F\}_F\cdot[\mathbf{v}_h]_F+\gamma\{ 2\mu_h\boldsymbol{\epsilon}(\mathbf{v}_h)\mathbf{n}_F\}_F\cdot[\mathbf{u}_h]_F\right),\\
&b_h(\mathbf{v}_h,q_h):=-\sum_{T\in\mathcal{T}_h}\int_T q_h\nabla\cdot\mathbf{v}_h +\sum_{F\in\mathcal{F}_h^\Gamma}\int_F\{q_h\}_F[\mathbf{v}_h\cdot\mathbf{n}_F]_F,\\
&l_h(\mathbf{v}_h):=\int_{\Omega} \mathbf{f}^{BK}\cdot\mathbf{v}_h.
\end{aligned}
\end{equation}
The discretization of (\ref{weakform}) with $\mathbf{g}=\mathbf{0}$ reads: Find $(\mathbf{u}_h, p_h)\in \widetilde{\boldsymbol{V}Q_{h,0}}(\Omega)$ such that 
\begin{equation}\label{IFE_method_hom}
\mathcal{A}_h(\mathbf{u}_h,p_h; \mathbf{v}_h,q_h )=l_h(\mathbf{v}_h) \qquad \forall (\mathbf{v}_h, q_h)\in  \widetilde{\boldsymbol{V}Q_{h,0}}(\Omega).
\end{equation}

\subsection{Extension to the case $\mathbf{g}\not=\mathbf{0}$}\label{subsec_corre_fun}
Now, we extend the IFE method to the case with non-homogeneous jump conditions.
We need a correction function $(\mathbf{u}_h^J, p_h^J)$ to deal with the jump condition (\ref{jp_cond1}) when $\mathbf{g}\not=\mathbf{0}$. On non-interface elements, we set $(\mathbf{u}_h^J, p_h^J)=(\mathbf{0},0)$, and on each interface element $T\in\mathcal{T}_h^\Gamma$, we construct  $(\mathbf{u}_h^J, p_h^J)\in (\boldsymbol{P}_1(\cup T_h^\pm), P_1(\cup T_h^\pm))$ satisfying
$\mathbf{u}_h^J(\mathbf{a}_{i,T})=\mathbf{0}, p_h^J(\mathbf{a}_{i,T})=0$  for all $i\in \mathcal{I}_{N+1}$
and the jump conditions in (\ref{dis_jp0}) with (\ref{dis_jp1}) changed to 
\begin{equation}\label{def_uhj}
[\sigma(\mu_h,\mathbf{u}_h^J,p_h^J)\mathbf{n}_h]_{\Gamma_{h,T}}={\rm avg}_{\Gamma_{R_T}}(\mathbf{g}),~ \mbox{ where }  ~{\rm avg}_{\Gamma_{R_T}}(\mathbf{g}):=|\Gamma_{R_T}|^{-1}\int_{\Gamma_{R_T}}\mathbf{g}.
\end{equation}
Here $\Gamma_{R_T}:=\Gamma\cap R_T$ and $R_T\subset \mathbb{R}^N$ is a larger domain containing $T$
such that its diameter $h_{R_T}\leq Ch_T$ and $|\Gamma_{R_T}|\geq Ch_T^{N-1}$. We note that, in this occasion,  $|\cdot|$ means the measure of domains or manifolds.

\begin{figure} [htbp]
\centering
\includegraphics[width=0.3\textwidth]{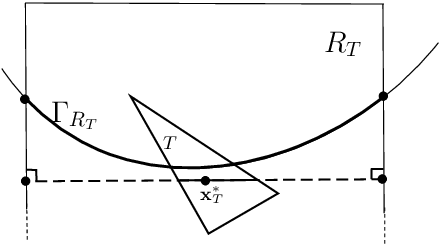}
 \caption{2D illustration for  the construction of $R_T$ and $\Gamma_{R_T}$. \label{fig_RT}} 
\end{figure}

Here we give an example of $R_T$. For ease of implementation, we choose 
$R_T=\{\mathbf{x}\in \mathbb{R}^N : \mathbf{x}=\mathbf{x}_{T}^*+\xi_N\mathbf{n}_{h}+\sum_{i=1}^{N-1} \xi_i\mathbf{t}_{i,h}, |\xi_i|\leq  h_T ~\forall i\in \mathcal{I}_N \}$. A 2D illustration of $R_T$ and $\Gamma_{R_T}$ is shown in Figure~\ref{fig_RT}.
By (\ref{ph_esti}), it is  not hard to see that there exists a constant $C_\Gamma$ such that
$\mbox{dist}(\mathbf{x},\Gamma_{h,T}^{ext})\leq C_\Gamma h_T^2$ for all $\mathbf{x}\in \Gamma_{R_T}$.
Then we further assume $h<\min\{1/(2C_\Gamma),\delta_0/(\sqrt{N})\}$ so that $\overline{R_T}\subset U(\Gamma,\delta_0)$ and $\mbox{dist}(\mathbf{x},\Gamma_{h,T}^{ext})\leq h_T/2$ for all $\mathbf{x}\in \Gamma_{R_T}$.
Now $\Gamma_{R_T}$ can be parameterized as $\mathbf{x}(\xi_1,...,\xi_{N-1})=\mathbf{x}_{T}^*+ \xi_N(\xi_1,...,\xi_{N-1})\mathbf{n}_{h}+ \sum_{i=1}^{N-1}\xi_i\mathbf{t}_{i,h}$, $-h_T\leq \xi_1,...,\xi_{N-1}\leq h_T$, where $\xi_N(\xi_1,...,\xi_{N-1})$ is chosen such that $\mathbf{x}(\xi_1,...,\xi_{N-1})\in\Gamma$. The parametric representation of $\Gamma_{R_T}$ is useful in computing  ${\rm avg}_{\Gamma_{R_T}}(\mathbf{g})$ practically. Obviously, the requirements $|\Gamma_{R_T}|\geq Ch_T^{N-1}$ and $h_{R_T}\leq Ch_T$ are fulfilled. Applying these and the following well-known trace inequality for parts of $\Gamma$ (see, e.g., \cite[Lemma 1]{guzman2018inf}):
\begin{equation}\label{tac_ine}
\|v\|^2_{L^2(\Gamma_{R_T})}\leq C(h_T^{-1}\|v\|^2_{L^2(R_T)}+h_T|v|^2_{H^1(R_T)})\quad \forall v\in H^1(R_T),
\end{equation}
we obtain the following useful inequality:
\begin{equation}\label{avr_tac_ine}
\begin{aligned}
|{\rm avg}_{\Gamma_{R_T}}(v)|^2\leq |\Gamma_{R_T}|^{-1}\|v\|^2_{L^2(\Gamma_{R_T})}\leq Ch_T^{-N}(\|v\|^2_{L^2(R_T)}+h_T^2|v|^2_{H^1(R_T)})\quad \quad \forall v\in H^1(R_T).
\end{aligned}
\end{equation}

For the nonzero force $\mathbf{g}$ defined on $\Gamma$, we map it to $\Gamma_h$ using $\mathbf{g}_h=\mathbf{g}\circ \mathbf{p}_h$. Consequently, we modify the linear form $l_h(\cdot)$, previously defined in (\ref{def_Ah}), to
$$l_h(\mathbf{v}_h):=\int_{\Omega} \mathbf{f}^{BK}\cdot\mathbf{v}_h-\int_{\Gamma_h}\mathbf{g}_h\cdot\mathbf{v}_h.$$
Alternatively, we can use $\mathbf{g}_h=\mathbf{g}\circ \mathbf{p}$.
However, as shown in \cite{burman2018Acut}, the computation of $\mathbf{p}(\mathbf{x})$ for a given $\mathbf{x}\in\Gamma_h$ is substantially more costly than that of $\mathbf{p}_h(\mathbf{x})$.

The discretization of (\ref{weakform}) now reads: Find $(\mathbf{U}_h, P_h):=(\mathbf{u}_h, p_h)+(\mathbf{u}_h^J,p_h^J-c_h^J)$ such that 
\begin{equation}\label{IFE_method}
\begin{aligned}
 &(\mathbf{u}_h, p_h)\in \widetilde{\boldsymbol{V}Q_{h,0}}(\Omega),\quad c_h^J=\int_\Omega p_h^J,\\
&\mathcal{A}_h(\mathbf{u}_h,p_h; \mathbf{v}_h,q_h )=l_h(\mathbf{v}_h)-\mathcal{A}_h(\mathbf{u}_h^J,p_h^J-c_h^J; \mathbf{v}_h,q_h ) \qquad \forall (\mathbf{v}_h, q_h)\in  \widetilde{\boldsymbol{V}Q_{h,0}}(\Omega),
\end{aligned}
\end{equation}
where the constant $c_h^J$ is used to ensure $P_h\in L_0^2(\Omega)$.
\begin{remark}
When $\mathbf{g}=\mathbf{0}$, it obviously follows that $\mathbf{u}_h^J=\mathbf{0}$ and $p_h^J=c_h^J=0$. Consequently,  (\ref{IFE_method}) simplifies to (\ref{IFE_method_hom}).
For completeness, in sections 3 and 4, we conduct the theoretical analysis for (\ref{IFE_method}), applicable to the general interface problem (\ref{weakform}) where $\mathbf{g}$ may or may not be zero. It's important to highlight that, by setting $\mathbf{u}_h^J=\mathbf{0}$ and $p_h^J=c_h^J=0$ in these sections, we achieve a more straightforward analysis for the case when $\mathbf{g}=\mathbf{0}$.
\end{remark}

For practical implementations, we next derive explicit forms of basis and correction functions.
\subsection{Explicit formulas for IFE basis functions}\label{sec_Unisolvence} 
The space $\widetilde{\boldsymbol{P}_1P_1}(T)$ is not empty since $(\mathbf{0},0)\in\widetilde{\boldsymbol{P}_1P_1}(T)$. For each $(\mathbf{v},q)\in \widetilde{\boldsymbol{P}_1P_1}(T)$, if  
$$
[\sigma(1,\mathbf{v},q)\mathbf{n}_h]_{\Gamma_{h,T}}=\sum_{i=1}^{N-1}c_i\mathbf{t}_{i,h}+c_N\mathbf{n}_h,
$$
 we have 
\begin{equation}\label{decomp1}
(\mathbf{v},q)=(\mathbf{v}^{J_0},q^{J_0})+\sum_{i=1}^{N}c_i(\mathbf{v}^{J_i}, q^{J_i}),
\end{equation}
where $(\mathbf{v}^{J_0}, q^{J_0})$ is defined by
\begin{equation}\label{def_j0}
(\mathbf{v}^{J_0}, q^{J_0})\in (\boldsymbol{P}_1(T), P_1(T)),~ {\rm DoF}_{i,T}(\mathbf{v}^{J_0}, q^{J_0})={\rm DoF}_{i,T}(\mathbf{v},q)~\forall i\in\{1, 2, ..., (N+1)^2\},
\end{equation}
and $(\mathbf{v}^{J_i}, q^{J_i})$ is defined by
\begin{equation}\label{def_ji}
\begin{aligned}
&(\mathbf{v}^{J_i}, q^{J_i})\in (\boldsymbol{P}_1(\cup T_h^\pm), P_1(\cup T_h^\pm)),~ {\rm DoF}_{j,T}(\mathbf{v}^{J_i}, q^{J_i})=0~\forall j\in\{1, 2, ..., (N+1)^2\},\\
&[\sigma(1,\mathbf{v}^{J_i},q^{J_i})\mathbf{n}_h]_{\Gamma_{h,T}}=\left\{
\begin{array}{ll}
\mathbf{t}_{i,h} & \mbox{ if }~ i\in\{1,..., N-1\},\\
\mathbf{n}_h&\mbox{ if }~ i=N,
\end{array}
\right.\\
&[\mathbf{v}^{J_i}]_{\Gamma_{h,T}}(\mathbf{x}_T^P)=\mathbf{0}, ~[(\nabla \mathbf{v}^{J_i})\mathbf{t}_{j,h}]_{\Gamma_{h,T}}=\mathbf{0}~\forall  j\in\{1,..., N-1\},\\
&[\nabla\cdot\mathbf{v}^{J_i}]_{\Gamma_{h,T}}=0,~[\nabla q^{J_i}]_{\Gamma_{h,T}}=\mathbf{0}.
\end{aligned}
\end{equation}
Now the problem is to find  $c_i$  so that the jump condition $[\sigma(\mu_h,\mathbf{v},q)\mathbf{n}_h]_{\Gamma_{h,T}}=\mathbf{0}$ is satisfied.
Substituting (\ref{decomp1}) into this jump condition, we obtain a $N$-by-$N$  system of  linear equations:
\begin{equation}\label{eq_cj1}
\sum_{i=1}^{N}[\sigma(\mu_h,\mathbf{v}^{J_i}, q^{J_i})\mathbf{n}_h]_{\Gamma_{h,T}}c_i=-[\sigma(\mu_h,\mathbf{v}^{J_0},q^{J_0})\mathbf{n}_h]_{\Gamma_{h,T}}=-\sigma([\mu_h]_{\Gamma_{h,T}},\mathbf{v}^{J_0},0)\mathbf{n}_h,
\end{equation}
where we note that $[\sigma(\mu_h,\mathbf{v}^{J_i}, q^{J_i})\mathbf{n}_h]_{\Gamma_{h,T}}$ is a $N$-dimensional column vector.
To solve the above system of  linear equations,  we need an explicit expression of $(\mathbf{v}^{J_i}, q^{J_i})$. Let $I_{h,T}$  be the standard linear nodal interpolation operator on the element $T$, then we have the following lemma.
\begin{lemma}\label{lem_vj}
Let $(\mathbf{v}^{J_i}, q^{J_i})$  be defined in (\ref{def_j0}) and (\ref{def_ji}). There holds  
\begin{equation}{\label{def_vjqj}}
(\mathbf{v}^{J_i}, q^{J_i})=
\left\{
\begin{aligned}
&(I_{h,T}\mathbf{v}, I_{h,T}q ) &&\mbox{ if }~ i=0,\\
&((w_T-I_{h,T}w_T)\mathbf{t}_{i,h}, 0) &&\mbox{ if }~ i\in\{1,...,N-1\},\\
&(\mathbf{0}, z_T-I_{h,T}z_T)  &&\mbox{ if }~ i=N,
\end{aligned}\right.
\end{equation}
where
\begin{equation}\label{def_zw}
z_T(\mathbf{x})=\left\{
\begin{aligned}
&-1~~ &&\mbox{ if }~\mathbf{x}\in T_h^+,\\
&0\quad &&\mbox{ if }~\mathbf{x}\in T_h^-,
\end{aligned}\right.
\qquad
w_T(\mathbf{x})=\left\{
\begin{aligned}
&\mbox{dist}(\mathbf{x},\Gamma_{h,T}^{ext})~~ &&\mbox{ if }~\mathbf{x}\in T_h^+,\\
&0\quad &&\mbox{ if }~\mathbf{x}\in T_h^-.
\end{aligned}\right.
\end{equation}
\end{lemma}
\begin{proof}
The result is trivial for $i=0$.  For $i\in\mathcal{I}_N$,  we define a new function having the same interface jump conditions as $(\mathbf{v}^{J_i}, q^{J_i})$ by 
\begin{equation}\label{pro_vi0}
(\mathbf{v}_i, q_i)=\left\{
\begin{aligned}
&(\mathbf{v}_i^+, q_i^+)\in (P_1(T_h^+)^N, P_1(T_h^+)) &&\mbox{ in } T_h^+,\\
&(\mathbf{0},0) &&\mbox{ in } T_h^-,
\end{aligned}
\right.
\end{equation}
where the linear functions $\mathbf{v}_i^+$ and $q_i^+$ are chosen such that 
\begin{equation}\label{pro_vi1}
\begin{aligned}
&\sigma(1,\mathbf{v}_i^+,q_i^+)\mathbf{n}_h=\left\{
\begin{array}{ll}
\mathbf{t}_{i,h} &\mbox{ if }~ i\in\{1,..., N-1\},\\
\mathbf{n}_h&\mbox{ if }~ i=N,
\end{array}
\right.\\
&\mathbf{v}_i^+(\mathbf{x}_T^P)=\mathbf{0},~(\nabla \mathbf{v}_i^+)\mathbf{t}_{j,h}=\mathbf{0}~\forall j\in\{1,..., N-1\},~~\nabla\cdot\mathbf{v}_i^+=0,~\nabla q_i^+=\mathbf{0}.
\end{aligned}
\end{equation}
Then, by definition, the function $(\mathbf{v}^{J_i}, q^{J_i})$ can be constructed as
\begin{equation}\label{pro_vi2}
(\mathbf{v}^{J_i}, q^{J_i})=(\mathbf{v}_i, q_i)-(I_{h,T}\mathbf{v}_i, I_{h,T}q_i),
\end{equation}
where we used the facts that $(I_{h,T}\mathbf{v}_i, I_{h,T}q_i)$ has no jumps across $\Gamma_{h,T}$ and $(\mathbf{v}_i, q_i)-(I_{h,T}\mathbf{v}_i, I_{h,T}q_i)$ vanishes on all vertices of $T$. 
Considering (\ref{pro_vi1}) in the $\mathbf{t}_{i,h}$-$\mathbf{n}_h$ coordinate system and using an argument similar to the one in the proof of Lemma 4.5 in \cite{2021ji_IFE_stoke}, it is not hard to see 
\begin{equation}\label{pro_vi3}
(\mathbf{v}_i^+(\mathbf{x}), q_i^+)=
\left\{
\begin{aligned}
&(\mbox{dist}(\mathbf{x},\Gamma_{h,T}^{ext})\mathbf{t}_{i,h}, 0)&&\mbox{ if }~i\in\{1,...,N-1\},\\
&(\mathbf{0}, -1) &&\mbox{ if }~i=N.
\end{aligned}
\right.
\end{equation}
The result (\ref{def_vjqj}) follows from (\ref{pro_vi0}), (\ref{pro_vi2}) and (\ref{pro_vi3}).
\end{proof}

Substituting (\ref{def_vjqj}) into (\ref{eq_cj1}) and  using the $\mathbf{t}_{i,h}$-$\mathbf{n}_h$ coordinate system, we obtain 
\begin{equation}\label{eq_cj2}
\begin{aligned}
(1+(\mu^-/\mu^+-1)\nabla I_{h,T}w_T\cdot \mathbf{n}_h)c_i&=\mathbf{t}_{i,h}^T\sigma(\mu^-/\mu^+-1,I_{h,T}\mathbf{v},0)\mathbf{n}_h,~i=1,...,N-1,\\
c_N&=\mathbf{n}_{h}^T\sigma(\mu^--\mu^+,I_{h,T}\mathbf{v},0)\mathbf{n}_h,
\end{aligned}
\end{equation}
where we have used the identity $[\mu_h\partial_{\mathbf{n}_h} (w_T-I_{h,T}w_T) ]_{\Gamma_{h,T}}=\mu^++(\mu^--\mu^+)\nabla I_{h,T}w_T\cdot\mathbf{n}_h$ in the derivation.
On the Cartesian meshes used in this paper, the relation
$0 \leq \nabla I_{h,T}w_T\cdot \mathbf{n}_h\leq 1$ holds for both $N=2$ and $N=3$ (see Lemmas 1 and 16 in \cite{2021ji_IFE} and Remark 3 in \cite{kvrivzek1992maximum}). Then we have
\begin{equation}\label{fengmu_jie}
1+(\mu^-/\mu^+-1)\nabla I_{h,T}w_T\cdot \mathbf{n}_h\geq\left\{
\begin{aligned}
&1\qquad&&\mbox{ if }\mu^-/\mu^+\geq 1,\\
&\mu^-/\mu^+\qquad&&\mbox{ if }0<\mu^-/\mu^+< 1,
\end{aligned}\right.
\end{equation}
which implies that (\ref{eq_cj2}) is unisolvent. Combining (\ref{fengmu_jie}), (\ref{decomp1}) and (\ref{def_vjqj}) yields the following lemma.
\begin{lemma}\label{lem_IFEbasis}
For any $T\in\mathcal{T}_h^\Gamma$ with vertices $\mathbf{a}_{i,T}$, $i\in \mathcal{I}_{N+1}$, the  function $(\mathbf{v},q)\in\widetilde{\boldsymbol{P}_1P_1}(T)$ is uniquely determined by nodal values $\mathbf{v}(\mathbf{a}_{i,T})$ and $q(\mathbf{a}_{i,T})$. Furthermore, we have 
\begin{equation}\label{explicit_formula}
(\mathbf{v},q)=(I_{h,T}\mathbf{v}, I_{h,T}q )+\sum_{i=1}^{N-1}(c_i(w_T-I_{h,T}w_T)\mathbf{t}_{i,h}, 0)+(\mathbf{0}, c_N(z_T-I_{h,T}z_T)),
\end{equation}
where $c_i$, $i=1,...,N$, are determined from (\ref{eq_cj2}).
\end{lemma}

\begin{remark}\label{rema_dep}
From the above lemma, one can see that $\widetilde{\boldsymbol{P}_1P_1}(T)=(\boldsymbol{P}_1(T), P_1(T))$ if $\mu^+=\mu^-$. Thus, the IFE reduces to the conventional mini element method in the absence of interfaces.  We also find that $\mathbf{v}$ depends only on $\mathbf{v}(\mathbf{a}_i)$, while  $q$  depends not only on $\mathbf{v}(\mathbf{a}_i)$  but also on $q(\mathbf{a}_i)$. Note that $(\mathbf{0}, c)\in\widetilde{\boldsymbol{P}_1P_1}(T)$ for  any constant $c$. We can choose a proper constant $c$ to make the pressure have average zero.
\end{remark}

Define the basis functions of $\widetilde{\boldsymbol{P}_1P_1}(T)$ by ${\rm DoF}_{l,T}(\boldsymbol{\phi}_{k,T}^{\Gamma}, \varphi_{k,T}^{\Gamma})=\delta_{kl}$ (the Kronecker symbol) for all  $k, l=1,..., (N+1)^2.$ The standard basis functions of $(\boldsymbol{P}_1(T),P_1(T))$, denoted by $(\boldsymbol{\phi}_{k,T}, \varphi_{k,T})$, are defined analogously.
By Lemma~\ref{lem_IFEbasis} we can write the IFE  basis functions explicitly as follows:
\begin{equation*}
\boxed{
\begin{aligned}
&\boldsymbol{\phi}_{k,T}^{\Gamma}=\left\{
\begin{aligned}
&\boldsymbol{\phi}_{k,T}+\sum_{i=1}^{N-1}\frac{\mathbf{t}_{i,h}^T\sigma(\mu^-/\mu^+-1,\boldsymbol{\phi}_{k,T},0)\mathbf{n}_h}{1+(\mu^-/\mu^+-1)\nabla I_{h,T}w_T\cdot \mathbf{n}_h}(w_T- I_{h,T}w_T)\mathbf{t}_{i,h},~k=1,..., N(N+1),\\
&\mathbf{0},\qquad\qquad \qquad\qquad\qquad\qquad\qquad\qquad\qquad\qquad k=N(N+1)+1,..., (N+1)^2,
\end{aligned}
\right.
\\
&\varphi_{k,T}^{\Gamma}=\left\{
\begin{aligned}
&\mathbf{n}_h^T\sigma(\mu^--\mu^+,\boldsymbol{\phi}_{k,T},0)\mathbf{n}_h(z_T-I_{h,T} z_T),~&&k=1,..., N(N+1),\\
&\varphi_{k,T}, &&k=N(N+1)+1,..., (N+1)^2.
\end{aligned}
\right.
\end{aligned}}
\end{equation*}
Clearly, adopting the $\mathbf{t}_{i,h}$-$\mathbf{n}_h$ coordinate system significantly simplifies the formulas. These are detailed in Appendix~\ref{sec_app_basis}.

Using the above explicit formulas, the definition of $z_T$ and $w_T$ in (\ref{def_zw})  and the estimates of standard  basis functions,  it is easy to prove the following lemma.
\begin{lemma}\label{lem_basisest}
There exists a positive constant $C$  such that for $m=0,1$,
\begin{equation*}
\begin{aligned}
&|(\boldsymbol{\phi}_{i,T}^{\Gamma})^\pm|_{W_\infty^m(T)}\leq Ch_T^{-m},~~\|(\varphi_{i,T}^{\Gamma})^\pm\|_{L^\infty(T)}\leq Ch_T^{-1},~~i=1,..., N(N+1),\\
&|(\boldsymbol{\phi}_{i,T}^{\Gamma})^\pm|_{W_\infty^m(T)}=0,~~\|(\varphi_{i,T}^{\Gamma})^\pm\|_{L^\infty(T)}=1,~~i=N(N+1)+1,..., (N+1)^2.
\end{aligned}
\end{equation*}
\end{lemma}

\subsection{Explicit formulas for correction functions}
Similarly to (\ref{decomp1}), we can write
$
(\mathbf{u}_h^J,p_h^J)|_T=\sum_{i=1}^{N}c_i^J(\mathbf{v}^{J_i}, q^{J_i}),
$
where the $c_i^J$'s  are constants to be determined. Substituting this equation into (\ref{def_uhj}) yields
$
\sum_{i=1}^{N}[\sigma(\mu_h,\mathbf{v}^{J_i}, q^{J_i})\mathbf{n}_h]_{\Gamma_{h,T}}c_i^J={\rm avg}_{\Gamma_{R_T}}(\mathbf{g}).
$
Analogous to (\ref{eq_cj2}), we have
\begin{equation*}
\mu^+(1+(\mu^-/\mu^+-1)\nabla I_{h,T}w_T\cdot \mathbf{n}_h)c_i^J=\mathbf{t}_{i,h}^T{\rm avg}_{\Gamma_{R_T}}(\mathbf{g}),~i=1,...,N-1,~~c_N^J=\mathbf{n}_{h}^T{\rm avg}_{\Gamma_{R_T}}(\mathbf{g}).
\end{equation*}
Combining this with Lemma~\ref{lem_vj} yields the following explicit formulas:
\begin{equation}\label{ex_uphj}
\boxed{
\mathbf{u}_h^J|_T=\sum_{i=1}^{N-1}\frac{\mathbf{t}_{i,h}^T{\rm avg}_{\Gamma_{R_T}}(\mathbf{g})(w_T-I_{h,T}w_T)\mathbf{t}_{i,h}}{\mu^+(1+(\mu^-/\mu^+-1)\nabla I_{h,T}w_T\cdot \mathbf{n}_h)},~~p_h^J|_T=(z_T-I_{h,T}z_T)\mathbf{n}_{h}^T{\rm avg}_{\Gamma_{R_T}}(\mathbf{g}).}
\end{equation}

\section{Approximation capabilities of IFE spaces}\label{sec_approximation}
We first introduce some norms,  extensions and interpolation operators.
For all non-negative integers $m$, define broken Sobolev spaces $H^m(\cup D^\pm)=\{v\in L^2(D)  : v|_{D^\pm} \in H^m(D^\pm)\}$ equipped with the norm $\|\cdot\|_{H^m(\cup D^\pm)}$ and the  semi-norm $|\cdot|_{H^m(\cup D^\pm)}$ satisfying 
$
\|\cdot\|^2_{H^m(\cup D^\pm)}=\|\cdot\|^2_{H^m(D^+)}+\|\cdot\|^2_{H^m(D^-)}$ and $|\cdot|^2_{H^m(\cup D^\pm)}=|\cdot|^2_{H^m(D^+)}+|\cdot|^2_{H^m(D^-)}.
$
For convergence analysis, we assume the solution to (\ref{weakform}) belongs to the space
\begin{equation}\label{def_h2h1}
\begin{aligned}
\widetilde{\boldsymbol{H}^2H^1}(\mathbf{g}):=\{ (\mathbf{v}, q) \in (H^2(\cup\Omega^\pm)^N, H^1(\cup\Omega^\pm)) : [\sigma(\mu,\mathbf{v},q)\mathbf{n}]_\Gamma=\mathbf{g},  [\mathbf{v}]_\Gamma=\mathbf{0}, [\nabla\cdot \mathbf{v}]_\Gamma=0 \}.
\end{aligned}
\end{equation}
For any $v\in H^m(\cup\Omega^\pm)$, we let $v^\pm=v|_{\Omega^\pm}$. There exist extensions of $v^\pm$, denoted by $v_E^\pm$, such that (see \cite{Gilbargbook})
\begin{equation}\label{extension}
\|v_E^\pm\|_{H^m(\Omega)}\leq C\|v^\pm\|_{H^m(\Omega^\pm)}.
\end{equation}
For any $v\in P_k(\cup T_h^\pm)$, with a small ambiguity of notation,  we also use  $v^\pm$  to represent the polynomial extensions of $v|_{T_h^\pm}$, i.e.,
$v^\pm\in  P_k(\mathbb{R}^N)$ and $v^\pm|_{T_h^\pm}=v|_{T_h^\pm}.$

Given two functions $v^+$ and $v^-$, we will frequently  use the notation
$
[\![v^\pm]\!](\mathbf{x}):=v^+(\mathbf{x})-v^-(\mathbf{x})
$ in the analysis.
Therefore, for any $v\in H^m(\cup \Omega^\pm)$, we have
$[\![v_E^\pm]\!](\mathbf{x})=v_E^+(\mathbf{x})-v_E^-(\mathbf{x})$ for all $\mathbf{x}\in \Omega,$
which can be viewed as an extension of the jump $[v]_\Gamma$.

The extensions of tangential gradients along  surfaces are defined as follows. Recalling $\mathbf{n}$ is well-defined in $U(\Gamma,\delta_0)$, we define 
$$(\nabla_\Gamma v)(\mathbf{x}):=\nabla v-(\mathbf{n}\cdot \nabla v)\mathbf{n}~\forall \mathbf{x}\in U(\Gamma,\delta_0).$$
Similarly, for any $T\in\mathcal{T}_h^\Gamma$, we define 
$$(\nabla_{\Gamma_{h,T}^{ext}} v)(\mathbf{x}):=\nabla v-(\mathbf{n}_h(\mathbf{x}_T^*) \cdot \nabla v)\mathbf{n}_h(\mathbf{x}_T^*) ~\forall \mathbf{x}\in R_T.$$
By (\ref{ass_Gamma_h}) and the fact $\mathbf{n}\in C^1(\overline{R_T})^N$, we have 
\begin{equation}\label{ineq_nh_BT}
|\mathbf{n}_h(\mathbf{x}_T^*)-\mathbf{n}(\mathbf{x})|\leq |\mathbf{n}_h(\mathbf{x}_T^*)-\mathbf{n}(\mathbf{x}_T^S)|+|\mathbf{n}(\mathbf{x}_T^S)-\mathbf{n}(\mathbf{x})|\leq Ch_T ~~\forall\mathbf{x}\in \overline{R_T}.
\end{equation}
Therefore, by (\ref{ineq_nh_BT}) and (\ref{tac_ine}) there holds
\begin{equation}\label{ineq_sur_sca}
\begin{aligned}
\|\nabla_{\Gamma_{h,T}^{ext}}v-&\nabla_{\Gamma}v\|_{L^2(\Gamma_{R_T})}=\left\|(\mathbf{n}\cdot \nabla v)(\mathbf{n}-\mathbf{n}_h(\mathbf{x}_T^*))+((\mathbf{n}-\mathbf{n}_h(\mathbf{x}_T^*))\cdot \nabla v)\mathbf{n}_h(\mathbf{x}_T^*)\right\|_{L^2(\Gamma_{R_T})}\\
&\leq C\|\mathbf{n}-\mathbf{n}_h(\mathbf{x}_T^*)\|_{L^\infty(R_T)}\|\nabla v\|_{L^2(\Gamma_{R_T})}\leq Ch_T^{1/2}|v|_{H^1(R_T)}+Ch_T^{3/2}|v|_{H^2(R_T)}.
\end{aligned}
\end{equation}
The tangential gradient acting on $\mathbf{v}=(v_1,...,v_N)^T$ is defined by 
$\nabla_{\Gamma} \mathbf{v}=(\nabla_{\Gamma} v_1,...,\nabla_{\Gamma} v_N)^T.$

Let $\widetilde{I}_{h}$ be the Scott-Zhang interpolation mapping from $H^1(\Omega)$ to the standard $P_1$-conforming finite element space associated with $\mathcal{T}_h$ (see \cite{scott1990finite}). The operator acting on vector-valued functions are also denoted by $\widetilde{I}_{h}$ for simplicity.
The restriction on an element $T$ is defined by $\widetilde{I}_{h,T}v=(\widetilde{I}_{h}v)|_T$.
For all $T\in\mathcal{T}_h$, let $\omega_{T}= {\rm interior}(\bigcup\{\overline{T^\prime} : \overline{T^\prime} \cap \overline{T}\not=\emptyset,  T^\prime\in\mathcal{T}_h \})$.
We recall the standard interpolation error estimates (see \cite{scott1990finite,brenner2008mathematical}):
\begin{equation}\label{std_inter_est}
\begin{aligned}
&|v-\widetilde{I}_{h,T}v|_{H^m(T)}\leq Ch_T^{l-m}|v|_{H^l(\omega_{T})},~~0\leq m\leq l\leq 2,\\
&|v-I_{h,T}v|_{H^m(T)}\leq Ch_T^{2-m}|v|_{H^2(T)},~~m=0,1.
\end{aligned}
\end{equation}
For every non-interface $T\in\mathcal{T}_h$ with $T\subset \Omega^s$, $s=+$ or $-$, we let 
$$\Pi_{h,T}(\mathbf{v}, q)=(I_{h,T}\mathbf{v}, \widetilde{I}_{h,T}q_E^s)~~ \mbox{ and } ~~\widetilde{\Pi}_{h,T}(\mathbf{v}, q)=(\widetilde{I}_{h,T}\mathbf{v}_E^s, \widetilde{I}_{h,T}q_E^s).$$
And, on each interface element $T\in\mathcal{T}_h^\Gamma$, we define $\Pi_{h,T}^{\Gamma}$ and $\widetilde{\Pi}_{h,T}^{\Gamma}$, respectively,  by 
\begin{equation}
\begin{aligned}
&(\mathbf{v}_I,q_I):=\Pi_{h,T}^{\Gamma}(\mathbf{v},q)\in \widetilde{\boldsymbol{P}_1P_1}(T),\\
&\mathbf{v}_I(\mathbf{a}_{i,T})=\mathbf{v}(\mathbf{a}_{i,T}),~q_I(\mathbf{a}_{i,T})= (\widetilde{I}_{h}q^s_E)(\mathbf{a}_{i,T}),~~\forall i\in \mathcal{I}_{N+1},
\end{aligned}
\end{equation}
and 
\begin{equation}\label{def_tilde_ih}
\begin{aligned}
&(\widetilde{\mathbf{v}}_I,q_I):=\widetilde{\Pi}_{h,T}^{\Gamma}(\mathbf{v},q)\in \widetilde{\boldsymbol{P}_1P_1}(T),\\
&\widetilde{\mathbf{v}}_I(\mathbf{a}_{i,T})=(\widetilde{I}_{h}\mathbf{v}_E^s)(\mathbf{a}_{i,T}),~q_I(\mathbf{a}_{i,T})=(\widetilde{I}_{h}q^s_E)(\mathbf{a}_{i,T}),~~\forall i\in \mathcal{I}_{N+1},
\end{aligned}
\end{equation}
where $s=+$ if $\mathbf{a}_{i,T}\in\Omega^+$ and $s=-$ if $\mathbf{a}_{i,T}\in\overline{\Omega^-}$. 
The global IFE interpolation operators $\Pi_{h}^{\Gamma}$ and $\widetilde{\Pi}_{h}^{\Gamma}$ are then defined by 
\begin{equation*}
(\Pi_h^{\Gamma}(\mathbf{v},q))|_{T}=
\left\{
\begin{aligned}
&\Pi_{h,T}^{\Gamma}(\mathbf{v},q),\quad&&\mbox{ if }T\in\mathcal{T}_h^\Gamma,\\
&\Pi_{h,T} (\mathbf{v},q),\quad&&\mbox{ if }T\in\mathcal{T}_h^{non}.
\end{aligned}\right.
\quad 
\end{equation*}
and 
\begin{equation*}
(\widetilde{\Pi}_h^{\Gamma}(\mathbf{v},q))|_{T}=
\left\{
\begin{aligned}
&\widetilde{\Pi}_{h,T}^{\Gamma}(\mathbf{v},q),\quad&&\mbox{ if }T\in\mathcal{T}_h^\Gamma,\\
&\widetilde{\Pi}_{h,T} (\mathbf{v},q),\quad&&\mbox{ if }T\in\mathcal{T}_h^{non}.
\end{aligned}\right.
\quad 
\end{equation*}
In our analysis, we  use $\Pi_{\mathbf{v}}^{\Gamma}\mathbf{v}$ and $\Pi_{\mathbf{v},q}^{\Gamma} q$ to denote the velocity and pressure components of  $\Pi_h^{\Gamma}(\mathbf{v},q)$, respectively, i.e.,
$ (\Pi_{\mathbf{v}}^{\Gamma}\mathbf{v}, \Pi_{\mathbf{v},q}^{\Gamma} q)=\Pi_h^{\Gamma}(\mathbf{v},q).$
As noted in Remark~\ref{rema_dep},  $\Pi_{\mathbf{v}}^{\Gamma}\mathbf{v}$ is solely dependent on the velocity $\mathbf{v}$, while $\Pi_{\mathbf{v},q}^{\Gamma} q$ depends on both  the velocity $\mathbf{v}$ and the pressure $q$.
Similarly, $\widetilde{\Pi}_{\mathbf{v}}^{\Gamma}\mathbf{v}$ and $\widetilde{\Pi}_{\mathbf{v},q}^{\Gamma} q$ are used to denote the velocity and pressure components of  $\widetilde{\Pi}_h^{\Gamma}(\mathbf{v},q)$, respectively.
\begin{remark}
$\Pi_h^{\Gamma}(\mathbf{v},q)$ is well defined for all $(\mathbf{v},q)\in (H^2(\cup \Omega^\pm)^N, H^1(\cup \Omega^\pm))$. To establish inf-sup stability, however, we are required to interpolate $(\mathbf{v},q)$ with $\mathbf{v}\in H^1(\Omega)^N$ (see Lemma~\ref{lem_infsup1}).
This is why we introduce the operator $\widetilde{\Pi}_h^{\Gamma}$.
\end{remark}

For any $(\mathbf{v},q) \in (H^2(\Omega)^N, H^1(\Omega))$, we define $\Pi_{h}(\mathbf{v},q)=(I_h\mathbf{v},\widetilde{I}_hq)$, and for any  $(\mathbf{v},q)\in (H^2(\cup\Omega^\pm)^N, H^1(\cup\Omega^\pm)) $, we define $E^{BK}_h$ and $\Pi_h^{BK}$ by
$$(E^{BK}_h(\mathbf{v},q))|_{\Omega_h^\pm}=(\mathbf{v}_E^\pm,q_E^\pm)|_{\Omega_h^\pm}~~\mbox{ and  } ~~
(\Pi_h^{BK}(\mathbf{v},q))|_{\Omega_h^\pm}=(\Pi_{h}(\mathbf{v}_E^\pm,q_E^\pm))|_{\Omega_h^\pm}.$$ 
To handle the surface force, we define
\begin{equation}\label{def_ifej}
\Pi_h^{\Gamma,J}(\mathbf{v},q):=(\Pi_{\mathbf{v}}^{\Gamma,J}\mathbf{v},\Pi_{\mathbf{v},q}^{\Gamma,J}q)=\Pi_h^{\Gamma}(\mathbf{v},q)+(\mathbf{u}_h^J, p_h^J).
\end{equation}
One can easily easy see that for all  $T\in\mathcal{T}_h^\Gamma$,
\begin{equation}\label{bk_dof_eq}
(\Pi_h^{\Gamma,J}(\mathbf{v},q))|_T\in (\boldsymbol{P}_1(\cup T_h^\pm), P_1(\cup T_h^\pm)),~ {\rm DoF}_{i,T}(\Pi_{h}^{BK}(\mathbf{v}, q)={\rm DoF}_{i,T}(\Pi_{h}^{\Gamma,J}(\mathbf{v}, q)).
\end{equation}

\subsection{Interpolation error decomposition}
To prove the optimal approximation capabilities of the IFE space,  it suffices to consider interface elements. For each $T\in\mathcal{T}_h^\Gamma$, 
the triangle inequality gives
\begin{equation}\label{err_pro_it1}
\interleave(\mathbf{u}, p)- \Pi_{h}^{\Gamma,J}(\mathbf{u}, p)\interleave_{m,T}\leq \interleave(\mathbf{u}, p)- E_h^{BK}(\mathbf{u}, p)\interleave_{m,T}+ \interleave E_h^{BK}(\mathbf{u}, p)-\Pi_{h}^{\Gamma,J}(\mathbf{u}, p)\interleave_{m,T},
\end{equation}
where $\interleave\cdot\interleave_{m,T}$ is a norm defined by
$$\interleave(\mathbf{v},q)\interleave_{m,T}^2=|\mathbf{v}|^2_{H^m(T)}+\|q\|^2_{L^2(T)}.$$
It should be noted that in some situations, $\mathbf{v}$ may be discontinuous across $\Gamma_h$, and thus, $\mathbf{v}$ does not belong to $H^m(T)$ with $m=1$. For simplicity, we also use the notation $|\mathbf{v}|_{H^m(T)}$ to represent the piecewise 
$H^m$ norm.
The first term relates to the error caused by the mismatch of $\Gamma$ and $\Gamma_h$, which will be analyzed later. For the second term, we have 
\begin{equation}\label{err_pro_it2}
\interleave E_h^{BK}(\mathbf{u}, p)-\Pi_{h}^{\Gamma,J}(\mathbf{u}, p)\interleave^2_{m,T}\leq \sum_{s=\pm}\|(\mathbf{u}_E, p_E)^s- (\Pi_{h}^{\Gamma,J}(\mathbf{u}, p))^s\|^2_{m,T}.
\end{equation}
Our goal is to estimate $(\mathbf{u}_E^\pm, p_E^\pm)- (\Pi_{h}^{\Gamma,J}(\mathbf{u}, p))^\pm$ on each interface element $T$. Obviously, we have
\begin{equation}\label{int_pro_cha}
\begin{aligned}
(\mathbf{u}_E^\pm, p_E^\pm)- (\Pi_{h}^{\Gamma,J}(\mathbf{u}, p))^\pm
=\underbrace{(\mathbf{u}_E^\pm, p_E^\pm)-\Pi_{h}(\mathbf{u}_E^\pm, p_E^\pm)}_{({\rm I})_1}+\underbrace{\Pi_{h}(\mathbf{u}_E^\pm, p_E^\pm)- (\Pi_{h}^{\Gamma,J}(\mathbf{u}, p))^\pm }_{({\rm I})_2}.
\end{aligned}
\end{equation}
The estimate of Term $({\rm I})_1$ is trivial and  the main task is to estimate $({\rm I})_2$. 
By definition, one can see that the functions in $({\rm I})_2$ are polynomial extensions of $(\Pi_{h}^{BK}(\mathbf{u}, p)-\Pi_{h}^{\Gamma,J}(\mathbf{u}, p))|_{T_h^\pm}$, i.e., 
\begin{equation}\label{term2}
({\rm I})_2=\Pi_{h}(\mathbf{u}_E^\pm, p_E^\pm)- (\Pi_{h}^{\Gamma,J}(\mathbf{u}, p))^\pm=(\Pi_{h}^{BK}(\mathbf{u}, p)-\Pi_{h}^{\Gamma,J}(\mathbf{u}, p))^\pm.
\end{equation}
Our idea is to decompose  $(\Pi_{h}^{BK}(\mathbf{u}, p)-\Pi_{h}^{\Gamma,J}(\mathbf{u}, p))|_T$  by ${\rm DoF}_{i,T}$, $i=1,..., (N+1)^2$ and others related to interface jump conditions. For simplicity of notation, we let $\boldsymbol{\nu}_{i,h}:=\mathbf{t}_{i,h}$, $i=1,...N-1$ and $\boldsymbol{\nu}_{N,h}:=\mathbf{n}_{h}$. Define functionals $\Im_{i,T}$, $i=1,..., (N+1)^2$ for any $(\mathbf{v},q)\in (\boldsymbol{P}_1(\cup T_h^\pm), P_1(\cup T_h^\pm))$ by
\begin{subequations}\label{def_dof_jump0}
\begin{align}
&\Im_{i,T}(\mathbf{v},q)=[\![\mathbf{v}^\pm]\!](\mathbf{x}_T^P)\cdot\boldsymbol{\nu}_{i,h},~&&i=1,..., N,\label{def_dof_jump1}\\
&\Im_{N+(i-1)N+j,T}(\mathbf{v},q)=[\![\nabla\mathbf{v}^\pm\boldsymbol{\nu}_{i,h}]\!]\cdot\boldsymbol{\nu}_{j,h},~&&i=1,..., N-1,~j=1,...,N,\label{def_dof_jump2}\\
&\Im_{N^2+i,T}(\mathbf{v},q)=[\![\sigma(\mu^\pm, \mathbf{v}^\pm ,q^\pm )\mathbf{n}_h]\!](\mathbf{x}_T^*)\cdot\boldsymbol{\nu}_{i,h}, ~&&i=1,..., N,\label{def_dof_jump3}\\
&\Im_{N^2+N+1,T}(\mathbf{v},q)=[\![\nabla \cdot \mathbf{v}^\pm ]\!],&&\label{def_dof_jump4}\\
&\Im_{N^2+N+1+i,T}(\mathbf{v},q)=[\![\nabla q^\pm ]\!]\cdot \boldsymbol{\nu}_{i,h},~&&i=1,..., N.\label{def_dof_jump5}
\end{align}
\end{subequations}
If ${\rm DoF}_{i,T}(\mathbf{v},q)=\Im_{i,T}(\mathbf{v},q)=0$ for all $i=1,..., (N+1)^2$, there holds $(\mathbf{v},q)=(\mathbf{0},0)$ by Lemma~\ref{lem_IFEbasis}. This 
implies that the function in the space $(\boldsymbol{P}_1(\cup T_h^\pm), P_1(\cup T_h^\pm))$ is uniquely determined by ${\rm DoF}_{i,T}$ and $\Im_{i,T}$.
Defining the auxiliary functions $(\boldsymbol{\theta}_{i,T}, \vartheta_{i,T})\in (\boldsymbol{P}_1(\cup T_h^\pm),~ P_1(\cup T_h^\pm))$ by
$${\rm DoF}_{i,T}(\boldsymbol{\theta}_{i,T}, \vartheta_{i,T})=0~\mbox{ and } ~\Im_{j,T}(\boldsymbol{\theta}_{i,T}, \vartheta_{i,T})=\delta_{ij} ~\mbox{ for all } i, j=1,..., (N+1)^2,$$
 and recalling (\ref{bk_dof_eq}), we have the following interpolation error decomposition:
\begin{equation}\label{decomp_int_er}
(\Pi_{h}^{BK}(\mathbf{u}, p)-\Pi_{h}^{\Gamma,J}(\mathbf{u}, p))|_T=\sum_{i}\alpha_{i,T}(\boldsymbol{\theta}_{i,T}, \vartheta_{i,T}),~~ \alpha_{i,T}:=\Im_{i,T}(\Pi_{h}^{BK}(\mathbf{u}, p)-\Pi_{h}^{\Gamma,J}(\mathbf{u}, p)).
\end{equation}

Next, we estimate the constants $\alpha_{i,T}$ and the  auxiliary functions $(\boldsymbol{\theta}_{i,T}, \vartheta_{i,T})$.
\begin{lemma}\label{lem_auxest}
The auxiliary functions $(\boldsymbol{\theta}_{i,T}, \vartheta_{i,T})$ exist uniquely and satisfy  for $m=0,1$, that
\begin{equation*}
|\boldsymbol{\theta}_{i,T}^{\pm}|_{W_\infty^m(T)} \left\{
\begin{aligned}
&\leq Ch_T^{-m},~~ &&i=1,...,N,\\
&\leq Ch_T^{1-m}, &&i=N+1,..., N^2+N+1,\\
&=0,&& i=N^2+N+2,..., (N+1)^2,
\end{aligned}\right. 
\end{equation*}
\begin{equation*}
\|\vartheta_{i,T}^{\pm}\|_{L^\infty(T)} \left\{
\begin{aligned}
&\leq Ch_T^{-1},~~ &&i=1,...,N,\\
&\leq C, &&i=N+1,..., N^2+N+1,\\
&\leq Ch_T,&&i=N^2+N+2,..., (N+1)^2.
\end{aligned}\right.
\end{equation*}
\end{lemma}
\begin{proof}
The proof is in the same spirit  as Lemma~\ref{lem_vj}.
Define $(\mathbf{v}_i, q_i)\in (\boldsymbol{P}_1(\cup T_h^\pm), P_1(\cup T_h^\pm))$ having the same interface jump conditions as $(\boldsymbol{\theta}_{i,T}, \vartheta_{i,T})$ by setting
\begin{equation}\label{pro_vj_aux}
(\mathbf{v}_i^-, q_i^-)=(\mathbf{0}, 0)\mbox{ and }\Im_{j,T}(\mathbf{v}_i,q_i)=\delta_{ij},~i,j=1,..., (N+1)^2.
\end{equation}
Then the auxiliary function can be constructed as
\begin{equation}\label{pro_cos_theta}
(\boldsymbol{\theta}_{i,T}, \vartheta_{i,T})=(\mathbf{v}_i, q_i)-(I_{h,T}\mathbf{v}_i, I_{h,T}q_i)=(\mathbf{v}_i, q_i)-\sum_{j}{\rm DoF}_{j,T}(\mathbf{v}_i, q_i)(\boldsymbol{\phi}_{j,T}^{\Gamma}, \varphi_{j,T}^{\Gamma}).
\end{equation}
Similarly to (\ref{pro_vi3}), using the $\mathbf{t}_{i,h}$-$\mathbf{n}_h$ coordinate system one can deduce from (\ref{pro_vj_aux}) and (\ref{def_dof_jump0})  that the functions $(\mathbf{v}_i^+, q_i^+)$ exist uniquely and satisfy for all $l,m=1,...,N$, that
\begin{equation*}
\begin{aligned}
&|\mathbf{v}_i^+(\mathbf{x}_T^P)|=1,~ |\partial_{\boldsymbol{\nu}_{m,h}} (\mathbf{v}_i^+\cdot \boldsymbol{\nu}_{l,h})|=0,~ q_i^+=0, &&~i=1,...,N,\\
&|\mathbf{v}_i^+(\mathbf{x}_T^P)|=0,~ |\partial_{\boldsymbol{\nu}_{m,h}} (\mathbf{v}_i^+\cdot \boldsymbol{\nu}_{l,h})|=0,1, \mbox{ or } (\mu^+)^{-1}, ~|q_i^+|=0, 1,\mbox{ or }2\mu^+,&&i=N+1,..., N^2+N+1,\\
&\mathbf{v}_i^+=\mathbf{0},~ q_i^+(\mathbf{x}_T^*)=0,~ |\partial_{\boldsymbol{\nu}_{m,h}} q_i^+|=0 \mbox{ or } 1, &&i=N^2+N+2,..., (N+1)^2.
\end{aligned}
\end{equation*}
The lemma follows from the above identities, (\ref{pro_cos_theta}), (\ref{ineq_xpxs}) and Lemma~\ref{lem_basisest}.
\end{proof}

\begin{lemma}\label{lem_alphas}
Let $\alpha_{i,T}$ be defined in (\ref{decomp_int_er}) and assume $(\mathbf{u},p)\in\widetilde{\boldsymbol{H}^2H^1}(\mathbf{g})$, then there holds
\begin{subequations}
\begin{align}
&\alpha_{i,T}^2\leq Ch_T^{4-N}\sum_{s=\pm}\sum_{m=1,2}|\mathbf{u}_E^s|^2_{H^m(T)},~~i=1,...,N,\label{est_alpa1}\\
&\alpha_{i,T}^2\leq Ch_T^{2-N}\sum_{s=\pm}\sum_{m=1,2}|\mathbf{u}_E^s|^2_{H^m(R_T)},~~i=N+1,...,N^2,\label{est_alpa2}\\
\begin{split}
&\alpha_{i,T}^2\leq Ch_T^{2-N}\sum_{s=\pm}\sum_{m=1,2}(|\mathbf{u}_E^s|^2_{H^m(R_T)}+|p_E^s|^2_{H^{m-1}(\omega_T\cup R_T)}),~~~i=N^2+1,...,N^2+N,\\
\end{split}\label{est_alpa3}\\
&\alpha_{i,T}^2\leq Ch_T^{2-N}\sum_{s=\pm}|\mathbf{u}_E^s|^2_{H^2(R_T)},~i=N^2+N+1,\label{est_alpa4}\\
&\alpha_{i,T}^2\leq Ch_T^{-N}\sum_{s=\pm}|p_E^s|^2_{H^1(\omega_T)},~~~~i=N^2+N+2,...,(N+1)^2.\label{est_alpa5}
\end{align}
\end{subequations}
\end{lemma}
\begin{proof}
It follows from (\ref{dis_jp0}) and (\ref{def_dof_jump0})  that 
$\Im_{i,T}(\Pi_{h}^{\Gamma}(\mathbf{u}, p))=0$. 
With this identity and  (\ref{def_ifej}), the constants $\alpha_{i,T}$ in (\ref{decomp_int_er}) becomes $\alpha_{i,T}=\Im_{i,T}(\Pi_{h}^{BK}(\mathbf{u}, p))-\Im_{i,T}(\mathbf{u}_h^J,p_h^J).$  Recalling the definition of $(\mathbf{u}_h^J,p_h^J)$ in subsection~\ref{subsec_corre_fun}, we further have
\begin{equation}\label{alpha_it0}
\alpha_{i,T}=\left\{
\begin{aligned}
&\Im_{i,T}(\Pi_{h}^{BK}(\mathbf{u}, p))-{\rm avg}_{\Gamma_{R_T}}([\![\sigma(\mu^\pm,\mathbf{u}_E^\pm,p_E^\pm)\mathbf{n}]\!])\cdot\boldsymbol{\nu}_{i,h},\quad &&i=N^2+1,...,N^2+N,\\
&\Im_{i,T}(\Pi_{h}^{BK}(\mathbf{u}, p)),&& Others.
\end{aligned}
\right.
\end{equation}
For $i=1,...,N$, by  (\ref{def_dof_jump1}), (\ref{ineq_xpxs}), (\ref{std_inter_est}) and the fact $[\![\mathbf{u}_E^\pm]\!](\mathbf{x}_T^S)=\mathbf{0}$, there holds 
\begin{equation*}
\begin{aligned}
|\alpha_{i,T}|&\leq|[\![I_{h,T} \mathbf{u}_E^\pm]\!](\mathbf{x}_T^P)|\leq  \left|[\![I_{h,T} \mathbf{u}_E^\pm]\!](\mathbf{x}_T^S)\right|+\left|[\![\nabla I_{h,T} \mathbf{u}_E^\pm]\!]\right|\left|\mathbf{x}_T^S-\mathbf{x}_T^P\right|\\
&\leq  \left|[\![I_{h,T} \mathbf{u}_E^\pm]\!](\mathbf{x}_T^S)\right|+Ch_T^2\left|[\![\nabla I_{h,T} \mathbf{u}_E^\pm]\!]\right|\\
&\leq  \left|[\![I_{h,T} \mathbf{u}_E^\pm-\mathbf{u}_E^\pm]\!](\mathbf{x}_T^S)\right|+Ch_T^{2-N/2}\left|([\![ I_{h,T} \mathbf{u}_E^\pm]\!]-[\![\mathbf{u}_E^\pm]\!])+[\![\mathbf{u}_E^\pm]\!]\right|_{H^1(T)}\\
&\leq Ch_T^{2-N/2}\sum_{s=\pm}|\mathbf{u}_E^s|_{H^2(T)}+Ch_T^{2-N/2}\sum_{s=\pm}|\mathbf{u}_E^s|_{H^1(T)},
\end{aligned}
\end{equation*}
which yields the desired result (\ref{est_alpa1}).

In what follows, we use polynomial extensions of $I_{h,T}\mathbf{u}_{E}^{\pm}$ and $\widetilde{I}_{h,T}p_{E}^{\pm}$, also denoted by $I_{h,T}\mathbf{u}_{E}^{\pm}$ and $\widetilde{I}_{h,T}p_{E}^{\pm}$, to carry out the analysis. In other words, $I_{h,T}\mathbf{u}_{E}^{\pm}$ and $\widetilde{I}_{h,T}p_{E}^{\pm}$ are polynomials defined on $\mathbb{R}^N$ and have the same expressions as $I_{h,T}\mathbf{u}_{E}^{\pm}$ and $\widetilde{I}_{h,T}p_{E}^{\pm}$ on $T$. We can prove the following estimates (see Appendix~\ref{sec_app_a}):
\begin{align}
&|v-\widetilde{I}_{h,T}v|_{H^m(R_T)}\leq Ch_T^{l-m}|v|_{H^l(\omega_{T}\cup R_T)},~~0\leq m\leq l\leq 2,\label{res_ap2}\\
&|v-I_{h,T}v|_{H^m(R_T)}\leq Ch_T^{2-m}|v|_{H^2(R_T)},~~~~~~m=0,1.\label{res_ap3}
\end{align}
For $i=N+1,...,N^2$, by  (\ref{alpha_it0}) and  (\ref{def_dof_jump2}), we have 
\begin{equation*}
\begin{aligned}
|\alpha_{i,T}|&\leq\max_{k=1,...,N-1}|[\![ \nabla (I_{h,T} \mathbf{u}_E^\pm )\boldsymbol{\nu}_k]\!]|\leq |[\![ \nabla_{\Gamma_{h,T}^{ext}} (I_{h,T} \mathbf{u}_E^\pm )]\!]|\leq Ch_T^{-\frac{N-1}{2}}\|[\![ \nabla_{\Gamma_{h,T}^{ext}} (I_{h,T} \mathbf{u}_E^\pm )]\!]\|_{L^2(\Gamma_{R_T})}.
\end{aligned}
\end{equation*}
Using the triangle inequality,  the fact $[\![\nabla_{\Gamma}\mathbf{u}_E^\pm]\!]|_{\Gamma}=\mathbf{0}$ and the estimate  (\ref{ineq_sur_sca}), we deduce
\begin{equation*}
\begin{aligned}
\|[\![ \nabla_{\Gamma_{h,T}^{ext}} (I_{h,T} \mathbf{u}_E^\pm )]\!]\|_{L^2(\Gamma_{R_T})}
&\leq \|[\![ \nabla_{\Gamma_{h,T}^{ext}} (I_{h,T} \mathbf{u}_E^\pm-\mathbf{u}_E^\pm )]\!]\|_{L^2(\Gamma_{R_T})}+\|[\![(\nabla_{\Gamma_{h,T}^{ext}}-\nabla_{\Gamma})\mathbf{u}_E^\pm]\!]\|_{L^2(\Gamma_{R_T})}\\
&\leq \|[\![ \nabla (I_{h,T} \mathbf{u}_E^\pm-\mathbf{u}_E^\pm )]\!]\|_{L^2(\Gamma_{R_T})} +C\sum_{s=\pm}\sum_{l=1,2}h_T^{l-1/2}|\mathbf{u}_E^s|_{H^l(R_T)}.
\end{aligned}
\end{equation*}
It follows from  (\ref{tac_ine}) and (\ref{res_ap3}) that 
\begin{equation*}
\begin{aligned}
\|[\![ \nabla (I_{h,T} \mathbf{u}_E^\pm-\mathbf{u}_E^\pm )]\!]\|_{L^2(\Gamma_{R_T})} 
 &\leq C(h_T^{-1/2}\|[\![ \nabla (I_{h,T} \mathbf{u}_E^\pm-\mathbf{u}_E^\pm  ) ]\!]\|_{L^2(R_T)}+h_T^{1/2}|[\![ \nabla \mathbf{u}_E^\pm  ]\!]|_{H^1(R_T)}) \\
&\leq  Ch_T^{1/2}\sum_{s=\pm}|\mathbf{u}_E^s|_{H^2(R_T)}.
\end{aligned}
\end{equation*}
Collecting the above three estimates yields the desired result (\ref{est_alpa2}).

Next, we consider $i=N^2+1,...,N^2+N$. For simplicity of notation, we let  
$
\sigma_{I}^\pm:=\sigma(\mu^\pm, I_{h,T}\mathbf{u}_{E}^{\pm}, \widetilde{I}_{h,T}p_{E}^{\pm})$,
$\sigma^\pm:=\sigma(\mu^\pm,\mathbf{u}_{E}^{\pm}, p_{E}^{\pm})$
and 
$\mathbf{n}_{h,T}:=\mathbf{n}_h(\mathbf{x}_T^*)$. 
By  (\ref{alpha_it0}) and (\ref{def_dof_jump3}),  we get
\begin{equation}\label{pro_al_sig0}
\begin{aligned}
|\alpha_{i,T}|&\leq|[\![\sigma_{I}^\pm\mathbf{n}_h]\!](\mathbf{x}_T^*)-{\rm avg}_{\Gamma_{R_T}}([\![\sigma^\pm\mathbf{n}]\!])|\\
&\leq |[\![\sigma_{I}^\pm\mathbf{n}_h]\!](\mathbf{x}_T^*)-{\rm avg}_{\Gamma_{R_T}}([\![\sigma_{I}^\pm\mathbf{n}_{h,T}]\!])|+|{\rm avg}_{\Gamma_{R_T}}([\![\sigma_{I}^\pm\mathbf{n}_{h,T}-\sigma^\pm\mathbf{n}]\!])|.
\end{aligned}
\end{equation}
Since $I_{h,T}\mathbf{u}_{E}^{\pm}$ and  $\widetilde{I}_{h,T}p_{E}^{\pm}$ are linear, the first term on the right-hand side can be estimated as
\begin{equation}\label{pro_al_sig1}
\begin{aligned}
&|[\![\sigma_{I}^\pm\mathbf{n}_h]\!](\mathbf{x}_T^*)-{\rm avg}_{\Gamma_{R_T}}([\![\sigma_{I}^\pm\mathbf{n}_{h,T}]\!])|= |{\rm avg}_{\Gamma_{R_T}}([\![\nabla (\widetilde{I}_{h,T}p_{E}^{\pm})\cdot (\mathbf{x}-\mathbf{x}_T^*) ]\!])|\\
&\qquad\leq Ch_T|[\![\nabla \widetilde{I}_{h,T}p_{E}^{\pm}]\!]|\leq Ch_T^{1-N/2}\sum_{s=\pm}|\widetilde{I}_{h,T}p_{E}^s|_{H^1(T)}\leq Ch_T^{1-N/2}\sum_{s=\pm}|p_{E}^s|_{H^1(\omega_T)}.
\end{aligned}
\end{equation}
For the second term on the right-hand side of (\ref{pro_al_sig0}), we use (\ref{avr_tac_ine}) to obtain 
\begin{equation*}
\begin{aligned}
&|{\rm avg}_{\Gamma_{R_T}}([\![\sigma_{I}^\pm\mathbf{n}_{h,T}-\sigma^\pm\mathbf{n}]\!])|\leq Ch_T^{-N/2}\sum_{l=0,1}h_T^{l}|[\![\sigma_{I}^\pm\mathbf{n}_{h,T}-\sigma^\pm\mathbf{n}]\!]|_{H^l(R_T)}\\
&\qquad\leq Ch_T^{-N/2}\sum_{l=0,1}h_T^l|[\![(\sigma_{I}^\pm-\sigma^\pm)\mathbf{n}_{h,T}+\sigma^\pm(\mathbf{n}_{h,T}-\mathbf{n})]\!]|_{H^l(R_T)}\\
&\qquad\leq Ch_T^{-N/2}\sum_{s=\pm}\sum_{l=0,1}h_T^l\left(|\sigma_{I}^s-\sigma^s|_{H^l(R_T)}+\|\mathbf{n}_h(\mathbf{x}_T^*)-\mathbf{n}\|_{L^\infty(R_T)}|\sigma^s|_{H^l(R_T)}\right).
\end{aligned}
\end{equation*}
By (\ref{ineq_nh_BT}), (\ref{res_ap2}) and (\ref{res_ap3}), we have
\begin{equation}\label{pro_al_sig2}
|{\rm avg}_{\Gamma_{R_T}}([\![\sigma_{I}^\pm\mathbf{n}_{h,T}-\sigma^\pm\mathbf{n}]\!])|\leq Ch_T^{1-N/2}\sum_{s=\pm}\sum_{m=1,2}(|\mathbf{u}_E^s|_{H^m(R_T)}+|p_E^s|_{H^{m-1}(\omega_T\cup R_T)}).
\end{equation}
Substituting (\ref{pro_al_sig1}) and (\ref{pro_al_sig2}) into (\ref{pro_al_sig0}) yields the desired result (\ref{est_alpa3}). The proof of (\ref{est_alpa4}) and (\ref{est_alpa5}) is analogous. The details are omitted for brevity.
\end{proof}

\subsection{Interpolation error estimates}
With the help of  (\ref{decomp_int_er}) and  Lemmas~\ref{lem_auxest} and \ref{lem_alphas}, we can prove the following interpolation error estimates on interface elements.
\begin{lemma}\label{lem_int_mid}
Suppose $(\mathbf{u},p)\in \widetilde{\boldsymbol{H}^2H^1}(\mathbf{g})$, then there holds for all $T\in\mathcal{T}_h^\Gamma$ that
\begin{equation}\label{int_mid_es}
\begin{aligned}
&|\mathbf{u}_{E}^\pm-(\Pi_{\mathbf{u}}^{\Gamma,J}\mathbf{u})^\pm |^2_{H^m(T)}\leq Ch_T^{4-2m}\mathcal{L}_T(\mathbf{u},p),~~m=0,1,\\
&\|p_{E}^\pm- (\Pi_{\mathbf{u},p}^{\Gamma,J}p)^\pm\|^2_{L^2(T)}\leq Ch_T^{2}\mathcal{L}_T(\mathbf{u},p),
\end{aligned}
\end{equation}
where
$
\mathcal{L}_T(\mathbf{u},p):=\sum_{s=\pm}(\|\mathbf{u}_E^s\|^2_{H^2(R_T)}+\|p_E^s\|^2_{H^1(\omega_T\cup R_T)})$.
\end{lemma}
\begin{proof}
We use (\ref{decomp_int_er}) and Lemmas~\ref{lem_auxest} and \ref{lem_alphas} to get
\begin{equation*}
|I_{h} \mathbf{u}_E^\pm -(\Pi_{\mathbf{u}}^{\Gamma,J}\mathbf{u})^\pm|^2_{H^m(T)}\leq \sum_{i=1}^N C\alpha_{i,T}^2 h_T^{N-2m}+\sum_{i=N+1}^{N^2+N+1} C\alpha_{i,T}^2 h_T^{N+2-2m}\leq Ch_T^{4-2m}\mathcal{L}_T(\mathbf{u},p),
\end{equation*}
\begin{equation*}
\begin{aligned}
\|\widetilde{I}_{h} p_E^\pm -(\Pi_{\mathbf{u},p}^{\Gamma,J}p)^\pm\|^2_{L^2(T)}&\leq \sum_{i=1}^N C\alpha_{i,T}^2 h_T^{N-2}+\sum_{i=N+1}^{N^2+N+1} C\alpha_{i,T}^2 h_T^{N}+\sum_{i=N^2+N+2}^{(N+1)^2} C\alpha_{i,T}^2 h_T^{N+2}\\
&\leq Ch_T^{2}\mathcal{L}_T(\mathbf{u},p).
\end{aligned}
\end{equation*}
The result (\ref{int_mid_es}) follows from (\ref{int_pro_cha}), the triangle inequality, the above inequalities  and  (\ref{std_inter_est}).
\end{proof}

Taking into account the mismatch of $\Gamma$ and $\Gamma_h$, we prove in the following theorem that the IFE space has optimal approximation capabilities.
\begin{theorem}\label{lem_interIFE_up2}
Suppose $(\mathbf{u},p)\in \widetilde{\boldsymbol{H}^2H^1}(\mathbf{g})$, then there exists a positive constant $C$  such that
\begin{equation}\label{int_est_ulti}
\begin{aligned}
\sum_{T\in\mathcal{T}_h}|\mathbf{u}-\Pi_{\mathbf{u}}^{\Gamma,J}\mathbf{u} |^2_{H^m(T)}&\leq Ch^{4-2m}(\|\mathbf{u}\|^2_{H^2(\cup\Omega^\pm)}+\|p\|^2_{H^1(\cup\Omega^\pm)}),~~m=0,1,\\
\|p- \Pi_{\mathbf{u},p}^{\Gamma,J} p\|^2_{L^2(\Omega)}&\leq Ch^{2}(\|\mathbf{u}\|^2_{H^2(\cup\Omega^\pm)}+\|p\|^2_{H^1(\cup\Omega^\pm)}).
\end{aligned}
\end{equation}
\end{theorem}
\begin{proof}
In view of  (\ref{err_pro_it1}), (\ref{err_pro_it2}) and Lemma~\ref{lem_int_mid}, it suffices to consider the first term on the right-hand of (\ref{err_pro_it1}). Recalling  $\Omega^\triangle=(\Omega_h^+\backslash \Omega^+)\cup(\Omega_h^-\backslash \Omega^-)$, it is clear that  
\begin{equation*}
|\mathbf{u}-\mathbf{u}^{BK}|_{H^m(\Omega)}=|[\![\mathbf{u}_E^\pm]\!]|_{H^m(\Omega^\triangle)},\quad \|p-p^{BK}\|_{L^2(\Omega)}=\|[\![p^\pm]\!]\|_{L^2(\Omega^\triangle)}, 
\end{equation*}
where we let $(\mathbf{u}^{BK}, p^{BK}):=E_h^{BK}(\mathbf{u}, p)$ for simplicity.
By (\ref{ass_Gamma_h}), we have $\mbox{dist}(\Gamma,\Gamma_h)\leq Ch^2$, which implies $\Omega^\triangle \subset U(\Gamma,Ch^2)$. 
Therefore, using Lemma~\ref{lem_strip} with $\delta=Ch^2$ and the fact $[\![\mathbf{u}_E^\pm]\!]|_\Gamma=\mathbf{0}$ and the extension stability (\ref{extension}) we get  
\begin{equation}\label{error_Ehbk}
|\mathbf{u}-\mathbf{u}^{BK}|_{H^m(\Omega)}\leq Ch^{2-m}\|\mathbf{u}\|_{H^2(\cup\Omega^\pm)}, \quad  \|p-p^{BK}\|_{L^2(\Omega)}\leq Ch\|p\|_{H^1(\cup\Omega^\pm)}.
\end{equation}
The desired result (\ref{int_est_ulti}) follows from (\ref{error_Ehbk}),  (\ref{err_pro_it1}), (\ref{err_pro_it2}), Lemma~\ref{lem_int_mid},  the standard estimates (\ref{std_inter_est}) on non-interface elements, finite overlapping of the sets $R_T$ and $\omega_T$, and the extension stability (\ref{extension}).
\end{proof}

We also have the optimal interpolation error estimates under the $H^1$-regularity. 
\begin{lemma}\label{lem_h1_inter}
For each $T\in\mathcal{T}_h^\Gamma$, there exists a positive constant $C$ such that
\begin{equation}
|\mathbf{v}-\widetilde{\Pi}_{\mathbf{v}}^{\Gamma}\mathbf{v}|_{H^m(T)}\leq Ch_T^{1-m}|\mathbf{v}|_{H^1(\omega_T)},~~m=0,1,~~\forall \mathbf{v}\in H^1(\Omega)^N.
\end{equation}
\end{lemma}
\begin{proof}
By (\ref{def_tilde_ih}), we have $I_{h}(\widetilde{\Pi}_{\mathbf{v}}^{\Gamma}\mathbf{v})=\widetilde{I}_{h}\mathbf{v}$. Then, using Lemma~\ref{lem_IFEbasis}, we have 
\begin{equation*}
(\widetilde{\Pi}_{\mathbf{v}}^{\Gamma}\mathbf{v})|_T=\widetilde{I}_{h,T}\mathbf{v}+\sum_{i=1}^{N-1}c_i(w_T-I_{h,T}w_T)\mathbf{t}_{i,h}~~\mbox{with}~~c_i=\frac{\mathbf{t}_{i,h}^T\sigma(\mu^-/\mu^+-1,\widetilde{I}_{h,T}\mathbf{v},0)\mathbf{n}_h}{1+(\mu^-/\mu^+-1)\nabla I_{h,T}w_T\cdot \mathbf{n}_h}.
\end{equation*}
From (\ref{fengmu_jie}) and (\ref{def_zw}), it is easy to see 
$
|c_i|\leq C|\nabla \widetilde{I}_{h,T}\mathbf{v}|$ and  $|w_T-I_{h,T}w_T|_{W_\infty^m(T)}\leq Ch^{1-m}_T.
$
Then we have 
\begin{equation*}
|\widetilde{\Pi}_{\mathbf{v}}^{\Gamma}\mathbf{v}-\widetilde{I}_{h,T}\mathbf{v}|_{H^m(T)}\leq Ch^{1-m}_T|\widetilde{I}_{h,T}\mathbf{v}|_{H^1(T)}\leq Ch^{1-m}_T|\mathbf{v}|_{H^1(\omega_T)},
\end{equation*}
where we used (\ref{std_inter_est}) in the last inequality.
The lemma now follows from  the triangle inequality and (\ref{std_inter_est}).
\end{proof}

\section{Analysis of the IFE method}\label{sec_analysis_IFEM}
For all $(\mathbf{v}, q) \in (\boldsymbol{V}, Q)+\widetilde{\boldsymbol{V}Q_{h,0}}(\Omega)$,  we define some norms by
\begin{equation*}
\begin{aligned}
&\|\mathbf{v} \|^2_{1,h}:=\sum_{T\in\mathcal{T}_h} | \mathbf{v}|^2_{H^1(T)}, \quad \|q\|^2_{*,pre}:=\|q\|_{L^2(\Omega)}^2+\sum_{F\in\mathcal{F}_h^\Gamma}h_F\|\{q\}_F\|^2_{L^2(F)},\\
&\| \mathbf{v} \|^2_{*,h}:=\sum_{T\in\mathcal{T}_h} \|\sqrt{2\mu_h} \boldsymbol{\epsilon}(\mathbf{v})\|^2_{L^2(T)}+\sum_{F\in\mathcal{F}_h^\Gamma}(h_F\|\{ 2\mu_h\boldsymbol{\epsilon}(\mathbf{v})\mathbf{n}_F\}_F \|_{L^2(F)}^2+\frac{2+\eta}{h_F}\|[\mathbf{v}]_F \|_{L^2(F)}^2),\\
&\interleave(\mathbf{v}, q)\interleave^2:= \|\mathbf{v} \|^2_{1,h}+\|q\|_{L^2(\Omega)}^2,~~\interleave(\mathbf{v}, q)\interleave^2_*:= \| \mathbf{v} \|^2_{*,h}+\|q\|^2_{*,pre}.
\end{aligned}
\end{equation*}
Using the Korn inequality for piecewise $H^1$ vector functions (see \cite{brenner2004korn}), we immediately  have
\begin{equation}\label{korn_ineq}
\|\mathbf{v}\|_{1,h}\leq C\|\mathbf{v}\|_{*,h}\qquad \forall (\mathbf{v}, q)  \in (\boldsymbol{V}, Q)+\widetilde{\boldsymbol{V}Q_{h,0}}(\Omega).
\end{equation}

\subsection{Norm-equivalence for IFE functions}
We first collect some useful inequalities for the coupled IFE functions. For all $(\mathbf{v}_h, q_h) \in \widetilde{\boldsymbol{V}Q_{h}}(\Omega)$, we have the decomposition $\mathbf{v}_h=\mathbf{v}_L+\mathbf{v}_b$ with $(\mathbf{v}_L,q_h)|_T\in\widetilde{\boldsymbol{P}_1P_1}(T)$ and $\mathbf{v}_b|_T\in \mbox{span}\{b_T\}^N$. Define $I_h^B\mathbf{v}_h=I_h \mathbf{v}_L+\mathbf{v}_b$ and let $\mathcal{T}_h^F$ be the set of all elements in $\mathcal{T}_h$ sharing the face $F$. We have the following lemma whose proof is given in Appendix~\ref{sec_app_b}.
\begin{lemma}\label{lem_trac_IFE}
For any $T\in\mathcal{T}_h^\Gamma$ and $F\in\mathcal{F}_h^\Gamma$,  the following inequalities hold for all $(\mathbf{v}_h, q_h) \in \widetilde{\boldsymbol{V}Q_{h}}(\Omega)$:
\begin{align}
&\|\nabla \mathbf{v}_h\|_{L^2(\partial T)}\leq Ch_T^{-1/2}|\mathbf{v}_h|_{H^1(T)},\label{trac_IFE_inequality1}\\
&|\mathbf{v}_h|_{H^1(T)}\leq Ch_T^{-1}|\mathbf{v}_h|_{L^2(T)},\label{inver_IFE}\\
&|I_{h,T}\mathbf{v}_h|_{H^1(T)}\leq C|\mathbf{v}_h|_{H^1(T)},\label{IFE_inequ_v}\\
&h_T^{1/2}\|\mathbf{v}_h-I_{h}^B\mathbf{v}_h\|_{L^2(\Gamma_{h,T})}+h_T^m|\mathbf{v}_h-I_{h}^B\mathbf{v}_h|_{H^m(T)}\leq Ch_T|\mathbf{v}_h|_{H^1(T)}, ~ m=0,1,\label{IFE_FE_er}\\
&\|[\![q_h^\pm]\!]\|_{L^2(\Gamma_{h,T})}\leq Ch_T^{-1/2}|\mathbf{v}_h|_{H^1(T)},\label{trac_IFE_inequality_qj}\\
&\|q_h\|_{L^2(\partial T)}\leq Ch_T^{-1/2}(\|q_h\|_{L^2(T)}+|\mathbf{v}_h|_{H^1(T)}),\label{trac_IFE_inequality2}\\
&h_F\|[q_h]_F\|^2_{L^2(F)}+h_F^{-1}\| [\mathbf{v}_h]_{F} \|^2_{L^2(F)}\leq C\sum_{T\in\mathcal{T}_h^F}| \mathbf{v}_h|^2_{H^1(T)},\label{ineq_jup_v1}
\end{align} 
where the constants $C$ are independent of the mesh size $h$ and the interface position relative to the mesh, but may depend on the viscosity coefficients $\mu^\pm$.
\end{lemma}

Using (\ref{korn_ineq}), (\ref{trac_IFE_inequality1}), (\ref{trac_IFE_inequality2}) and (\ref{ineq_jup_v1}) we obtain the following norm-equivalence property.
\begin{lemma}
There exist constants $C_0$ and $C_1$ such that for all $(\mathbf{v}_h,q_h)\in\widetilde{\boldsymbol{V}Q_{h,0}}(\Omega)$,
\begin{equation}\label{norm_eq1}
C_1\|\mathbf{v}_h\|_{1,h}\leq \| \mathbf{v}_h \|_{*,h}\leq C_0\|\mathbf{v}_h\|_{1,h}~\mbox{ and }~C_1\interleave(\mathbf{v}_h,q_h)\interleave \leq \interleave(\mathbf{v}_h,q_h)\interleave_*\leq C_0\interleave(\mathbf{v}_h,q_h)\interleave.
\end{equation}
\end{lemma}

\subsection{Boundedness and coercivity}
It follows from the Cauchy-Schwarz inequality that there exist constants $C_b$ and $C_{A}$ such that for all $(\mathbf{v}, q), (\mathbf{w}, r)  \in (\boldsymbol{V}, Q)+\widetilde{\boldsymbol{V}Q_{h}}(\Omega)$,
\begin{equation}\label{bd}
\begin{aligned}
&a_h(\mathbf{v},\mathbf{w})\leq \| \mathbf{v} \|_{*,h}\| \mathbf{w} \|_{*,h},~~b_h(\mathbf{v}, q)\leq C_b\|\mathbf{v}\|_{1,h}\|q\|_{*,pre},\\
&\mathcal{A}_h(\mathbf{v},q;\mathbf{w},r)\leq C_{A}\interleave(\mathbf{v},q)\|_*\|(\mathbf{w},r)\interleave_*,
\end{aligned}
\end{equation}

With the help of  (\ref{korn_ineq}) and (\ref{trac_IFE_inequality1}), we can prove the following coercivity of the bilinear form $a_h(\cdot,\cdot)$. The proof is similar to the discontinuous Galerkin method (see, e.g., \cite{di2011mathematical}) and so is omitted.
\begin{lemma}\label{lem_cor}
There exists a positive constant $C_{a}$  such that 
\begin{equation}\label{inequ_cor}
a_h(\mathbf{v}_h,\mathbf{v}_h)\geq C_{a}\|\mathbf{v}_h\|^2_{1,h}\qquad \forall (\mathbf{v}_h,q_h)\in\widetilde{\boldsymbol{V}Q_{h,0}}(\Omega)
\end{equation}
is true for  $\gamma=-1$ with an arbitrary $\eta\geq 0$ and is true for $\gamma=1$ with a  sufficiently large $\eta$. 
\end{lemma}

\subsection{Inf-sup stability} We first show the inf-sup stability of $b_h(\cdot,\cdot)$ in the following lemma.
\begin{lemma}\label{lem_infsup1}
There exists a positive constant $C_2$ such that for all $(\mathbf{v}_h, q_h)\in\widetilde{\boldsymbol{V}Q_{h,0}}(\Omega)$,
\begin{equation}\label{infsup1}
C_2\|q_h\|_{L^2(\Omega)}\leq \sup_{(\mathbf{v}^\prime_h, q_h^\prime)\in\widetilde{\boldsymbol{V}Q_{h,0}}(\Omega)}\frac{b_h(\mathbf{v}^\prime_h, q_h) }{\|\mathbf{v}^\prime_h\|_{1,h}}+\|\mathbf{v}_h\|_{1,h}.
\end{equation}
\end{lemma}
\begin{proof}
Since $q_h\in Q$,  there is a function $\mathbf{v}^\prime\in \boldsymbol{V} $ satisfying (see \cite[Theorem 6.5]{di2011mathematical})
\begin{equation}\label{pro_inf_ad1}
\nabla \cdot \mathbf{v}^\prime=q_h~~\mbox{ and }~~ \|\mathbf{v}^\prime\|_{H^1(\Omega)}\leq C\|q_h\|_{L^2(\Omega)}.
\end{equation}
Integration by parts gives
\begin{equation}\label{pro_qhinf1}
\begin{aligned}
\|q_h\|_{L^2(\Omega)}^2&=\int_{\Omega}q_h\nabla\cdot \mathbf{v}^\prime=-\sum_{T\in\mathcal{T}_h}\int_T\nabla q_h\cdot \mathbf{v}^\prime+\sum_{F\in\mathcal{F}_h^\Gamma}\int_F[q_h]_F\mathbf{n}_F\cdot\mathbf{v}^\prime-\sum_{T\in\mathcal{T}_h^\Gamma}\int_{\Gamma_{h,T}}[\![q_h^\pm]\!]\mathbf{v}^\prime\cdot\mathbf{n}_h.
\end{aligned}
\end{equation}
Define an operator $\pi_h$ by  
\begin{equation}\label{def_pi_0}
(\pi_h \mathbf{w})|_T\in \mbox{span}\{b_T\}^N, \quad \int_T (\mathbf{w}-\pi_h \mathbf{w})=\mathbf{0} \quad \forall T\in\mathcal{T}_h~~\forall \mathbf{w}\in L^2(\Omega)^N.
\end{equation}
Let $I_h^{new}\mathbf{v}^\prime=\widetilde{\Pi}_{\mathbf{v}^\prime}^\Gamma\mathbf{v}^\prime+\pi_h(\mathbf{v}^\prime-\widetilde{\Pi}_{\mathbf{v}^\prime}^\Gamma\mathbf{v}^\prime)$. From (\ref{def_tilde_ih}) and the property of Scott-Zhang interpolation, we have $(I_h^{new}\mathbf{v}^\prime)|_{\partial \Omega}=\mathbf{0}$. By (\ref{explicit_formula}), we know there exists at least one function $q_I^\prime$ such that $(I_h^{new}\mathbf{v}^\prime, q_I^\prime)\in \widetilde{\boldsymbol{V}Q_{h,0}}(\Omega)$.
Integration by parts again yields 
\begin{equation}\label{pro_bh_eq1}
\begin{aligned}
-b_h&(I_h^{new}\mathbf{v}^\prime,q_h)=\sum_{T\in\mathcal{T}_h}\int_T q_h\nabla\cdot I_h^{new}\mathbf{v}^\prime-\sum_{F\in\mathcal{F}_h^\Gamma}\int_F\{q_h\}_F[I_h^{new}\mathbf{v}^\prime]_F\cdot\mathbf{n}_F\\
&=-\sum_{T\in\mathcal{T}_h}\int_T\nabla q_h\cdot I_h^{new}\mathbf{v}^\prime+\sum_{F\in\mathcal{F}_h^\Gamma}\int_F[q_h]_F\mathbf{n}_F\cdot\{I_h^{new}\mathbf{v}^\prime\}_F-\sum_{T\in\mathcal{T}_h^\Gamma}\int_{\Gamma_{h,T}}[\![q_h^\pm]\!]I_h^{new}\mathbf{v}^\prime \cdot\mathbf{n}_h,
\end{aligned}
\end{equation}
where we used $[I_h^{new}\mathbf{v}^\prime]_F=\mathbf{0},~[q_h]_F=0~\forall F\in\mathcal{F}_h^{non}$ and  $[I_h^{new}\mathbf{v}^\prime]_{\Gamma_{h,T}}=\mathbf{0}~\forall T\in\mathcal{T}_h^\Gamma$.
Combining (\ref{pro_qhinf1}) and (\ref{pro_bh_eq1}) yields 
\begin{equation}\label{pro_i0}
\begin{aligned}
\|q_h\|_{L^2(\Omega)}^2=&-b_h(I_h^{new}\mathbf{v}^\prime,q_h)- \sum_{T\in\mathcal{T}_h}\int_T\nabla q_h\cdot (\mathbf{v}^\prime-I_h^{new}\mathbf{v}^\prime)+\sum_{F\in\mathcal{F}_h^{\Gamma}}\int_F[q_h]_F\{\mathbf{v}^\prime-I_h^{new}\mathbf{v}^\prime\}_F\cdot\mathbf{n}_F\\
&\quad-\sum_{T\in\mathcal{T}_h^\Gamma}\int_{\Gamma_{h,T}}[\![q_h^\pm]\!](\mathbf{v}^\prime-I_h^{new}\mathbf{v}^\prime)\cdot\mathbf{n}_h:=({\rm II})_1+({\rm II})_2+({\rm II})_3+({\rm II})_4.
\end{aligned}
\end{equation}
By the definition of $I_h^{new}$, we have the relation 
\begin{equation}\label{rela_ih}
\mathbf{v}^\prime-I_h^{new}\mathbf{v}^\prime=(\mathbf{v}^\prime-\widetilde{\Pi}_{\mathbf{v}^\prime}^\Gamma\mathbf{v}^\prime)-\pi_h(\mathbf{v}^\prime-\widetilde{\Pi}_{\mathbf{v}^\prime}^\Gamma\mathbf{v}^\prime).
\end{equation}
 We then use (\ref{def_pi_0}) to get $({\rm II})_2=0$ since $(\nabla q_h)|_T\in P_0(T)^N$ from (\ref{dis_jp4}). 
To estimate other terms, we first use (\ref{rela_ih}) to deduce that, for $m=0,1$,
\begin{equation}\label{est_inew1}
\begin{aligned}
|\mathbf{v}^\prime-&I_h^{new}\mathbf{v}^\prime|_{H^m(T)}\leq  |(\mathbf{v}^\prime-\widetilde{\Pi}_{\mathbf{v}^\prime}^\Gamma\mathbf{v}^\prime)|_{H^m(T)}+|\pi_h(\mathbf{v}^\prime-\widetilde{\Pi}_{\mathbf{v}^\prime}^\Gamma\mathbf{v}^\prime)|_{H^m(T)}\\
&\leq  |(\mathbf{v}^\prime-\widetilde{\Pi}_{\mathbf{v}^\prime}^\Gamma\mathbf{v}^\prime)|_{H^m(T)}+Ch_T^{-m}\|\pi_h(\mathbf{v}^\prime-\widetilde{\Pi}_{\mathbf{v}^\prime}^\Gamma\mathbf{v}^\prime)\|_{L^2(T)}\\
&\leq  |(\mathbf{v}^\prime-\widetilde{\Pi}_{\mathbf{v}^\prime}^\Gamma\mathbf{v}^\prime)|_{H^m(T)}+Ch_T^{-m}\|\mathbf{v}^\prime-\widetilde{\Pi}_{\mathbf{v}^\prime}^\Gamma\mathbf{v}^\prime\|_{L^2(T)}\leq Ch^{1-m}|\mathbf{v}^\prime|_{H^1(\omega_T)},
\end{aligned}
\end{equation}
where we used the inverse inequality in the second inequality, the property $\|\pi_h \mathbf{w}\|_{L^2(T)}\leq C\|\mathbf{w}\|_{L^2(T)}$ for all $\mathbf{w}\in L^2(T)$ in the third inequality, and Lemma~\ref{lem_h1_inter} in the last inequality. 
By the triangle inequality, (\ref{est_inew1}) also implies the stability of $I_h^{new}$:
$
\|I_h^{new}\mathbf{v}^\prime\|_{1,h}\leq \|\mathbf{v}^\prime-I_h^{new}\mathbf{v}^\prime\|_{1,h}+|\mathbf{v}^\prime|_{H^1(\Omega)}\leq C|\mathbf{v}^\prime|_{H^1(\Omega)}.
$
Therefore,
\begin{equation}\label{pro_i1}
|({\rm II})_1|\leq C \frac{|b_h(I_h^{new}\mathbf{v}^\prime, q_h)|}{\| I_h^{new}\mathbf{v}^\prime \|_{1,h}} |\mathbf{v}^\prime|_{H^1(\Omega)} \leq C \left( \sup_{(\mathbf{v}^\prime_h, q_h^\prime)\in\widetilde{\boldsymbol{V}Q_{h,0}}(\Omega)}\frac{b_h(\mathbf{v}^\prime_h, q_h) }{\| \mathbf{v}^\prime_h \|_{1,h}}\right) |\mathbf{v}^\prime|_{H^1(\Omega)}.
\end{equation}
As for  Term $({\rm II})_3$, we use (\ref{ineq_jup_v1}), the standard trace inequality and (\ref{est_inew1}) to get 
\begin{equation}\label{pro_i2}
\begin{aligned}
&|({\rm II})_3|^2\leq \left(\sum_{F\in\mathcal{F}_h^\Gamma}h_F\|[q_h]_F\|^2_{L^2(F)} \right)\sum_{F\in\mathcal{F}_h^\Gamma}h_F^{-1} \|\{\mathbf{v}^\prime- I_h^{new}\mathbf{v}^\prime\}_F \|^2_{L^2(F)}\\
&\leq C\|\mathbf{v}_h\|^2_{1,h}\sum_{T\in\mathcal{T}_h^\Gamma}(h_T^{-2}\|\mathbf{v}^\prime-I_h^{new}\mathbf{v}^\prime\|^2_{L^2(T)}+|\mathbf{v}^\prime-I_h^{new}\mathbf{v}^\prime|^2_{H^1(T)})\leq C\|\mathbf{v}_h\|^2_{1,h}|\mathbf{v}^\prime|^2_{H^1(\Omega)}.
\end{aligned}
\end{equation}
For Term $({\rm II})_4$, similarly to (\ref{pro_i2}), we use (\ref{trac_IFE_inequality_qj}) and the trace inequality (\ref{tac_ine}) with $\Gamma_{R_T}$ and $R_T$ replaced by $\Gamma_{h,T}$ and $T$ (see \cite[Lemma 1]{guzman2018inf}) to obtain
$
|({\rm II})_4|\leq C\|\mathbf{v}_h\|_{1,h}|\mathbf{v}^\prime|_{H^1(\Omega)}.
$
Substituting the above estimates of $({\rm II})_1$--$({\rm II})_4$ into (\ref{pro_i0}) and using (\ref{pro_inf_ad1}) completes the proof.
\end{proof}

We are ready to prove the following inf-sup stability of the $\mathcal{A}_h$ form. 
\begin{theorem}\label{lem_inf_sup_ulti}
There exists a positive constant $C_{s}$ such that
\begin{equation}\label{inf_sup_ulti}
C_{s}\interleave(\mathbf{v}_h, q_h)\interleave \leq \sup_{(\mathbf{w}_h, r_h)\in \widetilde{\boldsymbol{V}Q_{h,0}}(\Omega)}\frac{\mathcal{A}_h(\mathbf{v}_h, q_h; \mathbf{w}_h, r_h)}{\interleave(\mathbf{w}_h, r_h)\interleave}\qquad \forall (\mathbf{v}_h, q_h)\in \widetilde{\boldsymbol{V}Q_{h,0}}(\Omega).
\end{equation}
\end{theorem}
\begin{proof}
Since $\widetilde{\boldsymbol{V}Q_{h,0}}(\Omega)$ is a finite-dimensional space, we assume that the supremum in (\ref{infsup1}) is achieved at $(\mathbf{v}_h^*, q_h^*)\in \widetilde{\boldsymbol{V}Q_{h,0}}(\Omega)$ with $\|\mathbf{v}_h^*\|_{1,h}=\|q_h\|_{L^2(\Omega)}$. Then (\ref{infsup1}) becomes 
\begin{equation}\label{infsup2}
C_2\|q_h\|^2_{L^2(\Omega)}\leq b_h(\mathbf{v}_h^*, q_h)  + \|\mathbf{v}_h\|_{1,h}\|q_h\|_{L^2(\Omega)}.
\end{equation}
We note that $q_h^*$ depends on $\mathbf{v}_h^*$ and  is not unique. Next, we describe how to choose  $q_h^*$. We first construct a function $q$ such that $(\mathbf{v}_h^*, q)\in \widetilde{\boldsymbol{V}Q_{h}}(\Omega)$ and $I_hq=0$, and then set $q_h^*=q-\int_\Omega q$ so that $q_h^*\in L_0^2(\Omega)$. 
It is clear that $q|_{T}=0$ for all $T\in\mathcal{T}_h^{non}$. And, on each interface element $T\in\mathcal{T}_h^{\Gamma}$, by (\ref{explicit_formula}) it holds
$q|_{T}=c_N(z_T-I_{h,T}z_T)$ with $c_N=\mathbf{n}_{h}^T\sigma(\mu^--\mu^+,I_{h,T}\mathbf{v}_h^*,0)\mathbf{n}_h$.
Using  (\ref{def_zw}) and (\ref{IFE_inequ_v}) we get
\begin{equation*}
\begin{aligned}
&\sum_{F\in\mathcal{F}_h^\Gamma}h_F\|q\|^2_{L^2(F)}\leq \sum_{T\in\mathcal{T}_h^\Gamma}Ch_T^N |\nabla I_{h,T}\mathbf{v}_h^*|^2 \leq C |I_{h,T}\mathbf{v}_h^*|^2_{H^1(T)}\leq C \|\mathbf{v}_h^*\|^2_{1,h},\\
&\left|\int_\Omega q\right|^2\leq  C\sum_{T\in\mathcal{T}_h^\Gamma}\|q\|^2_{L^2(T)}\leq \sum_{T\in\mathcal{T}_h^\Gamma}Ch_T^N |\nabla I_{h,T}\mathbf{v}_h^*|^2 \leq C \|\mathbf{v}_h^*\|^2_{1,h}.
\end{aligned}
\end{equation*}
We then using the relation $q_h^*=q-\int_\Omega q$ to conclude that there exists a constant $C_3$ such that
\begin{equation}\label{cstar}
\|q_h^*\|_{*,pre}\leq C_3\| \mathbf{v}_h^*\|_{1,h}.
\end{equation}
By (\ref{def_Ah}), (\ref{bd}), (\ref{norm_eq1}), (\ref{cstar}) and the fact $\|\mathbf{v}_h^*\|_{1,h}=\|q_h\|_{L^2(\Omega)}$, inequality (\ref{infsup2}) becomes
\begin{equation*}
\begin{aligned}
C_2\|q_h\|^2_{L^2(\Omega)}&\leq  \mathcal{A}_h(\mathbf{v}_h, q_h; \mathbf{v}_h^*, q_h^*)-a_h(\mathbf{v}_h, \mathbf{v}_h^*)+b_h(\mathbf{v}_h, q_h^*)+\|\mathbf{v}_h\|_{1,h}\|q_h\|_{L^2(\Omega)}\\
&\leq \mathcal{A}_h(\mathbf{v}_h, q_h; \mathbf{v}_h^*, q_h^*) + (C_0^{2}+C_3C_b+1)\| \mathbf{v}_h \|_{1,h} \|q_h\|_{L^2(\Omega)}\\
&\leq \mathcal{A}_h(\mathbf{v}_h, q_h; \mathbf{v}_h^*, q_h^*) +\frac{(C_0^{2}+C_3C_b+1)^2}{2C_2} \| \mathbf{v}_h \|_{1,h}^2+ \frac{C_2}{2}\|q_h\|_{L^2(\Omega)}^2.
\end{aligned}
\end{equation*}
Combining this with
$C_a\|\mathbf{v}_h\|^2_{1,h}\leq \mathcal{A}_h(\mathbf{v}_h, q_h; \mathbf{v}_h, q_h)$, we get
\begin{equation}\label{pro_C2}
C_4 \interleave(\mathbf{v}_h, q_h )\interleave^2\leq \mathcal{A}_h(\mathbf{v}_h, q_h; \mathbf{v}_h + \theta \mathbf{v}_h^*, q_h + \theta q_h^*),
\end{equation}
for a suitable $\theta >0$ and a constant $C_4>0$ depending only on $\theta$, $C_0$, $C_2$, $C_3$, $C_a$ and $C_b$.
The desired result (\ref{inf_sup_ulti}) then follows from 
$
(\mathbf{v}_h + \theta \mathbf{v}_h^*, q_h + \theta q_h^*)=(\mathbf{v}_h, q_h )+\theta (\mathbf{v}_h^*, q_h^*)\in \widetilde{\boldsymbol{V}Q_{h,0}}(\Omega)
$
and
\begin{equation*}
\begin{aligned}
\interleave(\mathbf{v}_h + \theta \mathbf{v}_h^*, &q_h + \theta q_h^*)\interleave\leq\interleave(\mathbf{v}_h , q_h )\interleave+\theta(\|\mathbf{v}_h^*\|_{1,h}+\|q_h^*\|_{L^2(\Omega)})\\
&\leq\interleave(\mathbf{v}_h , q_h )\interleave+\theta(1+C_3)\|\mathbf{v}_h^*\|_{1,h}\leq (1+\theta+\theta C_3)\interleave(\mathbf{v}_h, q_h )\interleave,
\end{aligned}
\end{equation*}
where we have used (\ref{cstar}) and the fact $\|\mathbf{v}_h^*\|_{1,h}=\|q_h\|_{L^2(\Omega)}$ in the derivation.
\end{proof}

\subsection{Consistency}
Recalling that $\mathbf{u}_E^\pm$ and $p_E^\pm$ are extensions of $\mathbf{u}^\pm$ and $p^\pm$, we can extend (\ref{originalpb1}) and (\ref{originalpb2}) from $\Omega^\pm$ to $\Omega$ as
$
\tilde{\mathbf{f}}_E^\pm:=-\nabla\cdot (2\mu^\pm\boldsymbol{\epsilon}(\mathbf{u}_E^\pm)) +\nabla p_E^\pm$ and 
$\tilde{h}_E^\pm:=\nabla\cdot \mathbf{u}_E^\pm.$
Obviously, $\tilde{\mathbf{f}}_E^\pm=\mathbf{f}_E^\pm$ and $\tilde{h}_E^\pm=0$ in $\Omega^\pm$, while $\tilde{\mathbf{f}}_E^\pm\not=\mathbf{f}_E^\pm$ and $\tilde{h}_E^\pm\not=0$ in $\Omega_h^\pm\backslash\Omega^\pm$. For brevity, we let $(\mathbf{u}^{BK}, p^{BK}):=E_h^{BK}(\mathbf{u}, p)$ and  define  $\tilde{\mathbf{f}}^{BK}$ by $\tilde{\mathbf{f}}^{BK}|_{\Omega_h^\pm}= \tilde{\mathbf{f}}_E^\pm|_{\Omega_h^\pm}$ and $\tilde{h}^{BK}$ by $\tilde{h}^{BK}|_{\Omega_h^\pm}= \tilde{h}_E^\pm|_{\Omega_h^\pm}$. Then it holds on the whole domain $\Omega$ that 
\begin{equation*}
-\nabla\cdot (2\mu_h\boldsymbol{\epsilon}(\mathbf{u}^{BK})) +\nabla p^{BK}=\tilde{\mathbf{f}}^{BK},  \quad
\nabla\cdot \mathbf{u}^{BK}=\tilde{h}^{BK},
\end{equation*}
where the  gradient and the divergence (also denoted by $\nabla $ and $\nabla \cdot$ for simplicity) are understood in a piecewise sense since $\mathbf{u}^{BK}$ and $p^{BK}$ are broken across $\Gamma_h$.
For all $(\mathbf{w}_h,r_h) \in \widetilde{\boldsymbol{V}Q_{h,0}}(\Omega)$,  integrating by parts  and using the facts $[\mathbf{u}^{BK}]_F=\mathbf{0}$ and $[p^{BK}]_F=0$, we infer that 
\begin{equation*}
\mathcal{A}_h(\mathbf{u}^{BK}, p^{BK}; \mathbf{w}_h,r_h)=\int_\Omega \tilde{\mathbf{f}}^{BK}\cdot \mathbf{w}_h+ \int_\Omega\tilde{h}^{BK}r_h-\int_{\Gamma_h} [\![\sigma^\pm \mathbf{n}_h]\!]\cdot \mathbf{w}_h,
\end{equation*}
where  the notation $\sigma^\pm:=\sigma(\mu^\pm,\mathbf{u}_{E}^{\pm}, p_{E}^{\pm})$ is used for simplicity. Combining this with (\ref{IFE_method}) and using the facts $\mathbf{f}^{BK}|_{\Omega\backslash\Omega^\triangle}=\tilde{\mathbf{f}}^{BK}|_{\Omega\backslash\Omega^\triangle}$  and $\tilde{h}^{BK}|_{\Omega\backslash\Omega^\triangle}=0$ yields
\begin{equation}\label{consis_eq}
\mathcal{A}_h(\mathbf{u}^{BK}-\mathbf{U}_h, p^{BK}-P_h; \mathbf{w}_h,r_h)=\underbrace{\int_{\Omega^\triangle} (\tilde{\mathbf{f}}^{BK}-\mathbf{f}^{BK})\cdot \mathbf{w}_h}_{({\rm III})_1}+\underbrace{\int_{\Omega^\triangle}\tilde{h}^{BK}r_h}_{({\rm III})_2}-\underbrace{\int_{\Gamma_h} ([\![\sigma^\pm \mathbf{n}_h]\!]-\mathbf{g}_h)\cdot \mathbf{w}_h}_{({\rm III})_3}.
\end{equation}
We now estimate the right-hand side of (\ref{consis_eq}) term by term. 

\emph{Term $({\rm III})_1$}. Using the Cauchy-Schwarz inequality and the triangle inequality we get
\begin{equation*}
|({\rm III})_1|\leq (\|\tilde{\mathbf{f}}^{BK}\|_{L^2(\Omega^\triangle)}+\|\mathbf{f}^{BK}\|_{L^2(\Omega^\triangle)})\|\mathbf{w}_h\|_{L^2(\Omega^\triangle)}.
\end{equation*}
Suppose $(\mathbf{u},p)\in \widetilde{\boldsymbol{H}^2H^1}(\mathbf{g})$, then by the definition of $\tilde{\mathbf{f}}^s_E$ and the extension stability (\ref{extension}) we have
\begin{equation*}
\begin{aligned}
\|\tilde{\mathbf{f}}^{BK}\|_{L^2(\Omega^\triangle)}\leq \sum_{s=\pm}\|\tilde{\mathbf{f}}^s_E\|_{L^2(\Omega)}\leq C(\|\mathbf{u}\|_{H^2(\cup \Omega^\pm)}+\|p\|_{H^1(\cup \Omega^\pm)}).
\end{aligned}
\end{equation*}
Adding and subtracting $I_h^B\mathbf{w}_h$ which belongs to $H_0^1(\Omega)^N$,  we obtain 
\begin{equation*}
\begin{aligned}
\|\mathbf{w}_h\|_{L^2(\Omega^\triangle)}&\leq\|I_h^B\mathbf{w}_h\|_{L^2(\Omega^\triangle)}+\|\mathbf{w}_h-I_h^B\mathbf{w}_h\|_{L^2(\Omega^\triangle)}\\
&\leq Ch|I_h^B\mathbf{w}_h-\mathbf{w}_h+\mathbf{w}_h|_{H^1(\Omega)}+\|\mathbf{w}_h-I_h^B\mathbf{w}_h\|_{L^2(\Omega^\triangle)}\leq Ch\|\mathbf{w}_h\|_{1,h},
\end{aligned}
\end{equation*}
where in the second inequality we used Lemma~\ref{lem_strip}, the fact  $\Omega^\triangle\subset U(\Gamma,Ch^2)$ and Poincar\'e's inequality, and in the last inequality we used the triangle inequality and (\ref{IFE_FE_er}). Collecting the above inequalities and the assumption $\|\mathbf{f}^{BK}\|_{L^2(\Omega^\triangle)}\leq C\|\mathbf{f}\|_{L^2(\Omega)}$ gives the bound 
\begin{equation*}
|({\rm III})_1|\leq Ch(\|\mathbf{u}\|_{H^2(\cup\Omega^\pm)}+\|p\|_{H^1(\cup\Omega^\pm)})\|\mathbf{w}_h\|_{1,h}.
\end{equation*}

\emph{Term $({\rm III})_2$}. Analogously, we have
\begin{equation*}
\begin{aligned}
|({\rm III})_2|\leq \sum_{s=\pm}\|\nabla\cdot \mathbf{u}_E^\pm\|_{L^2(\Omega^\triangle)}\|r_h\|_{L^2(\Omega^\triangle)}\leq Ch\|\mathbf{u}\|_{H^2(\cup\Omega^\pm)}\|r_h\|_{L^2(\Omega)}.
\end{aligned}
\end{equation*}

\emph{Term $({\rm III})_3$}. We need the following lemma whose proof is postponed in Appendix~\ref{sec_app_c}.
\begin{lemma}\label{lem_gah_ineq}
For all $v\in H^1(\Omega)$, there exists a positive constant $C$ such that  
\begin{align}
&\|v\|^2_{L^2(\Gamma_h)}\leq C\|v\|^2_{L^2(\Gamma)}+Ch^2 \|\nabla v\|^2_{L^2(\Omega)}, \label{gah_ieq1}\\
&\|v-v\circ \mathbf{p}_h\|^2_{L^2(\Gamma_h)}\leq Ch^2 \|\nabla v\|^2_{L^2(\Omega)}.\label{gah_ieq2}
\end{align}
\end{lemma}

We observe from the Cauchy-Schwarz inequality and the definition of $\mathbf{g}_h$ that 
\begin{equation*}
|({\rm III})_3|\leq \left\|[\![\sigma^\pm \mathbf{n}_h]\!]-[\![\sigma^\pm \mathbf{n}]\!]\circ\mathbf{p}_h\right\|_{L^2(\Gamma_h)}\| \mathbf{w}_h\|_{L^2(\Gamma_h)}.
\end{equation*}
Using (\ref{ass_Gamma_h}), Lemma~\ref{lem_gah_ineq}, the global trace inequality and (\ref{extension}) we have 
\begin{equation*}
\begin{aligned}
 \left\|[\![\sigma^\pm \mathbf{n}_h]\!]-[\![\sigma^\pm \mathbf{n}]\!]\circ\mathbf{p}_h\right\|_{L^2(\Gamma_h)}&\leq  \left\|[\![\sigma^\pm (\mathbf{n}_h-\mathbf{n})]\!]\right\|_{L^2(\Gamma_h)}+\left\|[\![\sigma^\pm \mathbf{n}]\!]-[\![\sigma^\pm \mathbf{n}]\!]\circ\mathbf{p}_h\right\|_{L^2(\Gamma_h)}\\
 &\leq Ch(\|\mathbf{u}\|_{H^2(\cup\Omega^\pm)}+\|p\|_{H^1(\cup\Omega^\pm)}).
 \end{aligned}
\end{equation*}
On the other hand, the triangle inequality gives 
$
 \| \mathbf{w}_h\|_{L^2(\Gamma_h)}\leq  \| I_h^B\mathbf{w}_h\|_{L^2(\Gamma_h)}+ \| \mathbf{w}_h-I_h^B\mathbf{w}_h\|_{L^2(\Gamma_h)}.
$
Since $I_h^B\mathbf{w}_h\in H_0^1(\Omega)^N$, it follows from Lemma~\ref{lem_gah_ineq}, the global trace inequality,  Poincar\'e's inequality and  (\ref{IFE_FE_er})  that 
\begin{equation*}
 \| I_h^B\mathbf{w}_h\|_{L^2(\Gamma_h)}\leq C\|I_h^B\mathbf{w}_h\|_{H^1(\Omega)}\leq C|I_h^B\mathbf{w}_h|_{H^1(\Omega)}\leq C\|\mathbf{w}_h\|_{1,h}.
\end{equation*}
By (\ref{IFE_FE_er}) again, we have
$
\| \mathbf{w}_h-I_h^B\mathbf{w}_h\|_{L^2(\Gamma_h)}\leq Ch^{1/2}\|\mathbf{w}_h\|_{1,h}.
$
Thus,
$
 \| \mathbf{w}_h\|_{L^2(\Gamma_h)}\leq C\|\mathbf{w}_h\|_{1,h}.
$
 Collecting the above inequalities yields 
$
|({\rm III})_3|\leq Ch(\|\mathbf{u}\|_{H^2(\cup\Omega^\pm)}+\|p\|_{H^1(\cup\Omega^\pm)})\|\mathbf{w}_h\|_{1,h}.
$
We now combine the bounds for Terms $({\rm III})_1$--$({\rm III})_3$ in  (\ref{consis_eq}) to obtain 
\begin{equation}\label{consis_ulti}
|\mathcal{A}_h(\mathbf{u}^{BK}-\mathbf{U}_h, p^{BK}-P_h; \mathbf{w}_h,r_h)|\leq Ch(\|\mathbf{u}\|_{H^2(\cup\Omega^\pm)}+\|p\|_{H^1(\cup\Omega^\pm)})\interleave(\mathbf{w}_h, r_h)\interleave.
\end{equation}

\subsection{A priori error estimates}
We note that the pressure component of $\Pi_{h}^{\Gamma}(\mathbf{u},p)$ and $\Pi_{h}^{\Gamma,J}(\mathbf{u},p)$ might not belong to $L^2_0(\Omega)$. To overcome this difficulty, we define $\Pi_{h,0}^{\Gamma}$ and $\Pi_{h,0}^{\Gamma,J}$ by
\begin{equation}
\Pi_{h,0}^{\Gamma}(\mathbf{u},p)=\Pi_{h}^{\Gamma}(\mathbf{u},p)-(\mathbf{0},c_h) \mbox{ and }  \Pi_{h,0}^{\Gamma,J}(\mathbf{u},p)=\Pi_{h}^{\Gamma}(\mathbf{u},p)+(\mathbf{u}_h^J,p_h^J-c_h-c_h^J),
\end{equation}
where $c_h=\int_{\Omega} \Pi^\Gamma_{\mathbf{u},p}p$ and $c_h^J=\int_{\Omega} p_h^J$. Recalling Remark~\ref{rema_dep}, it is not hard to verify that $\Pi_{h,0}^{\Gamma}(\mathbf{u},p)\in  \widetilde{\boldsymbol{V}Q_{h,0}}(\Omega)$, which will be useful in the proof of Theorem~\ref{theo_ulti_esti} (see (\ref{pro_l02})). Since $p\in L^2_0(\Omega)$, we have
\begin{equation*}
|c_h|=\left|\int_{\Omega} \Pi^\Gamma_{\mathbf{u},p}p+p_h^J-p-p_h^J\right|\leq C\| \Pi^\Gamma_{\mathbf{u},p}p+p_h^J-p\|_{L^2(\Omega)}+|c_h^J|
\end{equation*}
The first term can be bounded by Theorem~\ref{lem_interIFE_up2}. Since $p_h^J=0$ on all non-interface elements,  we use (\ref{ex_uphj}) and (\ref{avr_tac_ine}) to bound the second term
\begin{equation*}
\begin{aligned}
(c_h^J)^2&=\left(\int_{U(\Gamma,h)}p_h^J\right)^2\leq \|p_h^J\|^2_{L^2(U(\Gamma,h))}|U(\Gamma,h)|\leq Ch\sum_{T\in\mathcal{T}_h^\Gamma}\|p_h^J\|^2_{L^2(T)}\\
&\leq Ch\sum_{T\in\mathcal{T}_h^\Gamma}h_T^{N}|{\rm avg}_{\Gamma_{R_T}}(\mathbf{g})|^2 \leq Ch\sum_{T\in\mathcal{T}_h^\Gamma}(\|[\![\sigma^\pm\mathbf{n}]\!]\|^2_{L^2(R_T)}+h_T^2|\|[\![\sigma^\pm\mathbf{n}]\!]|^2_{H^1(R_T)})\\
&\leq Ch^2(\|\mathbf{u}\|^2_{H^2(\cup\Omega^\pm)}+\|p\|^2_{H^1(\cup\Omega^\pm)}),
\end{aligned}
\end{equation*}
where  in the last inequality we used the  finite overlapping of $R_T$, the fact $\bigcup_{T\in\mathcal{T}_h^\Gamma}R_T\subset U(\Gamma, Ch)$,  Lemma~\ref{lem_strip} with $\delta=Ch$, and the extension stability (\ref{extension}).
Collecting the above results yields 
\begin{equation}\label{est_ch}
\interleave(\mathbf{0},c_h)\interleave_{*}\leq C|c_h| \leq Ch(\|\mathbf{u}\|_{H^2(\cup\Omega^\pm)}+\|p\|_{H^1(\cup\Omega^\pm)}).
\end{equation}

The following lemma presents  an interpolation error estimate in terms of the  $\interleave\cdot\interleave_{*}$-norm.
\begin{lemma}\label{ener_app}
Let $(\mathbf{u}, p)$ be the solution of (\ref{weakform}) and suppose $(\mathbf{u},p)\in \widetilde{\boldsymbol{H}^2H^1}(\mathbf{g})$. Then there holds 
\begin{equation*}
\interleave(\mathbf{u}^{BK}, p^{BK})-\Pi_{h,0}^{\Gamma,J}(\mathbf{u},p)\interleave_{*} \leq Ch(\|\mathbf{u}\|_{H^2(\cup\Omega^\pm)}+\|p\|_{H^1(\cup\Omega^\pm)}).
\end{equation*}
\end{lemma}
\begin{proof}
It holds 
$
\interleave(\mathbf{u}^{BK}, p^{BK})-\Pi_{h,0}^{\Gamma,J}(\mathbf{u},p)\interleave_{*} \leq \interleave(\mathbf{u}^{BK}, p^{BK})-\Pi_{h}^{\Gamma,J}(\mathbf{u},p)\interleave_{*}+\interleave(\mathbf{0},c_h)\interleave_{*}.
$
In view of (\ref{est_ch}), it suffices to prove 
\begin{equation}\label{int_pro_01}
\interleave(\mathbf{u}^{BK}, p^{BK})-\Pi_{h}^{\Gamma,J}(\mathbf{u},p)\interleave_{*} \leq Ch(\|\mathbf{u}\|_{H^2(\cup\Omega^\pm)}+\|p\|_{H^1(\cup\Omega^\pm)}).
\end{equation}
First we have
$
\|\{ 2\mu_h\boldsymbol{\epsilon}(\mathbf{u}^{BK}-\Pi_{\mathbf{u}}^{\Gamma,J}\mathbf{u})\mathbf{n}_F\}_F \|_{L^2(F)}^2\leq \sum_{s=\pm}\|\{ 2\mu_h\boldsymbol{\epsilon}(\mathbf{u}_E^s-(\Pi_{\mathbf{u}}^{\Gamma,J}\mathbf{u})^s)\mathbf{n}_F\}_F \|_{L^2(F)}^2.
$
We then use the standard trace inequality and Lemma~\ref{lem_int_mid} to get
\begin{equation*}
\begin{aligned}
h_F\|\{ 2\mu_h\boldsymbol{\epsilon}(\mathbf{u}^{BK}-\Pi_{\mathbf{u}}^{\Gamma,J}\mathbf{u})\mathbf{n}_F\}_F \|_{L^2(F)}^2 &\leq C\sum_{s=\pm}\sum_{T\in\mathcal{T}_h^F}\left(|\mathbf{u}_E^s-(\Pi_{\mathbf{u}}^{\Gamma,J}\mathbf{u})^s|^2_{H^1(T)}+h_T^2|\mathbf{u}_E^s|^2_{H^2(T)}\right)\\
&\leq Ch^2\sum_{T\in\mathcal{T}_h^F}\sum_{s=\pm}(\|\mathbf{u}_E^s\|^2_{H^2(R_T)}+\|p_E^s\|^2_{H^1(\omega_T\cup R_T)}).
\end{aligned}
\end{equation*}
Other terms in the  $\interleave\cdot\interleave_{*}$-norm can be estimated similarly and the details are omitted.
The desired result (\ref{int_pro_01}) then follows from the finite overlapping of the sets $R_T\cup \omega_T$ and the extension stability (\ref{extension}).
 \end{proof}

With these preparations, we  are now ready to derive the error estimate of the proposed IFE method.
\begin{theorem}\label{theo_ulti_esti}
Let $(\mathbf{u}, p)$ and $(\mathbf{U}_h, P_h)$ be the solutions of (\ref{weakform}) and (\ref{IFE_method}), respectively. Suppose $(\mathbf{u}, p)\in  \widetilde{\boldsymbol{H}^2H^1}(\mathbf{g})$, then the following error estimate holds true:
\begin{equation}\label{ulti_esti}
\interleave (\mathbf{u}^{BK}, p^{BK})-(\mathbf{U}_h, P_h) \interleave_*\leq Ch(\|\mathbf{u}\|_{H^2(\cup \Omega^\pm)}+\|p\|_{H^1(\cup\Omega^\pm)}).
\end{equation}
\end{theorem}
\begin{proof}
The triangle inequality gives
\begin{equation*}
\interleave (\mathbf{u}^{BK}, p^{BK})-(\mathbf{U}_h, P_h) \interleave_*\leq \interleave (\mathbf{u}^{BK}, p^{BK})-\Pi_{h,0}^{\Gamma,J}(\mathbf{u},p) \interleave_*+\interleave \Pi_{h,0}^{\Gamma,J}(\mathbf{u},p)-(\mathbf{U}_h, P_h) \interleave_*.
\end{equation*}
By definition, it is true that
\begin{equation}\label{pro_l02}
\Pi_{h,0}^{\Gamma,J}(\mathbf{u},p)-(\mathbf{U}_h, P_h)=\Pi_{h,0}^{\Gamma}(\mathbf{u},p)-(\mathbf{u}_h, p_h)\in \widetilde{\boldsymbol{V}Q_{h,0}}(\Omega).
\end{equation}

Now we can use  (\ref{norm_eq1}) and (\ref{inf_sup_ulti}) to get
\begin{equation*}
\interleave \Pi_{h,0}^{\Gamma,J}(\mathbf{u},p)-(\mathbf{U}_h, P_h) \interleave_* \leq  C_0C_s^{-1}\sup_{(\mathbf{w}_h, r_h)\in \widetilde{\boldsymbol{V}Q_{h,0}}(\Omega)}\frac{\mathcal{A}_h(\Pi_{h,0}^{\Gamma,J}(\mathbf{u},p)-(\mathbf{U}_h, P_h) ; \mathbf{w}_h, r_h)}{\interleave(\mathbf{w}_h, r_h)\interleave}.
\end{equation*}
Adding and subtracting $(\mathbf{u}^{BK}, p^{BK})$ and using (\ref{consis_ulti}), (\ref{bd}) and (\ref{norm_eq1}), we deduce that 
\begin{equation*}
\interleave \Pi_{h,0}^{\Gamma,J}(\mathbf{u},p)-(\mathbf{U}_h, P_h) \interleave_* \leq Ch(\|\mathbf{u}\|_{H^2(\cup\Omega^\pm)}+\|p\|_{H^1(\cup\Omega^\pm)})+C\interleave (\mathbf{u}^{BK}, p^{BK})-\Pi_{h,0}^{\Gamma,J}(\mathbf{u},p) \interleave_*
\end{equation*}
The desired result (\ref{ulti_esti}) now follows from the above inequalities and Lemma~\ref{ener_app}.
\end{proof}

\begin{remark}
Using (\ref{error_Ehbk}) and (\ref{ulti_esti}), we can  establish  the following error estimate:
\begin{equation*}
\|\mathbf{u}-\mathbf{U}_h\|_{1,h}+\|p-P_h\|_{L^2(\Omega)}\leq Ch(\|\mathbf{u}\|_{H^2(\cup \Omega^\pm)}+\|p\|_{H^1(\cup\Omega^\pm)}).
\end{equation*}

\end{remark}
\subsection{Condition number analysis}
In this subsection we show that the condition number of the stiffness matrix of our IFE method has the same bound as that of conventional mini element methods with the hidden constant independent of the interface. 
We assume the family of triangulations is also quasi-uniform. Define $\interleave (\mathbf{v}_h, q_h) \interleave_0^2=\|\mathbf{v}_h\|^2_{L^2(\Omega)}+\|q_h\|^2_{L^2(\Omega)}$ for all $(\mathbf{v}_h, q_h) \in \widetilde{\boldsymbol{V}Q_{0,h}}(\Omega)$. By (\ref{inver_IFE}), we have the inverse estimate 
\begin{equation}\label{con_pro_inver}
\interleave (\mathbf{v}_h, q_h) \interleave \leq Ch^{-1}\interleave (\mathbf{v}_h, q_h) \interleave_0 \quad \forall (\mathbf{v}_h, q_h) \in \widetilde{\boldsymbol{V}Q_{0,h}}(\Omega).
\end{equation}
Using (\ref{ineq_jup_v1}) and the Poincar\'e-Friedrichs  inequalities for broken functions (see \cite[Chapter 10.6]{brenner2008mathematical}), we have 
the following Poincar\'e type inequality:
\begin{equation}\label{con_pro_Poin}
\interleave (\mathbf{v}_h, q_h) \interleave_0\leq C\interleave (\mathbf{v}_h, q_h) \interleave\quad \forall (\mathbf{v}_h, q_h) \in \widetilde{\boldsymbol{V}Q_{0,h}}(\Omega).
\end{equation}
By Lemma~\ref{lem_IFEbasis} and (\ref{pro_ap_ber01}), we have on each $T\in\mathcal{T}_h^\Gamma$ that
\begin{equation}\label{pro_iieq1}
\|\mathbf{v}_h-I_h^B\mathbf{v}_h\|^2_{L^2(T)}+\|q_h-I_hq_h\|^2_{L^2(T)} \leq Ch_T^2|I_{h,T}\mathbf{v}_h|^2_{H^1(T)}\leq \|\mathbf{v}_h\|^2_{L^2(T)}.
\end{equation}
Similarly to the second inequality in (\ref{pro_app_ber02}), it holds $|I_{h,T}\mathbf{v}_h|^2_{H^1(T)}\leq |I_{h}^B\mathbf{v}_h|^2_{H^1(T)}$. Therefore, we have
\begin{equation}\label{pro_iieq2}
\|\mathbf{v}_h-I_h^B\mathbf{v}_h\|^2_{L^2(T)}+\|q_h-I_hq_h\|^2_{L^2(T)}  \leq Ch_T^2 |I_h^B\mathbf{v}_h|^2_{H^1(T)}\leq   C\|I_h^B\mathbf{v}_h\|^2_{L^2(T)},
\end{equation}
where we used the standard inverse inequality in the last inequality. Using (\ref{pro_iieq1}) and (\ref{pro_iieq2}) we obtain
\begin{equation}\label{con_equi01}
c\interleave (\mathbf{v}_h, q_h) \interleave_0\leq \interleave (I_h^B\mathbf{v}_h,I_hq_h) \interleave_0\leq C\interleave (\mathbf{v}_h, q_h) \interleave_0,
\end{equation}
where $C$ and $c$ are positive constants independent of $h$ and the interface. 
Let $X$ be the vector such that $(\mathbf{v}_h, q_h)=\sum_i X(i) \Psi_i$ with $\{\Psi_i\}$ being the standard finite element basis in $\widetilde{\boldsymbol{V}Q_{0,h}}(\Omega)$.
Since  $(I_h^B\mathbf{v}_h,I_hq_h)$ belongs to the conventional mini finite element
space, we have the following standard result:
$$
ch^{-N}\interleave (I_h^B\mathbf{v}_h,I_hq_h)  \interleave^2_0\leq |X|^2 \leq Ch^{-N}\interleave (I_h^B\mathbf{v}_h,I_hq_h)  \interleave^2_0,
$$
where $C$ and $c$ are positive constants independent of $h$ and the interface. Then, by  (\ref{con_equi01}) we get
\begin{equation}\label{con_ineq_max}
ch^{-N}\interleave (\mathbf{v}_h, q_h)  \interleave^2_0\leq |X|^2 \leq Ch^{-N}\interleave (\mathbf{v}_h, q_h)  \interleave^2_0.
\end{equation}

Combining (\ref{con_ineq_max}), the inverse estimate (\ref{con_pro_inver}) and the Poincar\'e type inequality (\ref{con_pro_Poin}), we can use the approach in \cite[p. 109]{hansbo2014cut} to derive the bound:
$
\kappa(A_h)\leq Ch^{-2},
$
where $A_h$ is the stiffness matrix associated with $\mathcal{A}_h$ and $\kappa(A_h)$ is the the spectral condition number of $A_h$.
\section{Numerical results}\label{sec_num}
\subsection{2D numerical examples}
We first test two 2D numerical examples from \cite{caceres2020new}. Let $\Omega=(-1,1)^2$, $\Omega^-=\{\mathbf{x}\in\mathbb{R}^2 :|\mathbf{x}|< R\}$, $\Omega^+=\Omega\backslash \Omega^-$ and $\Gamma=\{\mathbf{x}\in\mathbb{R}^2 :|\mathbf{x}|=R\}$ with $R=1/\sqrt{\pi}$. We test our IFE method with $\gamma=-1$ and $\eta=0$ on uniform Cartesian meshes consisting of $2M\times M$ congruent triangles. The errors are denoted by 
$e_0(\mathbf{u}):= \|\mathbf{u}^{BK}-\mathbf{U}_h\|_{L^2(\Omega)}$,  $e_1(\mathbf{u}):= \|\mathbf{u}^{BK}-\mathbf{U}_h\|_{1,h}$ and $e_0(p):=\|p^{BK}-P_h\|_{L^2(\Omega)}.$
\textbf{Example 1} ($\mathbf{g}=\mathbf{0}$). The exact solutions are chosen as, for all $\mathbf{x}:=(x_1,x_2)^T\in\Omega$,
\begin{equation*}
\mathbf{u}^\pm(\mathbf{x})=\frac{r_0^2-|\mathbf{x}|^2}{\mu^\pm}\left(
\begin{array}{c}
-x_2\\
x_1\\
\end{array}
\right),
\quad p(\mathbf{x})=x_2^2-x_1^2.
\end{equation*}
It is easy to verify $\mathbf{g}:=[\sigma(\mu,\mathbf{u},p)\mathbf{n}]_\Gamma=\mathbf{0}$.
We consider three cases: $\mu^+=5, \mu^-=1$; $\mu^+=1000, \mu^-=1$; $\mu^+=1, \mu^-=1000$. The numerical results presented in Figure~\ref{ex1_fig} showing the optimal convergence of the proposed IFE method.
\begin{figure} [htbp]
\centering
\subfigure[$\mu^+=5, \mu^-=1$ ]{
\includegraphics[width=0.3\textwidth]{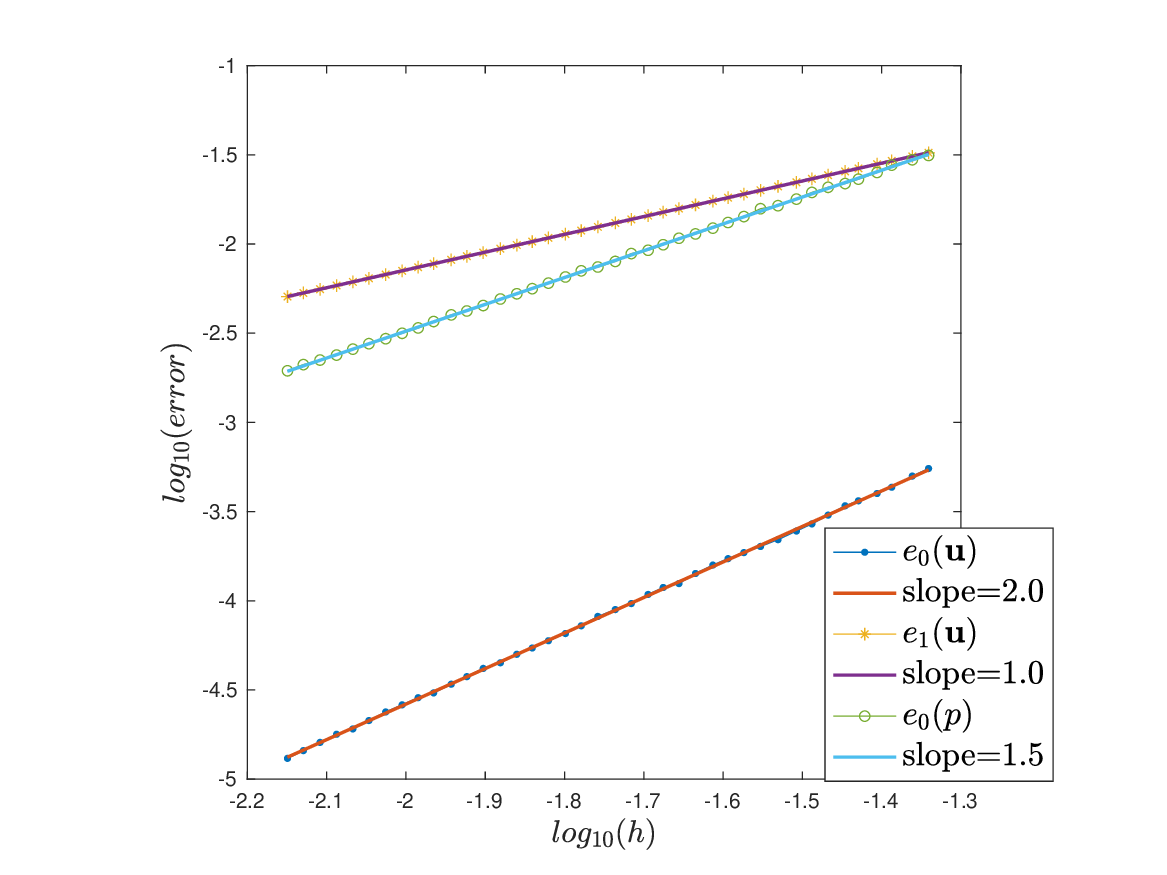}}
\subfigure[$\mu^+=1000, \mu^-=1$]{
\includegraphics[width=0.3\textwidth]{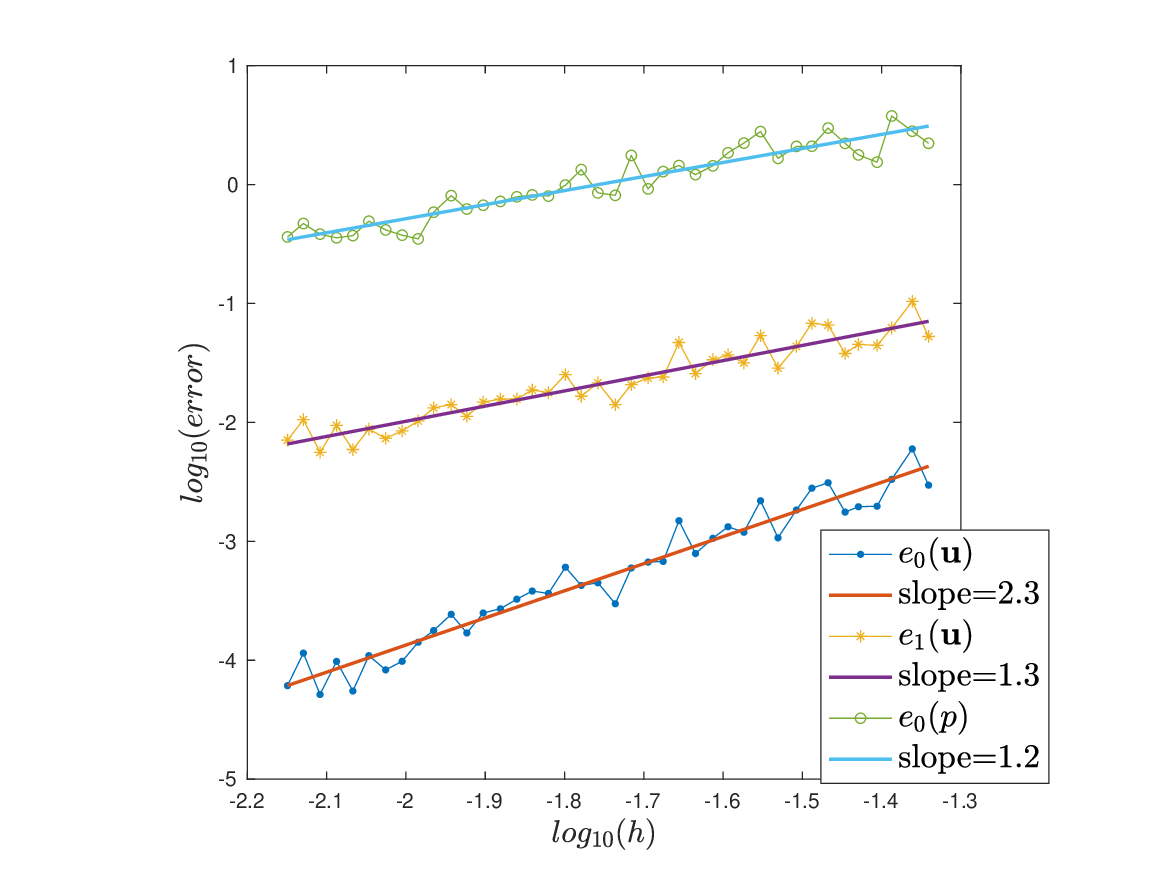}}
\subfigure[$\mu^+=1, \mu^-=1000$]{
\includegraphics[width=0.3\textwidth]{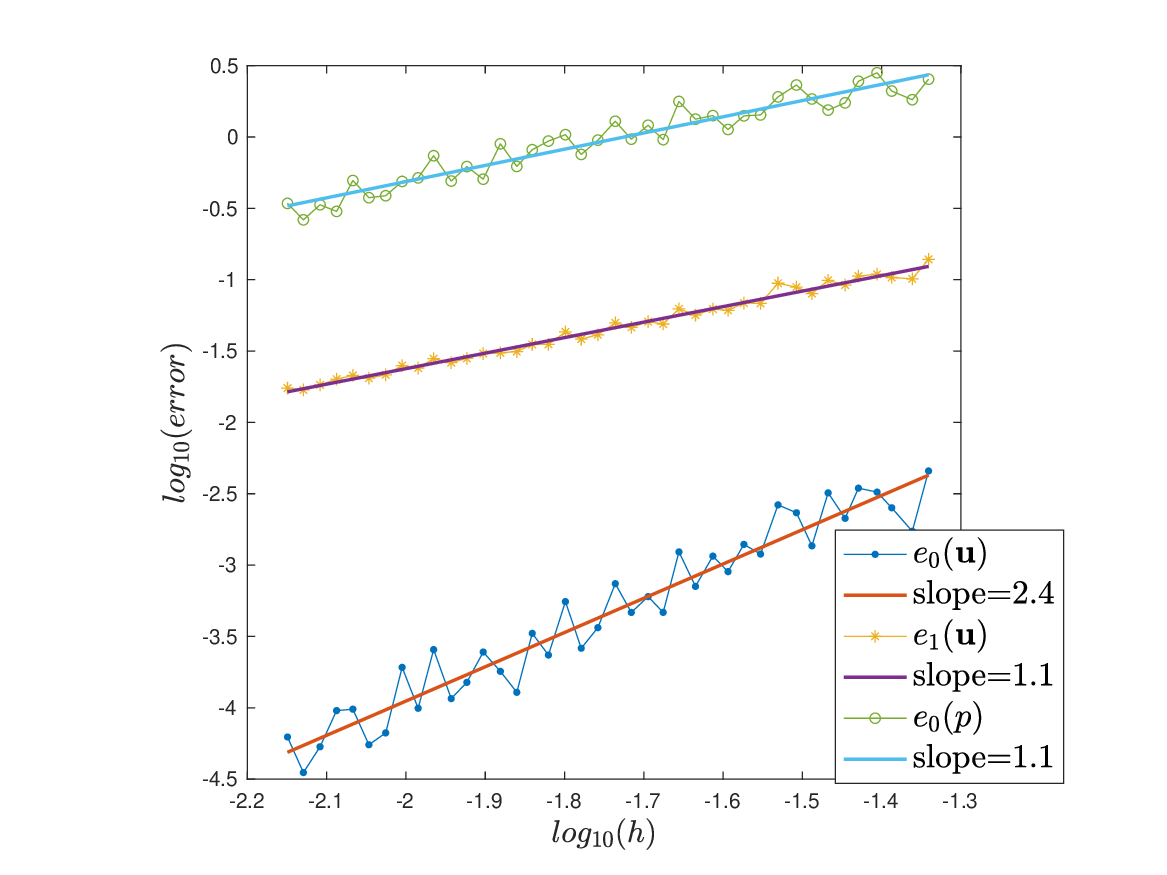}}
 \caption{Plots of $\log_{10}(error)$ versus $\log_{10} (h)$ for Example 1. The linear regression analysis is used to find an approximate order of convergence.\label{ex1_fig}} 
\end{figure}

\textbf{Example 2 }($\mathbf{g}\not=\mathbf{0}$). We consider the parameters $\mu^+=2, \mu^-=0.5$ and the exact solutions  given by
\begin{equation*}
\mathbf{u}(\mathbf{x})=\frac{1}{\pi}
\left(\begin{array}{c}
\sin \pi x_1 \sin \pi x_2\\
\cos \pi x_1 \cos\pi x_2\\
\end{array}
\right),
~~ p^+(\mathbf{x})=-\frac{1}{6\pi},~~p^-(\mathbf{x})=x_1^2+x_2^2.
\end{equation*}
In this example  $\mathbf{g}:=[\sigma(\mu,\mathbf{u},p)\mathbf{n}]_\Gamma\not=\mathbf{0}$ and $[p]_\Gamma\not=0$. We report the results in Table~\ref{ex2_biao} which again show the optimal convergence.
\begin{table}[H]
\caption{Errors and orders for Example 2.\label{ex2_biao}}
\begin{center}
{\small
\begin{tabular}{c|c c|c c|c c}
  \hline
       $M$  &  $e_0(\mathbf{u})$  &  rate   &  $e_1(\mathbf{u})$  &  rate  &  $e_0(p)$  &  rate  \\ \hline
        16   &   2.164E-02    &         &    3.903E-01    &        &    7.983E-01     &          \\ \hline
        32    &  5.280E-03   &   2.04   &   1.903E-01    &  1.04   &   2.213E-01     & 1.85 \\ \hline
        64    &  1.320E-03    &  2.00    &  9.442E-02    &  1.01   &   6.647E-02     & 1.74 \\ \hline
       128   &   3.258E-04   &   2.02    &  4.701E-02   &   1.01   &   2.106E-02   &   1.66 \\ \hline
       256    &  8.170E-05   &   2.00     & 2.347E-02  &    1.00    &  6.989E-03  &    1.59 \\ \hline
\end{tabular}
}
\end{center}
\end{table}
\subsection{A 3D numerical example}
\textbf{Example 3}. This example is taken from \cite{kirchhart2016analysis}. We consider the domain $\Omega=(-1,1)^3$, the parameters $\mu^+=2$, $\mu^-=0.5$, $R=2/3$, and  the exact solution given by
\begin{equation*}
\begin{aligned}
&\mathbf{u}(\mathbf{x})=\alpha(r)e^{-r^2}
\left(\begin{array}{c}
-x_2\\
x_1\\
0
\end{array}
\right),~~
 \alpha(r)=\left\{
\begin{array}{ll}
1/\mu^-&r<R,\\
\mu^++(1/\mu^--1/\mu^+)e^{r^2-R^2}&r\geq R,
\end{array}\right. \\
&p^+(\mathbf{x})=x_1^3-c,~~p^-(\mathbf{x})=x_1^3+10-c,
\end{aligned}
\end{equation*}
where $\mathbf{x}=(x_1,x_2,x_3)^T$ and $r=\sqrt{x_1^2+x_2^2+x_3^2}$. We set $c=5\pi R^3/3$ so that $\int_{\Omega}p=0$.  

In this example, we first test the mini IFE method  ($\gamma=-1$, $\eta=0$) and the conventional mini element method using uniform Cartesian meshes, which are obtained by first partitioning $\Omega$ into $M^3$ cubes and then subdividing each cube into six tetrahedra. The number of the elements is denoted by $\#Elem$.
Numerical results reported in Table~\ref{ex3_biao} clearly show optimal convergence for our IFE method and suboptimal convergence for the conventional mini element method.  
In Figure~\ref{ex3_fig}, we also show plots of the approximate pressures obtained by these two numerical methods on the mesh with $M=32$ . We observe that the IFE method reduces the non-physical oscillations near the interface substantially.
Then, we test the conventional mini element method on a sequence of interface-fitted meshes, obtained using Gmsh \cite{geuzaine2009gmsh} with $h=0.7, 0.35, 0.175$. Numerical results, as shown in Table~\ref{ex32_biao}, indicate suboptimal convergence rates for both velocity and pressure. Note that in this example, $\mathbf{g}\not=\mathbf{0}$, and thus, $[p]_\Gamma\not=0$.  Although interface-fitted meshes are used,  the conventional mini element method cannot achieve  optimal convergence.

\begin{table}[H]
\caption{Errors and orders obtained by two finite element methods using unfitted meshes for Example 3.\label{ex3_biao}}
\begin{center}
{\small
\begin{tabular}{|c|c c|c c|c c|}
  \hline
          \multicolumn{7}{|c|}{Conventional Mini Element Using Unfitted Meshes}  \\ \hline
       $M~~(\#Elem)$  &  $e_0(\mathbf{u})$  &  rate   &  $e_1(\mathbf{u})$  &  rate  &  $e_0(p)$  &  rate  \\ \hline
         4~(384)    &  1.309E-01    &             & 1.043E+00    &           &    5.768E+00  &            \\ \hline
         8~(3072)    & 6.405E-02    &  1.03    &  8.728E-01    &  0.26   &   3.837E+00   &   0.59 \\ \hline
        16~(24576)  &  2.834E-02   &   1.18    &  7.275E-01   &   0.26  &    2.981E+00  &    0.36 \\ \hline
        32~(196608)  &  1.260E-02  &    1.17    &  5.176E-01   &   0.49   &   2.012E+00   &   0.57 \\ \hline
                \multicolumn{7}{|c|}{Mini IFE }  \\ \hline
                  $M~~(\#Elem)$  &  $e_0(\mathbf{u})$  &  rate   &  $e_1(\mathbf{u})$  &  rate  &  $e_0(p)$  &  rate  \\ \hline
           4~(384)    &   3.911E-01 &               & 4.083E+00   &              & 1.475E+01    &            \\ \hline
         8~(3072)   &  7.283E-02   &   2.42    &  8.829E-01   &   2.21     & 2.107E+00   &   2.81 \\ \hline
        16~(24576)   &  1.628E-02   &   2.16  &    2.992E-01  &    1.56    &  4.003E-01   &   2.40 \\ \hline
        32~(196608)    &  4.216E-03    &  1.95 &     1.256E-01   &   1.25   &   1.143E-01   &   1.81 \\ \hline
\end{tabular}
}
\end{center}
\end{table}

\begin{table}[H]
\caption{Errors and orders obtained by conventional mini element  methods using fitted meshes for Example 3.\label{ex32_biao}}
\begin{center}
{\small
\begin{tabular}{|c|c c|c c|c c|}
  \hline
        \multicolumn{7}{|c|}{Conventional Mini Element Using Fitted Meshes}  \\ \hline
       $h ~~(\#Elem)$  &  $e_0(\mathbf{u})$  &  rate   &  $e_1(\mathbf{u})$  &  rate  &  $e_0(p)$  &  rate  \\ \hline
      0.7 ~(3495)    &   6.251E-02   &          &   1.111E+00  &           &  6.276E+00   &       \\ \hline
     0.35~ (23356)    &  2.369E-02    &  1.40   &   9.203E-01   &   0.27    &  4.142E+00    &  0.60\\ \hline
    0.175 ~(164479)    &  8.762E-03   &   1.43  &    6.822E-01  &    0.43   &   2.956E+00    &  0.49\\ \hline
\end{tabular}
}
\end{center}
\end{table}

\begin{figure} [htbp]
\centering
\subfigure{
\includegraphics[width=0.4\textwidth]{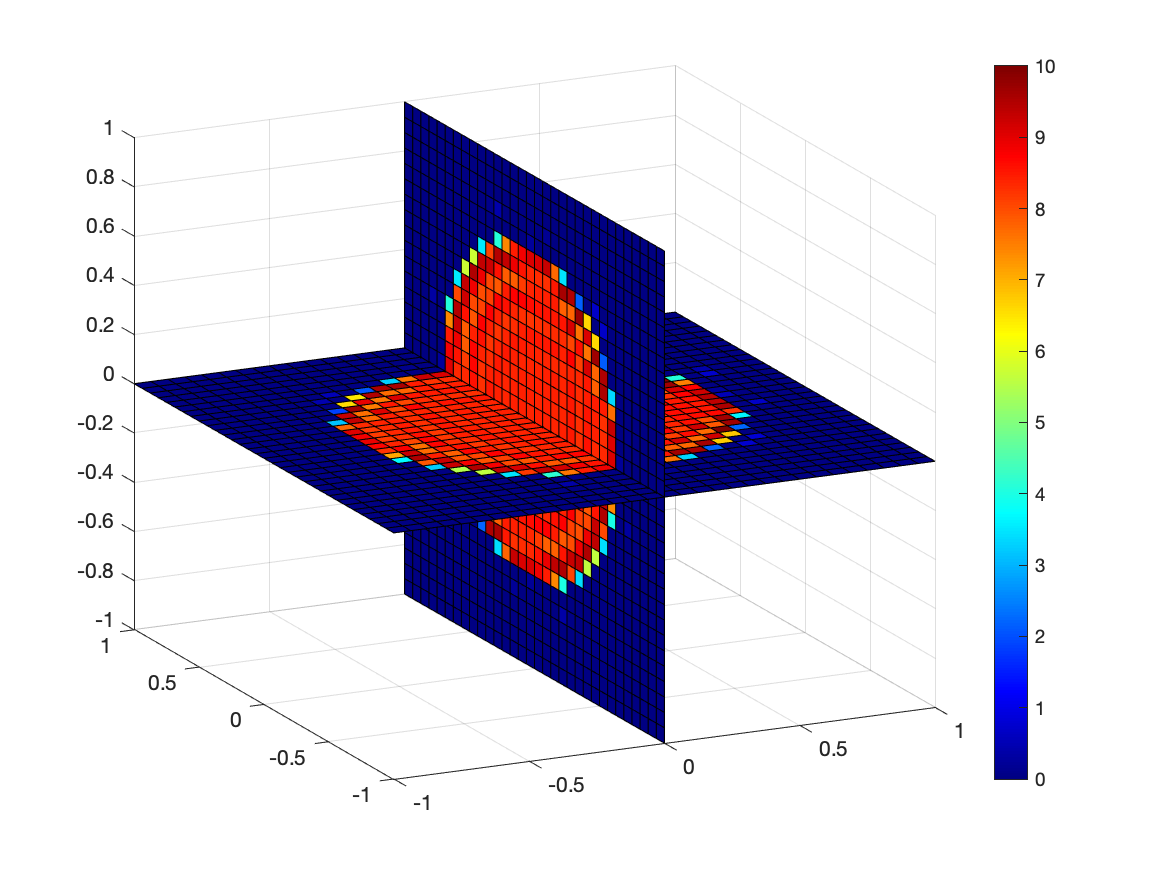}}
\subfigure{
\includegraphics[width=0.4\textwidth]{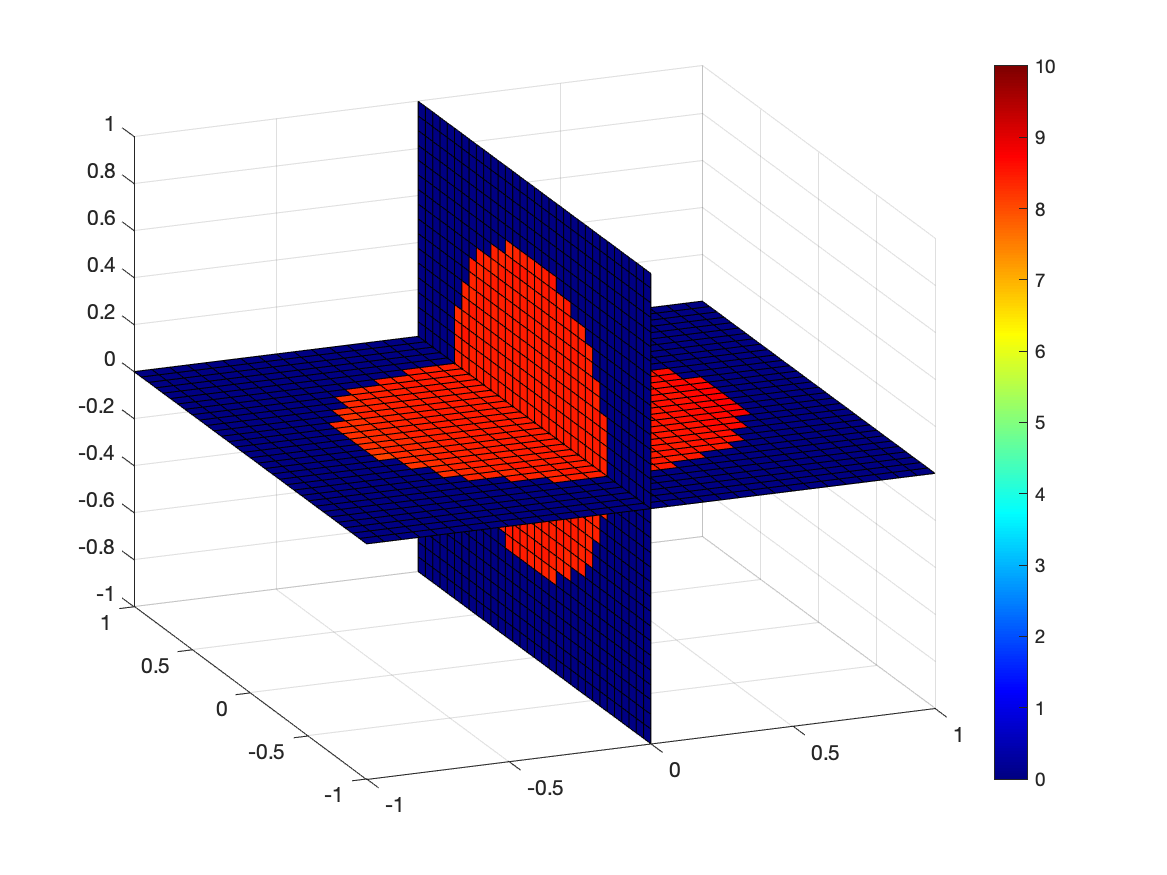}}
 \caption{Plots of approximate pressures obtained by the conventional mini element method (\emph{left}) and the mini IFE method (\emph{right}) for Example 3, using the uniform Cartesian mesh with $M=32$. \label{ex3_fig}} 
\end{figure}

\section*{Acknowledgements}
The authors would like to thank the anonymous referees sincerely for their careful reading and helpful suggestions that improved the quality of this paper.
H. Ji's work was partially supported by the NSFC grant 12371370 and Ministry of Education Key Laboratory of NSLSCS key project  202402.
D. Liang's work was partially supported by Natural Sciences and Engineering Research Council of Canada.
Q. Zhang's work was partially supported by the NSFC grant 12101327 and the Supporting Project of National Natural Science Youth Fund of Nanjing University of Chinese Medicine  (No. XPT12101327).

\appendix 
\section{Technical results}
\subsection{Proof of (\ref{res_ap2}) and (\ref{res_ap3})\label{sec_app_a}}
We construct a new box $R^\prime_T$  centered at $\mathbf{x}_T^c$ with diameter $\rho_T$, where $\mathbf{x}_T^c$ and $\rho_T$  are the center and  diameter of the largest ball inscribed in $T$. We assume that each face of $R^\prime_T$ is parallel to a corresponding face of $R_T$. One can easily obtain  $R^\prime_T\subset T\subset R_T$ and
$
(R_T-\mathbf{x}_T^*)/(2\sqrt{N}h_T)=(R^\prime_T-\mathbf{x}_T^c)/\rho_T.
$
Thus, it holds 
$R^\prime_T=\lambda(R_T-\mathbf{x}_0)+\mathbf{x}_0$ with $\lambda=\rho_T/(2\sqrt{N}h_T)$ and $\mathbf{x}_0=(\mathbf{x}_T^c-\lambda\mathbf{x}_T^*)/(1-\lambda)$. This implies that $R^\prime_T$ is  a homothetic image of $R_T$ (see \cite[Section 2.3]{2016High} for the definition). From the mesh regularity,  we know that the scaling factor $\lambda$ has lower and upper bounds independent of the mesh size and the interface location. Therefore, using Lemma 2.2 in \cite{2016High} we have the norm-equivalence property for polynomials:
\begin{equation}\label{poly_equ_app}
\|w\|_{L^2(R_T)}\leq C\|w\|_{L^2(R^\prime_T)}\quad \forall w\in P_k(R_T).
\end{equation}
For the box $R_T$, there exists a polynomial $v_h\in P_1(R_T)$ such that 
$
|v-v_h|_{H^m(R_T)}\leq Ch_T^{l-m}|v|_{H^l(R_T)},~0\leq m\leq l\leq 2$ (see \cite{brenner2008mathematical}).
Recalling that $\widetilde{I}_{h,T}v$ is a polynomial defined on $\mathbb{R}^N$, we deduce that 
\begin{equation*}
\begin{aligned}
|v-\widetilde{I}_{h,T}v|_{H^m(R_T)}&\leq |v-v_h|_{H^m(R_T)}+|v_h-\widetilde{I}_{h,T}v|_{H^m(R_T)}\\
&\leq  |v-v_h|_{H^m(R_T)}+C|v_h-\widetilde{I}_{h,T}v|_{H^m(R^\prime_T)}\\
&\leq  |v-v_h|_{H^m(R_T)}+C|v_h-\widetilde{I}_{h,T}v|_{H^m(T)}\\
&\leq  C|v-v_h|_{H^m(R_T)}+C|v-\widetilde{I}_{h,T}v|_{H^m(T)}\leq Ch_T^{l-m}|v|_{H^l(\omega_T\cup R_T)},
\end{aligned}
\end{equation*}
which completes the proof of (\ref{res_ap2}). The proof of (\ref{res_ap3}) is analogous.

\subsection{Proof of Lemma~\ref{lem_trac_IFE}\label{sec_app_b}}
We first prove the following useful result:
\begin{equation}\label{pro_app_l2}
|\mathbf{v}_h^\pm|_{H^m(T)} \leq C|\mathbf{v}_h|_{H^m(T)},\quad m=0,1.
\end{equation}
We have 
$\mathbf{v}_h^\pm=\mathbf{v}_b+\mathbf{v}_L^\pm$ with $(\mathbf{v}_L,q_h)\in\widetilde{\boldsymbol{P}_1P_1}(T)$ and $\mathbf{v}_b\in \mbox{span}\{b_T\}^N$. Since $(\mathbf{v}_L,q_h)$ satisfies (\ref{dis_jp0}), in view of the $\mathbf{t}_{i,h}$-$\mathbf{n}_h$ coordinate system, it is not hard to see
\begin{equation*}
\partial_{\boldsymbol{\nu}_{i,h}} (\mathbf{v}_L^+\cdot\boldsymbol{\nu}_{j,h})=\sum_{k}\sum_{l}c_{ijkl}\partial_{\boldsymbol{\nu}_{k,h}} (\mathbf{v}_L^-\cdot\boldsymbol{\nu}_{l,h})~ \mbox{ with }~ |c_{ijkl}|\leq C,
\end{equation*}
where $\boldsymbol{\nu}_{i,h}:=\mathbf{t}_{i,h}$, $i=1,...N-1$ and $\boldsymbol{\nu}_{N,h}:=\mathbf{n}_{h}$. 
This  implies 
$|\nabla \mathbf{v}_L^+ |\leq C|\nabla\mathbf{v}_L^-|$ and similarly, $|\nabla \mathbf{v}_L^- |\leq C|\nabla\mathbf{v}_L^+|$.

By (\ref{dis_jp2}), we have $[\![\mathbf{v}_L^\pm]\!]|_{\Gamma_{h,T}}=\mathbf{0}$, and then 
$
[\![\mathbf{v}_L^\pm]\!]|_{T_h^\pm}=\pm\mbox{dist}(\mathbf{x},\Gamma_{h,T}^{ext}) [\![\nabla\mathbf{v}_L^\pm\mathbf{n}_h]\!]. 
$
Combining these results with the fact $\|\mbox{dist}(\mathbf{x},\Gamma_{h,T}^{ext})\|_{L^\infty(T)}\leq Ch_T$ yields 
$$
\|\mathbf{v}_h^+\|^2_{L^2(T_h^-)}\leq C\|\mathbf{v}_h^-\|^2_{L^2(T_h^-)}+Ch_T^{2}|T_h^-||B|^{-1}\|\nabla \mathbf{v}_L^{s_0}\|^2_{L^2(B)},
$$
where  $s_0$= $+$ or $-$ and $B$ is a ball. The superscript $s_0$ and the ball $B$ are chosen such that $B\subset T_h^{s_0}$ and $|B|\geq Ch_T^N$ (see \cite[Lemma 5.8]{ji3Dnonconforming} for the construction of $B$). Since $\mathbf{v}_b|_{\partial T}=\mathbf{0}$ and $\mathbf{v}_L^{s_0}$ is linear, we have  $\int_{T} \nabla \mathbf{v}_L^{s_0}\cdot  \nabla \mathbf{v}_b=0$, which implies 
\begin{equation}\label{pro_app_ber02}
\|\nabla \mathbf{v}_L^{s_0}\|^2_{L^2(B)}\leq \|\nabla \mathbf{v}_L^{s_0}\|^2_{L^2(T)}\leq \|\nabla (\mathbf{v}_L^{s_0}+\mathbf{v}_b)\|^2_{L^2(T)}.
\end{equation}
Similarly to the proof in Appendix~\ref{sec_app_a}, we now need another ball $B^\prime$ with diameter $2h$ centered at the same point as $B$. Clearly, $T\subset B^\prime$ and $B^\prime=\lambda B$ with the scaling factor $2\leq \lambda\leq C$. Keeping in mind that $\mathbf{v}_L^{s_0}$ and $\mathbf{v}_b$ are viewed as polynomials defined in the whole space $\mathbb{R}^N$, we apply  the norm-equivalence property for polynomials as (\ref{poly_equ_app}) to obtain 
\begin{equation*}
\|\nabla \mathbf{v}_L^{s_0}\|^2_{L^2(B)} \leq \|\nabla (\mathbf{v}_L^{s_0}+\mathbf{v}_b)\|^2_{L^2(B^\prime)}\leq C\|\nabla (\mathbf{v}_L^{s_0}+\mathbf{v}_b)\|^2_{L^2(B)}=C\|\nabla \mathbf{v}_h^{s_0}\|^2_{L^2(B)}.
\end{equation*}
Collecting above results and using the standard inverse inequality on $B$, we have  
\begin{equation*}
\begin{aligned}
\|\mathbf{v}_h^+\|^2_{L^2(T_h)}\leq \|\mathbf{v}_h^+\|^2_{L^2(T_h^+)}+C\|\mathbf{v}_h^-\|^2_{L^2(T_h^-)}+C\| \mathbf{v}_h^{s_0}\|^2_{L^2(B)}\leq C\|\mathbf{v}_h\|^2_{L^2(T)},
\end{aligned}
\end{equation*}
which together with a similar result for $\mathbf{v}_h^-$ completes the proof of (\ref{pro_app_l2}) with $m=0$. The proof of (\ref{pro_app_l2}) with $m=1$ is analogous and so is omitted for brevity.

With the help of  (\ref{pro_app_l2}), the desired inequalities (\ref{trac_IFE_inequality1}) and (\ref{inver_IFE})  can be obtained  by
\begin{equation*}
\begin{aligned}
&\|\nabla \mathbf{v}_h\|_{L^2(\partial T)}\leq C\sum_{s=\pm}\|\nabla \mathbf{v}_h^s\|_{L^2(\partial T)}\leq \sum_{s=\pm}Ch_T^{-1/2}\|\nabla \mathbf{v}_h^s \|_{L^2(T)}\leq Ch_T^{-1/2}\|\nabla \mathbf{v}_h \|_{L^2(T)},\\
&|\mathbf{v}_h|_{H^1(T)} \leq C\sum_{s=\pm}|\mathbf{v}_h^s|_{H^1(T)}\leq \sum_{s=\pm}Ch_T^{-1}\| \mathbf{v}_h^s \|_{L^2(T)}\leq Ch_T^{-1}\| \mathbf{v}_h \|_{L^2(T)},
\end{aligned}
\end{equation*}
where we used the standard trace and inverse inequality on $T$ since $\mathbf{v}_h^\pm$ are polynomials.
As for (\ref{IFE_inequ_v}), we first use a similar argument to get 
\begin{equation}\label{pro_ap_ber01}
|I_{h,T}\mathbf{v}_h|_{H^1(T)}  \leq Ch_T^{N/2-1}\sum_{s=\pm}\|\mathbf{v}_h^s\|_{L^\infty(T)}\leq Ch_T^{-1}\sum_{s=\pm}\|\mathbf{v}_h^s\|_{L^2(T)}\leq Ch_T^{-1}\|\mathbf{v}_h\|_{L^2(T)}.
\end{equation}
Then the desired inequality (\ref{IFE_inequ_v}) can be obtained by 
\begin{equation*}
|I_{h,T}\mathbf{v}_h|_{H^1(T)} =|I_{h,T}(\mathbf{v}_h+\mathbf{c})|_{H^1(T)}\leq Ch_T^{-1}\|\mathbf{v}_h+\mathbf{c}\|_{L^2(T)}\leq C|\mathbf{v}_h|_{H^1(T)},
\end{equation*}
where the constant $\mathbf{c}$ is chosen as $\mathbf{c}=|T|^{-1}\int_T\mathbf{v}_h$.
By Lemma \ref{lem_IFEbasis}, it is easy to verify 
\begin{equation}\label{pro_jp}
|\mathbf{v}_h-I_{h}^B\mathbf{v}_h|\leq Ch_T|\nabla I_{h,T}\mathbf{v}_h|,~ |\nabla (\mathbf{v}_h-I_{h}^B\mathbf{v}_h)|\leq C|I_{h,T}\nabla\mathbf{v}_h|,~|[\![q_h^\pm]\!]|\leq C|I_{h,T}\nabla\mathbf{v}_h|,
\end{equation}
which together with (\ref{IFE_inequ_v})  yields (\ref{IFE_FE_er}) and (\ref{trac_IFE_inequality_qj}).
We also have $|q_h^\pm | \leq |q_h^\mp|+C|\nabla I_{h,T}\mathbf{v}_h|$. Thus,
\begin{equation*}
\|q_h^\pm\|_{L^2(T_h^\mp)}\leq \|q_h^\mp\|_{L^2(T_h^\mp)}+C|I_{h,T}\mathbf{v}_h|_{H^1(T)}\leq \|q_h\|_{L^2(T)}+C|\mathbf{v}_h|_{H^1(T)}.
\end{equation*}
Now (\ref{trac_IFE_inequality2}) is obtain by the above inequality  and
\begin{equation*}
\|q_h\|_{L^2(\partial T)}\leq C\sum_{s=\pm}\|q_h^s\|_{L^2(\partial T)}\leq Ch_T^{-1/2}\sum_{s=\pm}\|q_h^s\|_{L^2(T)}.
\end{equation*}
Finally, we consider (\ref{ineq_jup_v1}). 
Observing that $I_hq_h$ and $I_h^B\mathbf{v}_h$ are continuous, we have 
\begin{equation*}
\|[q_h]_F\|_{L^2(F)}= \|[q_h-I_hq_h]_F\|_{L^2(F)},\quad \|[\mathbf{v}_h]_{F}\|_{L^2(F)}=\|[\mathbf{v}_h-I_h^B\mathbf{v}]_{F}\|_{L^2(F)}.
\end{equation*}
By the first inequality of (\ref{pro_jp}) and (\ref{IFE_inequ_v}) we get
\begin{equation*}
h_F^{-1}\|\mathbf{v}_h-I^B_{h,T}\mathbf{v}_h\|^2_{L^2(F)} \leq Ch_F^{-1}|F|h_T^2|\nabla I_{h,T}\mathbf{v}_h|^2\leq C|I_{h,T}\mathbf{v}_h|^2_{H^1(T)}\leq C|\mathbf{v}_h|^2_{H^1(T)}.
\end{equation*}
By Lemma \ref{lem_IFEbasis}, we have  $|q_h-I_{h,T}q_h| \leq C|\nabla I_{h,T}\mathbf{v}_h|$, and then
\begin{equation*}
h_F\|q_h-I_{h,T}q_h\|^2_{L^2(F)} \leq Ch_F|F||\nabla I_{h,T}\mathbf{v}_h|^2\leq C|I_{h,T}\mathbf{v}_h|^2_{H^1(T)}\leq C|\mathbf{v}_h|^2_{H^1(T)}.
\end{equation*}
The desired inequality  (\ref{ineq_jup_v1})  follows from the above results and the triangle inequality.

\subsection{Proof of Lemma~\ref{lem_gah_ineq}\label{sec_app_c}}
See (A.4)-(A.6) in \cite{burman2018Acut} for the proof of (\ref{gah_ieq1}). In the following, we prove (\ref{gah_ieq2}) through a similar argument. For each $\mathbf{x}\in\Gamma_{h,T}$, $T\in\mathcal{T}_h^\Gamma$, using the representation
\begin{equation*}
v(\mathbf{x})=v(\mathbf{p}_h(\mathbf{x}))+\int_0^1\nabla v(s\mathbf{x}+(1-s)\mathbf{p}_h(\mathbf{x}))\cdot(\mathbf{x}-\mathbf{p}_h(\mathbf{x}))ds,
\end{equation*}
 the Cauchy-Schwarz inequality and (\ref{ph_esti}), we obtain 
\begin{equation*}
|v(\mathbf{x})-v(\mathbf{p}_h(\mathbf{x}))|^2\leq Ch_T^2\|\nabla v\|^2_{L^2(I_{\mathbf{x},\mathbf{p}_h(\mathbf{x})})},
\end{equation*}
where $I_{\mathbf{x},\mathbf{p}_h(\mathbf{x})}$ is the line segment between $\mathbf{x}$ and $\mathbf{p}_h(\mathbf{x})$. Integrating over $\Gamma_{h,T}$ yields
\begin{equation*}
\|v-v\circ\mathbf{p}_h\|_{L^2(\Gamma_{h,T})}^2\leq Ch_T^2\|\nabla v\|^2_{L^2(R_T)}.
\end{equation*}
The desired result (\ref{gah_ieq2}) follows from the above inequality and the finite overlapping of $R_T$.

\section{IFE basis functions in the $\mathbf{t}_{i,h}$-$\mathbf{n}_h$ coordinate system \label{sec_app_basis}}
Using the $\mathbf{t}_{i,h}$-$\mathbf{n}_h$ coordinate system, we can greatly simplify the IFE basis function formulas derived in subsection~\ref{sec_Unisolvence}. Here, we only consider the three-dimensional case.
Obviously, adopting the $\mathbf{t}_{i,h}$-$\mathbf{n}_h$ coordinate system is equivalent to considering the case where $\Gamma_{h,T}$ is parallel to the plane $x_1$-$x_2$. Thus, we have 
\begin{equation*}
\mathbf{t}_{1,h}=(1,0,0), ~\mathbf{t}_{2,h}=(0,1,0), ~\mathbf{n}_{h}=(0,0,1). 
\end{equation*}
Let $\lambda_{i,T}\in P_1(T)$, $i=1,...,4$, be the standard basis functions such that $ \lambda_{i,T}(\mathbf{a}_{j,T})=\delta_{ij}$ for all $j=1,...,4$. Therefore, the standard basis functions $\boldsymbol{\phi}_{k,T}$, $k=1,...,12$, for velocity can be written as
\begin{equation*}
\boldsymbol{\phi}_{i,T}=(\lambda_{i,T},0,0), ~\boldsymbol{\phi}_{i+4,T}=(0,\lambda_{i,T},0), ~ \boldsymbol{\phi}_{i+8,T}=(0,0,\lambda_{i,T}),~i=1,...,4.
\end{equation*}
After substituting the aforementioned expressions into the explicit formulas derived in subsection~\ref{sec_Unisolvence} and using the definition of $\sigma(\mu,\mathbf{u},p)$, we obtain the following IFE basis functions for velocity:
\begin{equation*}
\begin{aligned}
\boldsymbol{\phi}_{i,T}^{\Gamma}&=(\lambda_{i,T}+k_T\frac{\partial\lambda_{i,T}}{\partial x_3}, ~0, ~0),~~i=1,...,4,\\
\boldsymbol{\phi}_{i+4,T}^{\Gamma}&=(0,~ \lambda_{i,T}+k_T\frac{\partial\lambda_{i,T}}{\partial x_3}, ~0),~~i=1,...,4,\\
\boldsymbol{\phi}_{i+8,T}^{\Gamma}&=(k_T\frac{\partial\lambda_{i,T}}{\partial x_1}, ~k_T\frac{\partial\lambda_{i,T}}{\partial x_2}, ~\lambda_{i,T}+2k_T\frac{\partial\lambda_{i,T}}{\partial x_3}),~~i=1,...,4,\\
\boldsymbol{\phi}_{i+12,T}^{\Gamma}&=(0,~0,~0),~~i=1,...,4,
\end{aligned}
\end{equation*}
where 
$$k_T= (\mu^-/\mu^+-1)(w_T-I_{h,T}w_T)\left(1+(\mu^-/\mu^+-1)\frac{\partial (I_{h,T}w_T)}{\partial x_3}\right)^{-1}$$
is  a constant on each interface element. For the pressure, we have 
\begin{equation*}
\begin{aligned}
\varphi_{i,T}^\Gamma&=0,~~i=1,...,4,\\
\varphi_{i+4,T}^\Gamma&=0,~~i=1,...,4,\\
\varphi_{i+8,T}^\Gamma&=2(\mu^--\mu^+)\frac{\partial \lambda_{i,T}}{\partial x_3}(z_T-I_{h,T}z_T), ~~i=1,...,4,\\
\varphi_{i+12,T}^\Gamma&=\lambda_{i,T},~~i=1,...,4.
\end{aligned}
\end{equation*}
\bibliographystyle{plain}
\bibliography{refer}

\end{document}